\documentclass[a4paper, 11pt]{article} 

\usepackage{times}
\usepackage{graphicx}
\usepackage{amssymb,amsmath}
\usepackage{upref}
\usepackage[sf,SF]{subfigure} 
\usepackage{caption}
\usepackage{amsthm,euscript}
\usepackage[a4paper, portrait, margin=1in]{geometry}
\usepackage{xcolor}
\usepackage[bookmarks=true,hidelinks]{hyperref}
\usepackage{units}
\usepackage{bbm}
\usepackage{enumitem}
\usepackage{caption}
\usepackage[title]{appendix}

\newcommand{\R}{\mathbb{R}}

\newcommand{\Z}{\mathbb{Z}}
\renewcommand{\leq}{\leqslant}
\renewcommand{\geq}{\geqslant}
\renewcommand{\phi}{\varphi}
\newcommand{\hf}{{\unitfrac{1}{2}}}

\renewcommand{\epsilon}{\varepsilon}
\newcommand{\bx}{\mathbf{x}}

\newcommand{\diff}{\textrm{d}}
\newcommand{\rhostat}{\rho_\infty}
\newcommand{\brho}{{\bar{\rho}}}

\newcommand{\norme}[1]{\left\lVert#1\right\rVert}

\DeclareMathOperator{\sgn}{sgn}

\DeclareMathOperator\erf{erf}

\newtheorem{theorem}{Theorem}
\newtheorem{proposition}[theorem]{Proposition}

\newtheorem{remark}[theorem]{Remark}
\newtheorem{definition}[theorem]{Definition}
\newtheorem{lemma}[theorem]{Lemma}
\newtheorem{corollary}[theorem]{Corollary}

\numberwithin{equation}{section}
\numberwithin{theorem}{section}

\makeatletter
\def\@cite#1#2{\textup{[{#1\if@tempswa , #2\fi}]}}
\makeatother

\allowdisplaybreaks

\title{Well-posedness and stability of a stochastic neural field in the form of a partial differential equation}

\author{José Antonio Carrillo\thanks{Mathematical Institute, University of Oxford, Oxford OX2 6GG, UK (carrillo@maths.ox.ac.uk)} \and Pierre Roux\thanks{Mathematical Institute, University of Oxford, Oxford OX2 6GG, UK (pierre.roux@maths.ox.ac.uk)} \and Susanne Solem\thanks{Department of Mathematics, Norwegian University of Life Sciences, NO-1433 \AA s, Norway (susanne.solem@nmbu.no)}}

\begin{document}

\maketitle

\begin{abstract}
    A system of partial differential equations representing stochastic neural fields was recently proposed with the aim of modelling the activity of noisy grid cells when a mammal travels through physical space. The system was rigorously derived from a stochastic particle system and its noise-driven pattern-forming bifurcations have been characterised. However, due to its nonlinear and non-local nature, standard well-posedness theory for smooth time-dependent solutions of parabolic equations does not apply. In this article, we transform the problem through a suitable change of variables into a nonlinear Stefan-like free boundary problem and use its representation formulae to construct local-in-time smooth solutions under mild hypotheses. Then, we prove that fast-decaying initial conditions and globally Lipschitz modulation functions imply that solutions are global-in-time. Last, we improve previous linear stability results by showing nonlinear asymptotic stability of stationary solutions under suitable assumptions.
\end{abstract}

\vspace{0.5cm}

\section{Introduction}

Since their introduction \cite{WC1,WC2,A}, the study of neural field models has played a crucial role in the understanding of many cortical phenomena; see the reviews \cite{B,C,TWPC} for descriptions of the extensive literature. It has been argued that noise plays a pivotal role in the formation of complex patterned firing activity \cite{RD}. In the last years, numerous works have focused on stochastic neural field models and how noise strength affects the apparition and the evolution of patterns like travelling waves, wandering bumps, multistable equilibria and bifurcation branches (see for instance \cite{KE,MB,BK,B2,KS,CEK,CRS,KT} and the review \cite[Sec. 6]{B}).

The aim of this manuscript is to provide a mathematical framework for the study of solutions of a recent neural field model in the form of a nonlinear Fokker--Planck equation:
\begin{equation}\label{eq:PDE}
	\tau_c\dfrac{\partial \rho}{\partial t} = -\dfrac{\partial }{\partial s}\left[  \left(  \Phi\left( \int_{\mathbb{T}^d} W(x-y)\int_0^{+\infty}s\rho(y,s,t)d s d y + B(t)\right) -s  \right)\rho(x,s,t)   \right] + \sigma \dfrac{\partial^2 \rho}{\partial s^2},
\end{equation}
where $\rho(x,s,t)$ is at time $t$ the probability density of neurons at location $x\in \mathbb T^d := \R^d /( L \Z)^d$, $L\in\R_+^*$, with firing activity level $s\in[0, +\infty)$. The interaction between neurons at different space locations happens through a connectivity kernel $W\in L^1(\mathbb T^d)$ and a continuous modulation function $\Phi$. The external input to the system is denoted $B(t)$, and formal mass conservation and positivity of the activity variable are enforced with a no-flux boundary condition: for all $(x,t)\in\mathbb T^d\times\R_+$,
\begin{equation}\label{eq:BC}
	\left( \Phi\left( \int_{\mathbb T^d} W(x-y)\int_0^{+\infty}s\rho(y,s,t)d s d y + B(t)\right)\rho(x,s,t) - \sigma \dfrac{\partial \rho}{\partial s}(x,s,t) \right)\bigg|_{s=0} = 0.
\end{equation}
Last, we assume that neurons are homogeneously distributed in space, which is imposed by the condition
\begin{equation}\label{eq:homogeneousBrain}  
\forall t\in\R_+, \ \forall x\in\mathbb T^d, \quad \int_0^{+\infty} \rho(x,s,t) ds = \dfrac{1}{L^d}. 
\end{equation}
 
This type of nonlinear Fokker--Planck system was introduced in \cite{CHS} as a way to study the impact of noise upon networks of noisy grid cells. Grid cells are a specific type of neuron that plays a role in the navigational system of mammals; it has been shown that they exhibit a typical hexagonal firing pattern \cite{HFMMM,RYMM,MMB} in physical space. With a generalisation of \eqref{eq:PDE} (that we expose in more details in Section \ref{sec:4}), the authors of \cite{CHS} were able to numerically explain the apparition of neural firing patterns as noise-driven bifurcations from a space-homogeneous stationary state and provided a study of the stability of the linearised equations. The bifurcation analysis was then made rigorous in \cite{CRS} for both system \eqref{eq:PDE} and its generalisation. Furthermore, the systems were derived as mean-field limits of particle systems in \cite{CCS}, and the fluctuations of the empirical measure associated to the particle systems around their mean-field limits were proved to converge to the solution of a Langevin SPDE in \cite{C2}.

Systems of the form \eqref{eq:PDE} are challenging to analyse because of the non-local nonlinear drift and the nonlinear Robin-like boundary condition. Although the mean-field limit in \cite{CCS} provides the existence of weak solutions, there is to our knowledge no existing mathematical framework for the time evolution of classical solutions. In this work, we build on methods developed for the study of other nonlinear Fokker--Planck equations in neuroscience \cite{CGGS,CCP,CPSS, CRSS}: first, we use classical change of variables to reformulate the problem as a nonlinear Stefan-like free boundary problem in order to construct local and global in time solutions, and obtain representation formulae and universal estimates as a by-product (Section \ref{sec:2}); second, we use a generalised relative entropy method and insights from the linear stability proof of \cite{CHS} to prove results on the asymptotic behaviour of solutions (Section \ref{sec:3}).

Let us be more precise about the notion of solution we deal with, and the main results we show. Our goal is to show that continuous initial conditions with appropriate decay lead to classical solutions which are continuous in space, $C^2$ in activity and $C^1$ in time, and to propose representation formulae and criteria for global-in-time existence. Throughout this work, we use the notation $\R_+=[0,\infty)$ and $\R_+^*=(0,\infty)$. The space of continuous functions (resp. $C^k$ functions) from a space $X$ to a space $Y$ is denoted by $C^0(X,Y)$ (resp. $C^k(X,Y)$). When $Y = \R$, we simply denote $C^k(X)$; we use the same convention for Lebesgue spaces $L^p(X,Y)$ and $L^p(X)$. When regularity is uneven between different variables, we write for example $C^{K,l,m}(X\times Y\times Z)$ to tell $C^k$ on $X$, $C^l$ on $Y$, $C^m$ on $Z$ and $C^{\min(k,l,m)}$ on $X\times Y\times Z$.

Since \eqref{eq:PDE}-\eqref{eq:BC} involves the average in $s$ and due to specific technicalities, we require decay on the initial conditions which we then can propagate on the solutions.

\begin{definition}(Fast decay at $s=+\infty$)
    Let $a\in\R$. We say that a function $\rho : \mathbb T^d \times [a,+\infty) \to \R$, is fast-decaying if for all polynomial functions $P$ of one variable, for all $x\in\mathbb T^d$,
    \[  \lim_{s\to +\infty} \rho(x,s) P(s) = 0. \]
\end{definition}

\begin{definition}[Compatible initial condition]\label{def:compatible}
    We say that the function $\rho^0\in C^{0,1}(\mathbb T^d \times \R_+)$ is a compatible initial condition for \eqref{eq:PDE} if $\rho^0$ is a non-negative probability density, if $\rho^0$ and $\partial_s\rho^0$ are fast-decaying,
    and if
    \begin{equation*} \forall x\in\mathbb T^d, \quad \int_0^{+\infty} \rho^0(x,s)ds = \dfrac{1}{L^d}.  \end{equation*}
\end{definition}

\begin{remark}[Non-optimality of the existence results]
    We want to emphasise that the regularity and the decay of the solutions we construct are by no means optimal. Optimal regularity for this problem is beyond the scope of this work. With higher regularity on $W$, $\Phi$ and $B$, the parabolic nature of the equation and the smoothing effect of the convolution should lead to more regularity even with non-smooth initial conditions. 
    
    Concerning the decay assumptions, biologically relevant solutions do not have fat tails near infinite activity, and higher order moments can be useful to obtain estimates.  Therefore, although it is not necessary for construction of solutions, we choose to work with fast-decaying initial conditions.
    
    The condition \eqref{eq:homogeneousBrain} is not necessary for existence of solutions, but it simplifies the proof and it will be used in later estimates. From a neuroscience perspective, it is tantamount to assume a homogeneous repartition of the neurons in the brain area we model.
\end{remark}

\begin{definition}[Classical solution]\label{def:classical}
    We say that a non-negative function $\rho\in C^{0}(\mathbb T^d\times\R_+\times [0,T_{\mathrm{max}}))$, $T_{\mathrm{max}}\in [0,+\infty]$ is a classical solution of \eqref{eq:PDE}-\eqref{eq:BC} if    
    \begin{itemize}
        \item it has regularity 
        \begin{align*}
        \rho\in \, &C^{0,2,1}(\mathbb T^d\times\R_+^*\times (0,T_{\mathrm{max}}))\cap C^{0,1,1}(\mathbb T^d\times\R_+\times (0,T_{\mathrm{max}})),
    \end{align*}
    and satisfies \eqref{eq:PDE}-\eqref{eq:BC} in the classical sense;
        \item for all $t\in[0,T_{\mathrm{max}})$, $\rho(\cdot,\cdot,t)$, $\partial_s \rho(\cdot,\cdot,t)$ are fast-decaying, and $\int_{\mathbb T^d}\int_{0}^{+\infty} \rho(x,s,t) dsdx = 1 $;
        \item  for all $(x,t)\in\mathbb T^d\times [0,T_{\mathrm{max}})$, 
    \[ 
    \int_{0}^{+\infty} \rho(x,s,t)ds = \dfrac{1}{L^d}.
    \]
    \end{itemize}
    
\end{definition}

Section \ref{sec:2} is devoted to proving the different claims of the following theorem.

\begin{theorem}[Existence and uniform bounds]\label{thm:main}
    Assume $\Phi$ is locally Lipschitz, $B\in C^0(\R_+)$ and $W\in L^1(\mathbb T^d)$. For any compatible initial condition $\rho^0$, \eqref{eq:PDE}-\eqref{eq:BC} admits a unique maximal solution in the sense of Definition \ref{def:classical}.
    Moreover, if $\Phi$ is globally Lipschitz,
    \begin{itemize}
        \item the solution is global-in-time;
        \item if $\Phi$ is non-negative, then
    \begin{equation*}
    \forall t\in\R_+,\ \forall x\in\mathbb T^d, \quad \rho(x,0,t) \leqslant \dfrac{1}{L^d}\sqrt{ \dfrac{2}{ \pi \sigma } } \dfrac1{\sqrt{1-e^{-\frac{2t}{\tau_c}}}};
    \end{equation*}
    \item if $\Phi$ is non-positive, then
        \[ \forall t\in \R_+,\ \forall x\in \mathbb T^d, \quad  \bar\rho(x,t) \leqslant  \dfrac{1}{L^d}(1+\sigma - \sigma e^{-\frac{2t}{\tau_c}}) +  e^{-\frac{2t}{\tau_c}} \sup_{x\in\mathbb T^d}\int_0^{+\infty} s^2\rho^0(x,s)ds . \]
    \end{itemize}
\end{theorem}

Section \ref{sec:3} shows the asymptotic stability of stationary states as in the next result.

\begin{theorem}[Asymptotic stability]\label{thm:main2}
Assume $B(t)=B \in \R$. Let $\rhostat$ be a stationary state of \eqref{eq:PDE}-\eqref{eq:BC} for a given $\sigma$, and let $W \in L^2(\mathbb T^d)$. Assume either
\begin{enumerate}
    \item that $\Phi$ is Lipschitz and that $\sigma$ satisfies
    \begin{align*}
        \|\Phi' \|_\infty \frac{\|W\|_{L^2({\mathbb T^d})}}{L^\frac d2} \sup_{x \in {\mathbb T^d}} \left(\int_0^{+\infty}(s-\brho_\infty(x))^2\rhostat(x) ds \right)^\hf < \frac{\sigma}{2} \tilde\gamma (\rhostat )^\hf,
    \end{align*}
    where $\tilde\gamma (\rhostat ) = \inf_{x \in \mathbb T^d} \gamma(\rhostat (x))$, and $\gamma(\rhostat(x))$ is the Poincar\'e constant in $s$ with respect to $\rhostat(x)$, see \eqref{eq:poincare}, or
    \item that $\Phi \in C^2(\mathbb R)$, $W$ is componentwise symmetric, and that for a suitably small $\alpha >0$, $\sigma$ satisfies
    \begin{align*}
    \int_{\mathbb T^d} (1-\alpha)g^2(x) - \frac{\Phi^\delta_g(x)}{\sigma} g(x) \, dx > 0, \quad g \in L^2(\mathbb T^d),
    \end{align*}
    where $$\Phi^\delta_g(x) =\Phi\big( M_\infty W\ast g (x) +W_0 \brho_\infty +B\big)-\Phi_0,$$
    and $\rho_\infty\equiv\rho_\infty(s)$ is homogeneous in space.
\end{enumerate}
Then, the $L^1$ norm of the relative entropy 
\begin{align*}
 \int_{\mathbb T^d} \int_{0}^{+\infty} \left( \dfrac{ \rho(x,s,t) - \rho_\infty(x,s) }{ \rho_\infty(x,s) } \right)^2 \rho_\infty(x,s) ds dx,
\end{align*} 
decays exponentially fast whenever the compatible initial data $\rho^0$ is close enough to the stationary state $\rhostat$ in relative entropy.
\end{theorem}

In Section \ref{sec:4}, we extend the above results to the four population system from \cite{CHS} modelling networks of noisy grid cells.

\section{Well-posedness theory}\label{sec:2}

Let us first remark that we can always assume without loss of generality that $\tau_c=1$ and $\sigma = 1$. Notice that it is possible to set these values by considering the change of variables:
\[
\tilde \rho(x,s,t) = \sqrt{\sigma} \rho\left( x , \sqrt{\sigma} s, \tau_c\, t\right) \quad \mbox{or equivalently} \quad \rho(x,s,t) = \dfrac{1}{\sqrt\sigma} \tilde \rho\left(x, \dfrac{s}{\sqrt\sigma},\frac{t}{\tau_c} \right) .  
\]
The scaled density $\tilde\rho$ satisfies \eqref{eq:PDE} with $\tau_c=1$, $\sigma = 1$, $\tilde \Phi = \tfrac{1}{\sqrt\sigma}\Phi$ and $\tilde W = \tfrac{1}{\sqrt{\sigma}}W$.
For the sake of clarity, we will always drop the tilde and write $\rho(x,s,t), \Phi, W$ even while rescaling to normalise $\tau_c$ and $\sigma$. Some of the results below are obtained by implicitly undoing this change of variables after obtaining preliminary results in the normalised case.

\subsection{Equivalence to a free boundary problem with Robin boundary condition}

We first make the change of variables using selfsimilar variables for the linear Fokker-Planck equation followed by a change through characteristics as in \cite{CGGS,CRSS}. More precisely, we define the new variable as $y=e^{t}s$ and $\tau = \tfrac{1}{2}(e^{2t}-1)$,
or equivalently
\[
s=\dfrac{y}{\sqrt{2\tau +1}} = \alpha(\tau)y  \quad \mathrm{and} \quad t = \dfrac{1}{2}\log(2\tau + 1 ) = -\log(\alpha(\tau)),  
\]
where we denote
$\alpha(\tau)=(2\tau+1)^{-\tfrac12} = e^{-t}.$
We define the new function
\[ 
q(x,y,\tau) = \alpha(\tau) \rho\Big(x,y\alpha(\tau), -\log(\alpha(\tau))\Big) \quad
\mbox{or equivalently} \quad
 \rho(x,s,t) = e^t q(x,e^t s , \tfrac12 (e^{2t}-1)). 
 \]
An uninspiring exercise in utilising the classical chain rule leads to
\begin{align*}  \dfrac{\partial q}{\partial \tau}(x,y,\tau) =&\, \alpha'(\tau)\rho\big(x,y\alpha(\tau), -\log(\alpha(\tau)) \big)
+ \ y\alpha'(\tau)\alpha(\tau)\dfrac{\partial \rho}{\partial s}\big(x,y\alpha(\tau), -\log(\alpha(\tau)) \big)\\
&- \ \alpha'(\tau)\dfrac{\partial \rho}{\partial t}\big(x,y\alpha(\tau), -\log(\alpha(\tau)), 
\\ 
\dfrac{\partial q}{\partial y}(x,y,\tau) = &\,\alpha(\tau)^2 \dfrac{\partial \rho}{\partial s}\big(x,y\alpha(\tau), -\log(\alpha(\tau)) \big),  \\
\dfrac{\partial^2 q}{\partial y^2}(x,y,\tau) =&\, \alpha(\tau)^3 \dfrac{\partial^2 \rho}{\partial s^2}\big(x,y\alpha(\tau), -\log(\alpha(\tau)) \big)= -\alpha'(\tau) \dfrac{\partial^2 \rho}{\partial s^2}\big(x,y\alpha(\tau), -\log(\alpha(\tau)) \big),
\end{align*}
where $-\alpha'(\tau) =(2\tau+1)^{-\tfrac32} =  \alpha(\tau)^3$ is used.
Using \eqref{eq:PDE} and defining $\beta(\tau) = B(-\log(\alpha(\tau)))$, we obtain
\begin{align*}  \dfrac{\partial \rho}{\partial t}\big(x,y\alpha(\tau), -\log(\alpha(\tau))\big)  = &\,  y\alpha(\tau)\dfrac{\partial \rho}{\partial s}\big(x,y\alpha(\tau), -\log(\alpha(\tau)) \big) + \rho\big(x,y\alpha(\tau), -\log(\alpha(\tau)) \big)\\ &  - \Psi(x,\tau) \alpha(\tau)^{-1} \dfrac{\partial \rho}{\partial s}\big(x,y\alpha(\tau), -\log(\alpha(\tau)) \big)   \\ &+  \dfrac{\partial^2 \rho}{\partial s^2}\big(x,y\alpha(\tau), -\log(\alpha(\tau)) \big)  . 
\end{align*}
with $\Psi(x,\tau)=\Phi\big( W \ast \bar \rho(x,-\log(\alpha(\tau)) + \beta(\tau)  \big)\alpha(\tau)$. Denoting
\[  
\bar q (x,\tau) = \int_{0}^{+\infty} y q(x,y,\tau) d y ,
\]
we have 
\[
\bar \rho(x,t) = \int_0^{+\infty} s \rho(x,s,t) ds = \int_{0}^{+\infty} \alpha(\tau) y \dfrac{1}{\alpha(\tau)} q(x,y,\tau) \alpha(\tau)d y   =  \alpha(\tau) \bar q(x,\tau).
\] 
Collecting the previous computations, we conclude
\begin{equation}\label{eq:for_q}
    \dfrac{\partial q}{\partial \tau}(y,\tau) = \dfrac{\partial^2 q}{\partial y^2}(y,\tau) - \Psi(x,\tau) \dfrac{\partial q}{\partial y}(y,\tau).
\end{equation}
determining $q$ with boundary condition
\begin{equation}
     \Psi(x,\tau)q(x,0,\tau) - \dfrac{\partial q}{\partial y}(x,0,t)  = 0.
\end{equation}
In order to get rid of the drift term in \eqref{eq:for_q}, we make the second change of variable
\[
z = y - \int_{0}^{\tau} \Psi(x,\eta) d \eta.
\]
Define $u$ by $u(x,z,\tau)=q(x,y,\tau)$ and denote
\[ 
\gamma(x,\tau) = - \int_{0}^{\tau} \Psi(x,\eta) d \eta, \qquad \bar u(x,\tau) = \int_{\gamma(x,\tau)}^{+\infty} zu(x,z,\tau)dz. 
\]
Since 
$\bar q (x,\tau) = \bar u(x,\tau) - \gamma(x,\tau),$
the function $u$ solves, for all $x\in \mathbb T^d$, the Stefan-like free boundary problem
\begin{equation}\label{eq:Stefan}
\left\{\begin{array}{rcll}
\displaystyle \dfrac{\partial u}{\partial \tau}(x,z,\tau) &=& \dfrac{\partial^2 u}{\partial z^2}(x,z,\tau), \qquad\qquad\qquad\  z\in(\gamma(x,\tau), +\infty),&  \tau\in\R_+, \\[3mm]
\displaystyle \gamma(x,\tau) &=& - \displaystyle\int_{0}^{\tau} \Psi(x,\eta) d \eta, & \tau\in\R_+,\\
\dfrac{\partial u}{\partial y}(x,\gamma(x,\tau),\tau) &=& \Psi(x,\tau) u(x,\gamma(x,\tau),\tau), &\tau\in\R_+,\\
\Psi(x,\tau) &=& \Phi\big( \alpha(\tau) W \ast [ \bar u (x,\tau) - \gamma(x,\tau)] + \beta(\tau)\big)\alpha(\tau),\ &\tau\in\R_+,  \\[2mm]
\displaystyle u(x,z,0)&=&u^0(x,z),& z\in\,(0,+\infty).
\end{array}\right.
\end{equation}
This representation is similar to the ones obtained for the NNLIF model \cite{CGGS,CRSS}. Because it would be cumbersome to write in full details, we implicitely consider as a definition of solution for system \eqref{eq:Stefan} the natural extension of Definition \ref{def:classical} after applications of the changes of variables.

\subsection{Duhamel formulae}

In order to obtain a Duhamel formula for $u$, we use the heat kernel
\[  G(z,\tau,\xi,\eta) = \frac{1}{\sqrt{4\pi (\tau-\eta) }}e^{-\frac{(z-\xi)^2}{4(\tau-\eta)}}. \]
This kernel satisfies the Green identity
\[ \dfrac{\partial}{\partial \xi} \left( G\dfrac{\partial u}{\partial \xi} - u\dfrac{\partial G}{\partial \xi} \right) - \dfrac{\partial }{\partial \eta}\big( Gu \big)  = 0.\]
For each $x\in \mathbb T^d$, we integrate the identity in the domain $(\gamma(x,t),+\infty)\times (0,\tau)$; we thereby obtain
\[ \underbrace{\int_{0}^\tau\int_{\gamma(x,\eta)}^{+\infty}  \dfrac{\partial}{\partial \xi} \left( G\dfrac{\partial u}{\partial \xi}  \right) d\xi d\eta }_{I} - \underbrace{\int_{0}^\tau \int_{\gamma(x,\eta)}^{+\infty}  \dfrac{\partial}{\partial \xi} \left( u\dfrac{\partial G}{\partial \xi} \right) d\xi d\eta}_{II} - \underbrace{\int_{0}^\tau \int_{\gamma(x,\eta)}^{+\infty}  \dfrac{\partial }{\partial \eta}\big( Gu \big) d\xi d\eta}_{III}   = 0.  \]
We first compute, using the boundary condition and fast decay of the solution
\begin{align*}
I &=\displaystyle \int_0^\tau \left[ G(z,\tau,\xi,\eta)\dfrac{\partial u}{\partial \xi}(x,\xi,\eta) \right]^{\xi=+\infty}_{\xi=\gamma(x,\eta)} d\eta
= \displaystyle -  \int_0^\tau G(z,\tau,\gamma(x,\eta),\eta)\dfrac{\partial u}{\partial \xi}(x,\gamma(x,\eta),\eta) d\eta\\
&= \displaystyle -  \int_0^\tau G(z,\tau,\gamma(x,\eta),\eta) \Psi(x,\eta) u(x,\gamma(x,\eta),\eta) d\eta.
\end{align*}
Similarly, we get
\begin{align*}
II = \displaystyle  \int_0^\tau \left[ u(x,\xi,\eta)\dfrac{\partial G}{\partial \xi}(z,\tau,\xi,\eta) \right]^{\xi=+\infty}_{\xi=\gamma(x,\eta)} d\eta 
= \displaystyle - \int_0^\tau \dfrac{\partial G}{\partial \xi}(z,\tau,\gamma(x,\eta),\eta) u(x,\gamma(x,\eta),\eta) d\eta.
\end{align*}
Let us assume for now that $\gamma$ is smooth enough to be differentiated in time, which will be proved later; then
\begin{align*}
    \int_{\gamma(x,\eta)}^{+\infty}  \dfrac{\partial }{\partial \eta}\big( Gu \big) & d\xi d\eta\, \\=&\, \dfrac{\partial }{\partial \eta} \int_{\gamma(x,\eta)}^{+\infty}  G(z,\tau,\xi,\eta) u(x,\xi,\eta)  d\xi d\eta \, +\, \dfrac{\partial \gamma}{\partial \eta}(x,\eta) G(z,\tau,\gamma(x,\eta),\eta)u(x,\gamma(x,\eta),\eta)\\  
    =&\, \dfrac{\partial }{\partial \eta} \int_{\gamma(x,\eta)}^{+\infty}  G(z,\tau,\xi,\eta) u(x,\xi,\eta)  d\xi d\eta \, -\, \Psi(x,\eta) G(z,\tau,\gamma(x,\eta),\eta)u(x,\gamma(x,\eta),\eta).
\end{align*}
Last, using that in the limit $\eta\to\tau$ the heat kernel satisfies $G(z,\tau,\xi,\eta) = \delta_{z = \xi}$, the previous result yields
\begin{align*} 
III\, =&\, \displaystyle   \left[ \int_{\gamma(x,\eta)}^{+\infty} G(z,\tau, \xi ,\eta) u(x,\xi,\tau) d\xi\right]^{\eta=\tau}_{\eta=0} - \int_0^\tau  G(z,\tau,\gamma(x,\eta),\eta)\Psi(x,\eta)u(x,\gamma(x,\eta),\eta)d\eta\\
=&\,  u(x,z,\tau) - \int_{0}^{+\infty} G(z,\tau, \xi ,0) u^0(x,\xi) d\xi - \int_0^\tau  G(z,\tau,\gamma(x,\eta),\eta)\Psi(x,\eta) u(x,\gamma(x,\eta),\eta)d\eta.
\end{align*}
Hence, we conclude
\begin{align}\label{eq:Duhamel}
    u(x,z,\tau) =&\, \int_{0}^{+\infty} G(z,\tau, \xi ,0) u^0(x,\xi) d\xi + \int_{0}^{\tau} \dfrac{\partial G}{\partial \xi}(z,\tau,\gamma(x,\eta),\eta)u(x,\gamma(x,\eta),\eta) d\eta.
\end{align}
Let us denote $v(x,\tau)= u(x,\gamma(x,\tau),\tau)$. Before passing to the limit $z\to \gamma(x,\tau)$ to get the boundary condition, we handle with care the most singular term: by \cite[Lem. 1, p. 217]{god}, as long as the function $v$ is continuous and the function $\gamma$ is Lipschitz with respect to the second variable (these regularity assumptions will be enforced by construction later on), we have for each $x\in\mathbb T^d$
\[
\lim_{z\to \gamma(x,\tau)} \int_0^\tau \dfrac{\partial G}{\partial \xi}(z,\tau,\gamma(x,\eta),\eta) v(x,\eta) d\eta = \dfrac{v(x,\tau)}{2} + \int_0^\tau \dfrac{\partial G}{\partial \xi}(\gamma(x,\tau),\tau, \gamma(x,\eta),\eta) v(x,\eta) d\eta.  
\]
Hence, taking the limit $z\to \gamma(x,\tau)$ in \eqref{eq:Duhamel}, we obtain
\begin{equation}\label{du}
    v(x,\tau) = 2 \int_{0}^{+\infty} G(\gamma(x,\tau),\tau, \xi ,0) u^0(x,\xi) d\xi +2\int_{0}^{\tau}  \dfrac{\partial G}{\partial \xi}(\gamma(x,\tau),\tau,\gamma(x,\eta),\eta)v(x,\eta) d\eta .
\end{equation}
Then, we take the average of \eqref{eq:Duhamel} to deduce
\begin{align}\label{du_hast}
    \bar u (x,\tau) =& \, \int_{0}^{+\infty} \int_{\gamma(x,\tau)}^{+\infty} zG(z,\tau, \xi ,0)dz u^0(x,\xi) d\xi\nonumber\\& + \int_{0}^{\tau} \int_{\gamma(x,\tau)}^{+\infty} z \dfrac{\partial G}{\partial \xi}(z,\tau,\gamma(x,\eta),\eta) dz\, v(x,\eta) d\eta .
\end{align}
Last, the boundary $\gamma(x,\tau)$ satisfies
\begin{equation}\label{du_hast_mich}
    \gamma(x,\tau) =  - \int_0^\tau \Phi\big( \alpha(\eta) W \ast [ \bar u (x,\eta) - \gamma(x,\eta)] + \beta(\eta)\big)\alpha(\eta)d\eta =  - \int_0^\tau \Psi(x,\eta)d\eta .
\end{equation}
Hence, the triplet $(v, \gamma, \bar u)$ satisfies the coupled Duhamel formulae \eqref{du}--\eqref{du_hast}--\eqref{du_hast_mich} that we write in short as
\begin{equation}
    \left\{\begin{array}{rcl}
        v(x,\tau)      &=&   F_v[v,\gamma, \bar u](x,\tau) \\
        \gamma(x,\tau) &=&   F_\gamma[v,\gamma,\bar u](x,\tau)\\
        \bar u(x,\tau) &=&  F_{\bar u}[v,\gamma,\bar u](x,\tau).
    \end{array}\right.
\end{equation}

\subsection{Existence of a local solution}

A natural idea that is akin to the method used in \cite{CGGS} would be to look for a continuous solution of the closed system \eqref{du}-\eqref{du_hast}-\eqref{du_hast_mich} by applying the Banach fixed point theorem in a closed bounded subspace of $C^0(\mathbb T^d\times [0,\tau_0])^3$. An issue with this strategy is that we need slightly higher regularity for $\gamma$ than for $\bar u$ and $v$. Rather than including this regularity into the space, which would raise additional technical difficulties, it is more convenient to hide the equation for $\gamma$ inside a two-component system for $(v,\bar u)$ and derive some handy estimates on $\gamma$ along the way. Let us consider $\tau_0 \in (0,1]$; the choice $\tau_0\leqslant 1$ is arbitrary and serves the sole purpose of estimating the contribution of the exterior input function $B(t)$. We define the following mapping:
\begin{equation}
    \begin{array}{rccl}
         \mathcal T:& C^0(\mathbb T^d\times [0,\tau_0])^2 & \to & C^0(\mathbb T^d\times [0,\tau_0])^2  \\
                   & (v,\bar u) & \mapsto &  ( F_v[v,\gamma, \bar u], F_{\bar u}[v,\gamma,\bar u]), 
    \end{array}
\end{equation}
where the function $\gamma$ is the unique continuous solution of the initial value problem
\begin{equation}\label{eq:gamma_ODE}
     \left\{\begin{array}{rcll}
          \dfrac{\partial \gamma}{\partial \tau}(x,\tau) & = & -\Phi\big( \alpha(\tau) W \ast [ \bar u (x,\tau) - \gamma(x,\tau)] + \beta(\tau)\big)\alpha(\tau),  \quad & x\in\mathbb T^d,\ \tau\in(0,\tau_0], \\
          \gamma(x,0)& = &  \gamma^0(x),       &x\in\mathbb T^d,
     \end{array}\right.
\end{equation}
as we prove next.
\begin{lemma}\label{lm:gamma}
    Assume $\Phi$ is locally Lipschitz and $W\in L^1(\mathbb T^d)$. Let $m>0$ be a positive real number. There exists $\tau_0\in(0,1]\,$ small enough, depending only on $\Phi,\beta,W$ and $m$, such that for all $\bar u\in C^0(\mathbb T^d\times [0,\tau_0])$ and $\gamma^0\in C^0(\mathbb T^d)$ satisfying 
    \[  \norme{\bar u}_{L^\infty(\mathbb T^d\times[0,\tau_0])} \leqslant m, \qquad \mathrm{and} \qquad \norme{\gamma^0}_{L^\infty(\mathbb T^d)}\leqslant \frac m2, \]
    there exists a unique solution $\gamma\in C^0(\mathbb T^d\times [0,\tau_0])$ to \eqref{eq:gamma_ODE}, and for all $x\in \mathbb T^d$, $\tau\mapsto \gamma(x,\tau)$ is a $C^1$ function. Moreover, there exists a constant $C_\Psi>0$ depending only on $\Phi,\Psi,\norme{W}_{L^1},m$ such that
    \begin{equation*}
        \norme{\dfrac{\partial \gamma}{\partial \tau}}_{L^\infty(\mathbb T^d\times [0,\tau_0])} = \norme{\Psi}_{L^\infty(\mathbb T^d\times [0,\tau_0])} \leqslant C_\Psi, \qquad \mathrm{and} \qquad \norme{\gamma}_{L^\infty(\mathbb T^d\times [0,\tau_0])} \leqslant m.
    \end{equation*}
    and, for all $\bar u_1, \bar u_2 \in  C^0(\mathbb T^d\times [0,\tau_0])$ satisfying $\norme{\bar u_1},\norme{\bar u_2}\leqslant m$, if $\gamma_1$ and $\gamma_2$ are the respective solutions of \eqref{eq:gamma_ODE}, then for all $\tau\in(0,\tau_0)$,
    \begin{equation*}
        \norme{\gamma_1 - \gamma_2}_{L^{\infty}(\mathbb T^d\times [0,\tau])} \leqslant \tau C_\Psi\norme{\bar u_1 - \bar u_2 }_{L^{\infty}(\mathbb T^d\times [0,\tau_0])}.
    \end{equation*}
\end{lemma}

\begin{proof}
Note first that $\alpha(\tau)=e^{-t}\leqslant 1$. Then, given $\bar u \in C^0(\mathbb T^d\times [0,\tau_0])$, define the mapping
\begin{equation}
    \begin{array}{rccl}
         \mathcal T_{\bar u}:& C^0(\mathbb T^d\times [0,\tau_0]) & \to & C^0(\mathbb T^d\times [0,\tau_0]),  \\
                   & \gamma & \mapsto &\displaystyle  \tilde \gamma,
    \end{array}
\end{equation} 
where
\[\tilde \gamma (x,\tau) = \gamma^0(x) - \int_0^\tau \Phi\big( \alpha(\eta) W \ast [ \bar u (x,\eta) - \gamma(x,\eta)] + \beta(\eta)\big)\alpha(\eta) d\eta.\]
For any $\tau_1\in (0,1]$, define the space 
\[  \mathcal U_{\tau_1,m} = \{ \gamma \in C^0(\mathbb T^d \times [0,\tau_1])\ | \  \norme{\gamma}_{L^{\infty}(\mathbb T^d\times [0,\tau_1])} \leqslant m\},  \]
We will prove that for $\tau_1$ small enough, $\mathcal T_{\bar u}$ is a contraction on the complete metric space $\mathcal U_{\tau_1,m}$ endowed witht the norm $\norme{\cdot}_{L^{\infty}(\mathbb T^d\times [0,\tau_1])}$ and therefore admits a unique fixed point $\gamma$, the solution to \eqref{eq:gamma_ODE}.

Let us prove first that $\mathcal T_{\bar u}\left( \mathcal U_{\tau_1,m}\right) \subset \mathcal U_{\tau_1,m}$ for $\tau_1$ small enough. For all $\gamma\in C^0(\mathbb T^d\times [0,\tau_1])$, recalling that $\tau_1\leqslant1$, we can write the bound
    \[ |\alpha(\eta) W \ast [ \bar u (x,\eta) - \gamma(x,\eta)] + \beta(\eta)| \leqslant 2\norme{W}_{L^1} m + \norme{\beta}_{L^\infty(0,1)}  =: C. \]
    Hence, taking into account \eqref{eq:Stefan} we have
    \[ \norme{\Psi}_{L^{\infty}(\mathbb T^d\times [0,\tau_1])} \leqslant \norme{\Phi}_{L^\infty(0,C)},    \]
    and
    \[ \norme{\mathcal T_{\bar u}(\gamma)}_{L^{\infty}(\mathbb T^d\times [0,\tau_1])} \leqslant \norme{\gamma^0}_{L^\infty(\mathbb T^d)} + \int_{0}^{\tau_1} \norme{\Phi}_{L^\infty(0,C)}d \eta \leqslant  \frac m2 +  \tau_1\norme{\Phi}_{L^\infty(0,C)}.   \]
    If we choose $\tau_1$ small enough with respect to $\Phi$, $\beta$, $\norme{W}_{L^1}$ and $m$, then $\mathcal T_{\bar u}\left( \mathcal U_{\tau_1,m}\right) \subset \mathcal U_{\tau_1,m}$.
    
    Let us prove now that $\mathcal T_{\bar u}$ is a contraction for $\tau_1$ small enough. Let $\gamma_1,\gamma_2\in \mathcal U_{\tau_1,m}$. The function $\Phi$ being locally Lipschitz, there exists a constant $C_\Phi$ depending only on $C$ such that for all $p_1,p_2\in [-C,C]$, $ |\Phi(p_1)-\Phi(p_2)| \leqslant C_\Phi |p_1 - p_2| $. 
    Therefore,
    \begin{align*}
       |\mathcal T_{\bar u}(\gamma_1) - \mathcal T_{\bar u}(\gamma_2)| &\leqslant \int_0^\tau \big|\Phi(\alpha(\eta) W \ast [ \bar u (x,\eta) - \gamma_1(x,\eta)] + \beta(\eta)) \\
       &\qquad\quad- \Phi(\alpha(\eta) W \ast [ \bar u (x,\eta) - \gamma_2(x,\eta)] + \beta(\eta))\big| d\eta \\
        &\leqslant \int_0^\tau C_\Phi |W \ast [ \gamma_1 - \gamma_2](x,\eta)| d\eta \\
        & \leqslant \tau_1 C_\Phi \norme{W}_{L^1} \norme{\gamma_1 -\gamma_2}_{L^{\infty}(\mathbb T^d\times [0,\tau_1])}.
    \end{align*}
    If $\tau_1$ is small enough, then $\tau_1 C_\Phi \norme{W}_{L^1} < 1$ and $\mathcal T_{\bar u}$ is a contraction on $\mathcal U_{\tau_1,m}$.
    Hence, there exists a unique solution $\gamma$ to \eqref{eq:gamma_ODE} in $C^0(\mathbb T^d\times [0,\tau_1])$ and it satisfies $\norme{\gamma}_{L^\infty(\mathbb T^d\times [0,\tau_1])} \leqslant m$. For all $x\in\mathbb T^d$, $t\mapsto \gamma(x,t)$ is differentiable and admits the derivative $\Psi$. Since $\bar u, \gamma\in C^0(\mathbb T^d\times [0,\tau_1])$, $\Psi\in C^0(\mathbb T^d\times [0,\tau_1])$; thus $t\mapsto \gamma(x,t)$ is a $C^1$ function on $[0,\tau_1]$. The previous computations ensure that
    \[\norme{\dfrac{\partial \gamma}{\partial \tau}}_{L^\infty(\mathbb T^d\times [0,\tau_1])} = \norme{\Psi}_{L^\infty(\mathbb T^d\times [0,\tau_1])} \leqslant C_\Psi^1(\tau_1) := C(\Phi,\beta,\norme{W}_{L^1},m,\tau_1).\]
    Last, considering two functions $\bar u_1, \bar u_2\in \mathcal U_{\tau_1,m}$ and denoting $\gamma_1,\gamma_2$ the corresponding solutions, we compute like above
    \begin{align*}
       |\gamma_1 -\gamma_2 | &\leqslant \int_0^\tau \big|\Phi(\alpha(\eta) W \ast [ \bar u_1 (x,\eta) - \gamma_1(x,\eta)] + \beta(\eta)) \\
       &\qquad\quad- \Phi(\alpha(\eta) W \ast [ \bar u_2 (x,\eta) - \gamma_2(x,\eta)] + \beta(\eta))\big| d\eta \\
        &\leqslant \int_0^\tau C_\Phi |W \ast [\bar u_1 - \bar u_2 + \gamma_1 - \gamma_2](x,\eta)| d\eta \\
        & \leqslant \tau C_\Phi \norme{W}_{L^1} \norme{\gamma_1 -\gamma_2}_{L^{\infty}(\mathbb T^d\times [0,\tau])} + \tau C_\Phi \norme{W}_{L^1} \norme{\bar u_1 -\bar u_2}_{L^{\infty}(\mathbb T^d\times [0,\tau_1])}.
    \end{align*}
    Taking the supremum on $\mathbb T^d\times [0,\tau]$ in the left-hand side of the last inequality, we obtain
    \[  \norme{\gamma_1 -\gamma_2}_{L^{\infty}(\mathbb T^d\times [0,\tau])} \leqslant \dfrac{\tau\, C_\Phi \norme{W}_{L^1}}{1 -  \tau_1 C_\Phi \norme{W}_{L^1} } \norme{\bar u_1 -\bar u_2}_{L^{\infty}(\mathbb T^d\times [0,\tau_1])}.  \]
    By the previous choice of $\tau_1$, we are sure that the division was allowed and that
    \[ C_\Psi^2(\tau_1) :=  \dfrac{ C_\Phi \norme{W}_{L^1}}{1 -  \tau_1 C_\Phi \norme{W}_{L^1} }  >0. \]
    The constants $C_\Psi^1(\tau_1), C_\Psi^2(\tau_1)$ being non-increasing functions of $\tau_1$, we can choose any $\tau_0\in(0,\tau_1)$ and use the constant $C_\Psi = \max(C_\Psi^1(\tau_1), C_\Psi^2(\tau_1) )$ as a bound that is independent of $\tau_0$ small enough (smaller than $\tau_1$), which concludes the proof of the lemma.
\end{proof}

\begin{remark}[Blow-up of the free boundary]
    In the previous lemma, we treated the case where the initial position of the boundary is not 0 and we can see in the proof that $\norme{\gamma^0}_{L^\infty(\mathbb T^d)}$ impacts the local existence time $\tau_0$. Note that, regardless of $\bar u$, it is easy to construct examples of blow-up of the free boundary: taking $\Phi(p) = p^2$, $\gamma^0$ constant, $W(x)\equiv 1$, $\beta\equiv 0$, and $\bar u \equiv 0$, \eqref{eq:gamma_ODE} reduces to $\gamma'(\tau) = \alpha(\tau)^3 \gamma(\tau)^2$. If the initial value of $\gamma^0$ is spatially constant and large enough, this equation blows-up in finite time before $\alpha(\tau)$ has the chance to dampen the dynamics.
    
\end{remark} 

An initial position of the boundary which is not uniformly 0 allows us to take a solution of the Stefan problem at some time $\tau_0$ and construct a new local solution on $[\tau_0,\tau_1]$ from there. However, the initial position of the boundary would lead to unecessary technical difficulties in the following proofs. Therefore, we assume without loss of generality that $\gamma^0\equiv 0$. This can be achieved by a change of variable $z' = z - \gamma^0$ and it comes with a minor modification of $u^0$. The effect of $\gamma^0$ upon the local time of existence is not lost in this process, for the new initial average $\bar u^0(x) + \gamma^0(x) \norme{u^0(x)}_{L^1(\gamma^0(x),+\infty)}$ now bears it. This change of variables has to be done at each final time of the intervals in which the extension of the solution is performed.

\begin{proposition}\label{prop:local}
    Assume $\Phi$ is locally Lipschitz, $\beta\in C^0(\R_+)$ and $W\in L^1(\mathbb T^d)$. For all compatible initial conditions 
    $u^0$, there exists a unique maximal local-in-time solution $(v,\gamma,\bar u)\in C^0(\mathbb T^d \times [0,\tau^*))^3$ to the system \eqref{du}--\eqref{du_hast}--\eqref{du_hast_mich}. The maximal time of existence $\tau^*$ satisfies
    \begin{equation}\label{eq:time_max}
        \sup\,\{\ \tau\in\R_+^* \ | \  \sup_{x\in\mathbb T^d}\max(|v(x,\tau)|,|\gamma(x,\tau)|,|\bar u(x,\tau)|) <  +\infty \ \}.
    \end{equation}
\end{proposition}

\begin{proof}
For any $\tau_0\in(0,1]$ and $m\in(0,+\infty)$, consider the space
    \[  \mathcal C_{\tau_0,m} = \{ \textbf{w} \in C^0(\mathbb T^d \times [0,\tau_0])^2\ | \  \norme{\textbf{w}}_{\infty} < m\},  \]
where for any $\textbf{w}=(v,\bar u)$, we denote $\norme{\textbf{w}}_\infty = \max(\ \norme{v}_{L^{\infty}(\mathbb T^d\times [0,\tau_0])}\ ,\  \norme{\bar u}_{L^{\infty}(\mathbb T^d\times [0,\tau_0])}  \ )$.
Let us choose
\[ m = \max\left(\ 2\norme{u^0}_{L^\infty(\mathbb T^d \times \R_+)}\ ,\ \norme{z u^0}_{L^\infty(\mathbb T^d, L^1(\R_+))}\ \right) + 1.  \]
We are going to proceed with the following steps.
\begin{enumerate}
    \item Compute some preliminary upper bounds for quantities which appear a lot in the next steps.
    \item Prove that with the choice of $m$ above, if $\tau$ is small enough then $\mathcal T ( \mathcal C_{\tau_0,m}) \subset \mathcal C_{\tau_0,m} $.
    \item Prove that if $\tau_0$ is small enough, $\mathcal T$ is a contraction on $C_{\tau_0,m}$, \textit{i.e.} there exists $k\in(0,1)$ such that
    \[ \forall w_1,w_2\in \mathcal C_{\tau_0,m}, \quad \norme{\mathcal T(w_1)-\mathcal T(w_2)}_\infty \leqslant k\norme{w_1-w_2}_\infty \]
    \item Apply the Banach fixed point theorem on $C_{\tau_0,m}$ and conclude the existence of a local-in-time solution, then conclude that there is a maximal solution and a naturally associated criterion for finite maximal time.\\
\end{enumerate}

{\bf Step 1. Derive preliminary estimates.-}   By Lemma \ref{lm:gamma} and owing to $\norme{\bar u}_{L^{\infty}(\mathbb T^d\times [0,\tau_0])} \leqslant m$, we have the Lipschitz bound
    \begin{equation}\label{eq:lipsch_boundary}
        |\gamma(x,\tau) - \gamma(x,\eta)| \leqslant C_\Psi | \tau - \eta |.
    \end{equation}
    Moreover, we compute
    \begin{equation}\label{eq:partial_G}
        \left| \dfrac{\partial G}{\partial \xi} (z,\tau,\xi,\eta)\right| = \dfrac{1}{2\sqrt{4\pi}}\dfrac{|z - \xi|}{|\tau-\eta|^{\frac{3}{2}}}e^{-\frac{(z-\xi)^2}{4(\tau-\eta)}}.
    \end{equation}
    Combined with the estimate \eqref{eq:lipsch_boundary}, it yields
    \begin{equation}\label{eq:int_partial_G}
        \begin{array}{rcl}
        \displaystyle\int_0^\tau \left| \dfrac{\partial G}{\partial \xi}(\gamma(x,\tau),\tau,\gamma(x,\eta),\eta)\right| d\eta &\leqslant&\displaystyle \dfrac{C_\Psi}{2\sqrt{4\pi}} \int_0^\tau \dfrac{1}{\sqrt{\tau-\eta}}d \eta\\[4mm]
        &=& \dfrac{C_\Psi}{2\sqrt{\pi}} \sqrt{\tau}.\end{array}
    \end{equation}
    Then, given $\bar u_1,\bar u_2\in \mathcal C_{\tau_0,m}$, $\gamma_1,\gamma_2$ the associated boundary positions and denoting for $i=1,2$,
    \begin{align}\label{eq:psi}
    \Psi_i =  \Phi\big( \alpha(\tau) W \ast [ \bar u_i (x,\tau) - \gamma_i(x,\tau)] + \beta(\tau)\big)\alpha(\tau), 
    \end{align}
    we can apply Lemma \ref{lm:gamma} and parts of its proof in order to conclude that there exists $K_\Psi>0$ depending only on $\Phi,\norme{W}_{L^1},\beta,m$ such that
    \begin{equation}\label{eq:Psi_1-Psi_2}
        \norme{\Psi_1-\Psi_2}_{L^\infty(\mathbb T^d\times[0,\tau_0])} \leqslant K_\Psi \norme{\bar u_1 - \bar u_2}_{L^\infty(\mathbb T^d \times [0,\tau_0])}.
    \end{equation}
    In view of the formula \eqref{du_hast} for $\bar u$, we also pre-compute some of the integrals involved in order to make our life more convenient later on. Let us recall first that, given the error function
    \[\erf(z) = \dfrac{2}{\sqrt\pi}\int_0^z e^{-y^2}d y \in (-1,1), \]
   for all $\sigma>0, \mu\in\R$, we have
    \begin{equation}\label{eq:error} \int_0^{+\infty} e^{-\frac{(y-\mu)^2}{2\sigma}} dy = \sqrt{\dfrac{\pi\sigma}{2}}\left(1+ \erf\left( \dfrac{\mu}{\sqrt{2\sigma}}  \right)\right) \,.
    \end{equation}
    Using \eqref{eq:error}, we compute
    \begin{align*}  \int_{\gamma(x,\tau)}^{+\infty} zG(z,\tau, \xi ,\eta)dz
      &=  \int_{\gamma(x,\tau)}^{+\infty} z\dfrac{1}{\sqrt{4\pi(\tau-\eta)}}e^{-\frac{(z-\xi)^2}{4(\tau-\eta)}}dz
    \\
    &= \dfrac{\sqrt{\tau-\eta}}{\sqrt\pi}\left( \int_{\gamma(x,\tau)}^{+\infty} 2(z-\xi)\dfrac{1}{4(\tau-\eta)}e^{-\frac{(z-\xi)^2}{4(\tau-\eta)}}dz\right.\\ &\qquad \qquad \qquad \qquad \qquad \qquad + \left. \int_{\gamma(x,\tau)}^{+\infty} 2\xi\dfrac{1}{4(\tau-\eta)}e^{-\frac{(z-\xi)^2}{4(\tau-\eta)}}dz\right)\\
    &= \dfrac{\sqrt{\tau-\eta}}{\sqrt\pi}\left( \left[ - e^{-\frac{(z-\xi)^2}{4(\tau-\eta)}}  \right]_{z = \gamma(x,\tau)}^{z = +\infty} + \dfrac{\xi}{2(\tau-\eta)} \int_0^{+\infty} e^{-\frac{(\omega + \gamma(x,\tau) - \xi)^2}{4(\tau-\eta)}} d\omega \right) \\
    &= \dfrac{\sqrt{\tau-\eta}}{\sqrt{\pi}} e^{-\frac{(\xi-\gamma(x,\tau))^2}{4(\tau-\eta)}} + \dfrac{\xi\sqrt{\pi(\tau-\eta)}}{2\sqrt{\pi(\tau-\eta)}} \left(1+\erf\left(\dfrac{\xi-\gamma(x,\tau)}{2\sqrt{\tau-\eta}}\right)\right),
    \end{align*}
    which yields
    \begin{equation}\label{eq:brick1}
        \int_{\gamma(x,\tau)}^{+\infty} zG(z,\tau, \xi ,\eta)dz =  \sqrt{\dfrac{\tau-\eta}{\pi}} e^{-\frac{(\xi-\gamma(x,\tau))^2}{4(\tau-\eta)}} + \dfrac{\xi}{2} \left(1+\erf\left(\dfrac{\xi-\gamma(x,\tau)}{2\sqrt{\tau-\eta}}\right)\right).
    \end{equation}
    We also compute by integration by parts and application of \eqref{eq:error},
    \begin{align*}
         \int_{\gamma(x,\tau)}^{+\infty} & z \dfrac{\partial G}{\partial \xi}(z,\tau,\gamma(x,\eta),\eta) dz\\
        &=\int_{\gamma(x,\tau)}^{+\infty} z \dfrac{2(z-\gamma(x,\eta))}{\big(4\pi(\tau-\eta)\big)^{\frac32}} e^{-\frac{(z-\gamma(x,\eta))^2}{4(\tau-\eta)}} dz \\ &= \dfrac{1}{\sqrt{4\pi(\tau-\eta)}}\int_{\gamma(x,\tau)}^{+\infty} z \dfrac{2 (z-\gamma(x,\eta))}{4(\tau-\eta)} e^{-\frac{(z-\gamma(x,\eta))^2}{4(\tau-\eta)}} dz\\
     &=\dfrac{-1}{\sqrt{4\pi(\tau-\eta)}}\left(\left[ z e^{-\frac{(z-\gamma(x,\eta))^2}{4(\tau-\eta)}}  \right]_{z = \gamma(x,\tau)}^{z=+\infty} - \int_{\gamma(x,\tau)}^{+\infty} e^{-\frac{(z-\gamma(x,\eta))^2}{4(\tau-\eta)}} dz\right) \\
     &=\dfrac{1}{\sqrt{4\pi(\tau-\eta)}}\left(\gamma(x,\tau) e^{-\frac{(\gamma(x,\tau)-\gamma(x,\eta))^2}{4(\tau-\eta)}}   + \int_{0}^{+\infty} e^{-\frac{(\omega + \gamma(x,\tau) - \gamma(x,\eta))^2}{4(\tau-\eta)}} d\omega\right),
    \end{align*}
    which in combination with \eqref{eq:error} becomes
    \begin{multline}\label{eq:brick2}
        \quad \int_{\gamma(x,\tau)}^{+\infty} z \dfrac{\partial G}{\partial \xi}(z,\tau,\gamma(x,\eta),\eta) dz = \dfrac{1}{\sqrt{4\pi(\tau-\eta)}}\gamma(x,\tau) e^{-\frac{(\gamma(x,\tau)-\gamma(x,\eta))^2}{4(\tau-\eta)}}  \\ + \dfrac12\left(1+\erf\left( \dfrac{\gamma(x,\eta)-\gamma(x,\tau)}{2\sqrt{\tau-\eta}}\right)\right).
    \end{multline}
    Last, we want to estimate terms of the form
    \begin{align*}
        \left| G(\gamma_1(x,\tau),\tau, \gamma_1(x,\eta) ,\eta) \right.&-\left. G(\gamma_2(x,\tau),\tau, \gamma_2(x,\eta) ,\eta)   \right| \\
        =&\, \dfrac{1}{\sqrt{4\pi (\tau-\eta)}} \left| e^{-\frac{(\gamma_1(x,\tau)-\gamma_1(x,\eta))^2}{4(\tau-\eta)}} - e^{-\frac{(\gamma_2(x,\tau)-\gamma_2(x,\eta))^2}{4(\tau-\eta)}} \right|.  
    \end{align*} 
    Let us introduce for this purpose the function $(x,\tau)\mapsto \lambda(x,\tau) = \gamma_1(x,\tau)-\gamma_2(x,\tau)$, which satisfies $\partial_\tau \lambda = \Psi_1 - \Psi_2$, where $\Psi_i$ is defined in \eqref{eq:psi}. We apply the Lipschitz inequality \eqref{eq:lipsch_boundary}, the mean value theorem to $\lambda$ and \eqref{eq:Psi_1-Psi_2}, leading to 
    \begin{align*}
        \left|(\gamma_1(x,\tau)-\gamma_1(x,\eta))^2\right.&-\left.(\gamma_2(x,\tau)-\gamma_2(x,\eta))^2\right|\\
        =&\,\big| \gamma_1(x,\tau) - \gamma_1(x,\eta) + \gamma_2(x,\eta) - \gamma_2(x,\tau) \big|\,\big| \lambda(x,\tau) - \lambda(x,\eta) \big|\\
        \leqslant&\, 2 C_\Psi | \tau - \eta | \norme{\frac{\partial \lambda}{\partial\tau}}_{L^\infty(\mathbb T^d\times[0,\tau_0])} |\tau-\eta|\\
        \leqslant&\, 2 C_\Psi \norme{\Psi_1-\Psi_2}_{L^\infty(\mathbb T^d\times[0,\tau_0])} |\tau-\eta|^2 \\
        \leqslant &\, 2 C_\Psi K_\Psi \norme{\bar u_1 -\bar u_2}_{L^\infty(\mathbb T^d\times[0,\tau_0])} |\tau-\eta| \tau_0.
    \end{align*}
    Since the exponential function is $1$-Lipschitz on $(-\infty,0]$, if we denote $ C_\lambda= \tfrac12 C_\Psi K_\Psi  $, the previous bound yields
    \begin{align}
        \left| e^{-\frac{(\gamma_1(x,\tau)-\gamma_1(x,\eta))^2}{4(\tau-\eta)}} - e^{-\frac{(\gamma_2(x,\tau)-\gamma_2(x,\eta))^2}{4(\tau-\eta)}} \right|\,
        \leqslant \, C_\lambda\, \tau_0 \norme{\textbf{w}_1 - \textbf{w}_2}_\infty.\label{eq:brick3a}
    \end{align}
    We have thus proved
    \begin{equation}\label{eq:brick3}
        \left| G(\gamma_1(x,\tau),\tau, \gamma_1(x,\eta) ,\eta) - G(\gamma_2(x,\tau),\tau, \gamma_2(x,\eta) ,\eta)   \right| 
        \leqslant  \dfrac{C_\lambda\, \tau_0 }{\sqrt{4\pi (\tau-\eta)}}\norme{\textbf{w}_1 - \textbf{w}_2}_\infty.
    \end{equation}
    
    {\bf Step 2. Prove that $\mathcal T ( \mathcal C_{\tau_0,m}) \subset \mathcal C_{\tau_0,m} $ for $\tau_0$ small enough.-} Consider $m$ defined in the beginning of the proof.
    Let $\textbf{w}=(v,\bar u)\in C_{\tau_0,m}$ and denote $\mathcal T(\textbf{w}) = (\tilde v,\tilde u)$.
    
    We use the decomposition
    \[ \tilde v = 2( V_1  + V_2), \]
    with
    \begin{align*} V_1 &= \int_{0}^{+\infty} G(\gamma(x,\tau),\tau, \xi ,0) u^0(x,\xi) d\xi, \\ V_2&= \int_{0}^{\tau}\dfrac{\partial G}{\partial \xi}(\gamma(x,\tau),\tau,\gamma(x,\eta),\eta) v(x,\eta) d\eta. \end{align*}
    We have first
    \[ |V_1| \leqslant \norme{u^0}_\infty \int_{0}^{+\infty} G(\gamma(x,\tau),\tau, \xi ,0) d\xi \leqslant \norme{u^0}_\infty . \]
    Then, owing to \eqref{eq:int_partial_G}, we can write
    \begin{align*}|V_2| &\leqslant  m \int_{0}^\tau  \left| \dfrac{\partial G}{\partial \xi}(\gamma(x,\tau),\tau,\gamma(x,\eta),\eta)\right| d \eta \leqslant  \dfrac{mC_\Psi}{2\sqrt{\pi}} \sqrt{\tau_0}.\end{align*}
    In the same vein, we write for the second component
    \[ \tilde u = U_1 + U_2, \]
    where
    \begin{align*} U_1 &= \int_{0}^{+\infty} \int_{\gamma(x,\tau)}^{+\infty} zG(z,\tau, \xi ,0)dz u^0(x,\xi) d\xi,
    \\  U_2&= \int_{0}^{\tau} \int_{\gamma(x,\tau)}^{+\infty} z \dfrac{\partial G}{\partial \xi}(z,\tau,\gamma(x,\eta),\eta) dz v(x,\eta) d\eta. \end{align*}
    Using \eqref{eq:brick1} with $\eta=0$ and \eqref{eq:error}, we can perform the bounds
    \begin{align*} |U_1| &\leqslant \norme{u^0}_\infty \int_{0}^{+\infty} \sqrt{\dfrac{\tau}{\pi}} e^{-\frac{(\xi-\gamma(x,\tau))^2}{4\tau}} d\xi + \int_{0}^{+\infty} \left|\dfrac{\xi}{2} \left(1+\erf\left(\dfrac{\xi-\gamma(x,\tau)}{2\sqrt{\tau}}\right)\right) u^0(x,\xi) \right|d\xi
    \\ 
    &\leqslant \norme{u^0}_\infty \sqrt{\dfrac{\tau}{\pi}}\int_{0}^{+\infty} e^{-\frac{(\xi-\gamma(x,\tau))^2}{4\tau}}   d\xi + \int_{0}^{+\infty} \left|\xi u^0(x,\xi)\right| d\xi\\
    &\leqslant \norme{u^0}_\infty\sqrt{\dfrac{\tau}{\pi}} \sqrt{\pi\tau}\left(1+ \erf\left( \dfrac{\gamma(x,\tau)}{\sqrt{4\tau}}  \right)\right) + \sup_{x\in \mathbb T^d}\int_{0}^{+\infty} \left|\xi u^0(x,\xi)\right| d\xi\\
    &\leqslant 2\norme{u^0}_\infty \tau_0 + \sup_{x\in \mathbb T^d}\int_{0}^{+\infty} \left|\xi u^0(x,\xi)\right| d\xi. \end{align*}
    Eventually, we have
    \begin{equation*} |U_2|\leqslant m \int_{0}^{\tau} \left| \int_{\gamma(x,\tau)}^{+\infty} z \dfrac{\partial G}{\partial \xi}(z,\tau,\gamma(x,\eta),\eta) dz\right| d\eta. \end{equation*}
    By \eqref{eq:brick2}, 
    \[ \left| \int_{\gamma(x,\tau)}^{+\infty} z \dfrac{\partial G}{\partial \xi}(z,\tau,\gamma(x,\eta),\eta) dz \right| \leqslant \dfrac{|\gamma(x,\tau)|}{\sqrt{4\pi(\tau-\eta)}} + 1 \leqslant \dfrac{\norme{\gamma}_{L^\infty(\mathbb T^d \times [0,\tau_0])}}{\sqrt{4\pi(\tau-\eta)}} + 1 , \]
    Thereby, by Lemma \ref{lm:gamma} we conclude
    \[ |U_2| \leqslant \dfrac{m^2}{\sqrt\pi} \sqrt{\tau_0} + m\tau_0. \]
    
    Collecting the four previous bounds for $V_1,V_2,U_1,U_2$, we obtain
    \[ \norme{\tilde v}_{L^\infty(\mathbb T^d \times [0,\tau_0])} \leqslant 2 \norme{u^0}_{\infty} + \dfrac{mC_\Psi}{2\sqrt\pi} \sqrt{\tau_0} \]
    and
    \begin{align*}
    \norme{\tilde u}_{L^\infty(\mathbb T^d \times [0,\tau_0])} \leqslant  \norme{z u^0}_{L^\infty(\mathbb T^d,L^1(\R_+))} + \left(  \dfrac{m^2}{\sqrt\pi} +  \left[ 2 \norme{u^0}_{\infty}  + m  \right]\sqrt{\tau_0} \right)\sqrt{\tau_0}.    
    \end{align*} 
    Therefore, if $\tau_0$ is chosen small enough, $\norme{\mathcal T(\textbf{w})}_\infty \leqslant m$ and $\mathcal T ( \mathcal C_{\tau_0,m}) \subset \mathcal C_{\tau_0,m} $.
    
    {\bf Step 3. Prove that $\mathcal T$ is a contraction on $\mathcal C_{\tau_0,m}$ for $\tau_0$ small enough.-}\\
    Let $\textbf{w}_1=(v_1,\bar u_1)\in C_{\tau_0,m}$ and $\textbf{w}_2=(v_2,\bar u_2)\in C_{\tau_0,m}$ two initial elements of our functional space with their associated boundary functions $\gamma_1, \gamma_2$ and let us denote $\mathcal T(\textbf{w}_1) = (\tilde v_1,\tilde u_1)$, $\mathcal T(\textbf{w}_2) = (\tilde v_2,\tilde u_2)$.
    
    We use the same decompositions as in step 2:
    \[ \tilde v_1 = 2(V_1^1  + V_2^1),\qquad \tilde u_1 = U_1^1  + U_2^1,  \]
    and
    \[ \tilde v_2 = 2(V_1^2  + V_2^2),\qquad \tilde u_2 = U_1^2  + U_2^2.  \]
    We are going to pair each of these terms and prove that each one of those pairs can be bounded from above, for $\tau_0$ small enough, by $k\norme{\textbf{w}_1 - \textbf{w}_2}_\infty$ with $k=\frac18$ or $k=\frac14$.
    Our first pair is
    \[ V_1^1 - V_1^2 = \int_{0}^{+\infty} G(\gamma_1(x,\tau),\tau, \xi ,0) u^0(x,\xi) - \int_{0}^{+\infty} G(\gamma_2(x,\tau),\tau, \xi ,0) u^0(x,\xi) d\xi d\xi. \]
    We bound it from above,
    \begin{equation*}
       |V_1^1 - V_1^2| = \norme{u^0}_\infty\int_{0}^{+\infty} \left| G(\gamma_1(x,\tau),\tau, \xi ,0) - G(\gamma_2(x,\tau),\tau, \xi ,0)   \right|  d\xi.
    \end{equation*}
    By the mean value theorem, for all $\xi\in\R_+$, for all $\tau\in[0,\tau_0]$, for all $x\in\mathbb T^d$, there exists $\mu\in[\min(\gamma_1(x,\tau),\gamma_2(x,\tau)),\max(\gamma_1(x,\tau),\gamma_2(x,\tau))]$ such that
    \[ \left| G(\gamma_1(x,\tau),\tau, \xi ,0) - G(\gamma_2(x,\tau),\tau, \xi ,0)   \right| \leqslant \left|\dfrac{\partial G}{\partial z}(\mu,\tau, \xi ,0)\right| |\gamma_1(x,\tau)-\gamma_2(x,\tau)|.   \]
    Furthermore, using the inequality $ye^{-y^2}\leqslant e^{-\frac{y^2}{2}}$, we have
    \[ \left|\dfrac{\partial G}{\partial z}(\mu,\tau, \xi ,0)\right| = \dfrac{|\xi-\mu|}{4\sqrt \pi \tau^{\frac32}} e^{-\frac{(\xi-\mu)^2}{4\tau}} \leqslant \dfrac{1}{2\sqrt \pi \tau} e^{-\frac{(\xi-\mu)^2}{8\tau}} = \dfrac{\sqrt2}{\sqrt\tau} G(\mu,2\tau,\xi,0). \]
    The three previous inequalities, Lemma \ref{lm:gamma} and
    $ \int_0^{+\infty}G(\mu,2\tau,\xi,0) d\xi \leqslant 1 $
    yield the following bounds
    \begin{align*} |V_1^1 - V_1^2|   & \leqslant \norme{u_0}_\infty \frac{\sqrt 2}{\sqrt \tau} \int_0^{+\infty}G(\mu,2\tau,\xi,0) d\xi \norme{\gamma_1 - \gamma_2}_{L^{\infty}(\mathbb T^d\times [0,\tau])} \\
    &\leqslant \dfrac{\sqrt 2\norme{u_0}_\infty }{\sqrt \tau}\, \tau\, C_\Psi\norme{\bar u_1 - \bar u_2 }_{L^{\infty}(\mathbb T^d\times [0,\tau_0])}\leqslant \sqrt 2\norme{u_0}_\infty  C_\Psi \sqrt{\tau_0} \norme{\textbf{w}_1 - \textbf{w}_2}_\infty.\end{align*}
    Hence, choosing $\tau_0$ small enough, we can ensure that
    \begin{equation}\label{eq:b1}
        |V_1^1 - V_1^2| \leqslant \dfrac1{8} \norme{\textbf{w}_1 - \textbf{w}_2}_\infty.
    \end{equation}

    The second pair is none other than
    \begin{align*} V_2^1 - V_2^2 =&\, \int_{0}^{\tau} \dfrac{\partial G}{\partial \xi}(\gamma_1(x,\tau),\tau,\gamma_1(x,\eta),\eta) v_1(x,\eta)  d\eta\\ &- \int_{0}^{\tau} \dfrac{\partial G}{\partial \xi}(\gamma_2(x,\tau),\tau,\gamma_2(x,\eta),\eta) v_2(x,\eta)  d\eta, 
    \end{align*}
    that we readily decompose into $
     V_2^1 - V_2^2 = J_a + J_b $,
    with
    \begin{align*}
        J_a &= \int_{0}^{\tau}\left( \dfrac{\partial G}{\partial \xi}(\gamma_1(x,\tau),\tau,\gamma_1(x,\eta),\eta)- \dfrac{\partial G}{\partial \xi}(\gamma_2(x,\tau),\tau,\gamma_2(x,\eta),\eta)\right) v_1(x,\eta)  d\eta,\\
        J_b &= \int_{0}^{\tau} \dfrac{\partial G}{\partial \xi}(\gamma_2(x,\tau),\tau,\gamma_2(x,\eta),\eta) \big(v_1(x,\eta)-v_2(x,\eta)\big)  d\eta.
    \end{align*}
    We first write
    \[ |J_a| \leqslant m \int_{0}^{\tau}\left| \dfrac{\partial G}{\partial \xi}(\gamma_1(x,\tau),\tau,\gamma_1(x,\eta),\eta)- \dfrac{\partial G}{\partial \xi}(\gamma_2(x,\tau),\tau,\gamma_2(x,\eta),\eta)\right|  d\eta. \]
    We notice that
    \begin{align*}
        \left| \dfrac{\partial G}{\partial \xi}\right.&\left.(\gamma_1(x,\tau),\tau,\gamma_1(x,\eta),\eta) - \dfrac{\partial G}{\partial \xi}(\gamma_2(x,\tau),\tau,\gamma_2(x,\eta),\eta)\right| \\
        =&\, \dfrac{1}{2} \Big| \dfrac{\gamma_1(x,\tau)-\gamma_1(x,\eta)}{\tau-\eta} G(\gamma_1(x,\tau),\tau,\gamma_1(x,\eta),\eta) \\ & \qquad -  \dfrac{\gamma_2(x,\tau)-\gamma_2(x,\eta)}{\tau-\eta} G(\gamma_2(x,\tau),\tau,\gamma_2(x,\eta),\eta) \Big|  \\
        \leqslant &\,  \dfrac{1}{2} \left| \dfrac{(\gamma_1(x,\tau)-\gamma_1(x,\eta)) - (\gamma_2(x,\tau)-\gamma_2(x,\eta))}{\tau-\eta}\right| G(\gamma_1(x,\tau),\tau,\gamma_1(x,\eta),\eta) \\
        &\,  +  \dfrac{|\gamma_2(x,\tau)-\gamma_2(x,\eta)|}{2(\tau-\eta)} \big| G(\gamma_1(x,\tau),\tau,\gamma_1(x,\eta),\eta) - G(\gamma_2(x,\tau),\tau,\gamma_2(x,\eta),\eta) \big|.
    \end{align*}
    First, by the mean value theorem and inequality \eqref{eq:Psi_1-Psi_2}, there exists $\mu\in[\eta,\tau]$ such that
    \begin{align*}
         \frac12\left| \dfrac{(\gamma_1(x,\tau)-\gamma_1(x,\eta)) - (\gamma_2(x,\tau)-\gamma_2(x,\eta))}{\tau-\eta}\right|
        =&\, | \Psi_1(x,\mu) - \Psi_2(x,\mu)|\\
        \leqslant& \, \frac{K_\Psi}2  \norme{\bar u_1 - \bar u_2}_{L^\infty(\mathbb T^d \times [0,\tau_0])}.
    \end{align*}
    Then, by the Lipschitz estimate \eqref{eq:lipsch_boundary} and \eqref{eq:brick3},
    \begin{align*}
        & \dfrac{|\gamma_2(x,\tau)-\gamma_2(x,\eta)|}{2(\tau-\eta)} \big| G(\gamma_1(x,\tau),\tau,\gamma_1(x,\eta),\eta) - G(\gamma_2(x,\tau),\tau,\gamma_2(x,\eta),\eta) \big| \\
        \leqslant&\, \dfrac{C_\Psi C_\lambda\, \tau_0 }{2\sqrt{4\pi (\tau-\eta)}}\norme{\textbf{w}_1 - \textbf{w}_2}_\infty.
    \end{align*}
    Integrating in time the two final bounds of the difference of the derivatives of $G$, we find
   \[ |J_a| \leqslant \dfrac{m}{2\sqrt\pi}\left(  C_\Psi  C_\lambda  \tau_0 +    K_\Psi \right)\sqrt\tau_0 \norme{\textbf{w}_1 - \textbf{w}_2}_\infty. \]
   The bound of $J_b$ falls without much fight under the fire of our preliminary inequality \eqref{eq:int_partial_G}:
    \[ 
    |J_b| \leqslant \dfrac{C_\Psi}{2\sqrt{\pi}} \sqrt{\tau_0} \norme{\textbf{w}_1-\textbf{w}_2}_\infty \,. 
    \]
Taking $\tau_0$ small enough with respect to the constants involved in the $J_a$ and $J_b$ bounds shall suffice to obtain
    \begin{equation}\label{eq:b3}
        |V_2^1 - V_2^2| \leqslant \dfrac1{8} \norme{\textbf{w}_1 - \textbf{w}_2}_\infty.
    \end{equation}
    
    Spawning now on this page with a double spatial integral, the third pair is
    \begin{align*} U_1^1 - U_1^2 = &\,\int_{0}^{+\infty} \int_{\gamma_1(x,\tau)}^{+\infty} zG(z,\tau, \xi ,0)dz u^0(x,\xi) d\xi\\ &- \int_{0}^{+\infty} \int_{\gamma_2(x,\tau)}^{+\infty} zG(z,\tau, \xi ,0)dz u^0(x,\xi) d\xi.  \end{align*}
    We first notice that by \eqref{eq:Psi_1-Psi_2},
    \[  |\gamma_1(x,\tau) - \gamma_2(x,\tau)| \leqslant \int_0^\tau |\Psi_1(x,\eta)-\Psi_2(x,\eta)|d\eta \leqslant K_\Psi \tau \norme{\bar u_1 - \bar u_2}_{L^\infty(\mathbb T^d\times[0,\tau_0])}.  \]
    Then, bearing in mind that by Lemma \ref{lm:gamma}, $\norme{\gamma_i}_{L^\infty(\mathbb T^d\times[0,\tau_0])} \leqslant m $ , $i=1,2$, we compute
    \begin{align*}
        |U_1^1 - U_1^2| &= \left| \int_{0}^{+\infty} \int_{\gamma_1(x,\tau)}^{\gamma_2(x,\tau)} zG(z,\tau, \xi ,0)dz u^0(x,\xi) d\xi \right|\\
                        & \leqslant \int_{0}^{+\infty} u^0(x,\xi) d\xi  |\gamma_1(x,\tau) - \gamma_2(x,\tau)| \sup_{(z,\xi)\in [\gamma_1(x,\tau),\gamma_2(x,\tau)]\times \R_+ }  \left| zG(z,\tau, \xi ,0) \right| \\
                        & \leqslant \norme{u^0}_{L^\infty(\mathbb T^d, L^1(\R_+))} K_\Psi \tau \norme{\bar u_1 - \bar u_2}_{L^\infty(\mathbb T^d\times[0,\tau_0])} \dfrac{m}{2\sqrt{\pi\tau}}\\
                        & \leqslant \norme{u^0}_{L^\infty(\mathbb T^d, L^1(\R_+))} K_\Psi \sqrt{\tau_0} \dfrac{m}{2\sqrt{\pi}}\norme{\textbf{w}_1 - \textbf{w}_2}_\infty.
    \end{align*}
    Hence, for $\tau_0$ small enough, we harvest
    \begin{equation}\label{eq:b4}
        |U_1^1 - U_1^2| \leqslant \dfrac14 \norme{\textbf{w}_1 - \textbf{w}_2}_\infty.
    \end{equation}
    
    Last but not least, our fourth pair comes under the very form
    \begin{align*}  U_2^1 - U_2^2 = &\,\int_{0}^{\tau} \int_{\gamma_1(x,\tau)}^{+\infty} z \dfrac{\partial G}{\partial \xi}(z,\tau,\gamma_1(x,\eta),\eta) dz v_1(x,\eta) d\eta \\ 
    &- \int_{0}^{\tau} \int_{\gamma_2(x,\tau)}^{+\infty} z \dfrac{\partial G}{\partial \xi}(z,\tau,\gamma_2(x,\eta),\eta) dz v_2(x,\eta) d\eta, 
    \end{align*}
    and it shall suffer the same fate as its predecessors, being divided into $U_2^1 - U_2^2 = K_a + K_b$, 
    with
    \begin{align*}
        K_a  &  = \int_{0}^{\tau} \left( \int_{\gamma_1(x,\tau)}^{+\infty} z  \dfrac{\partial G}{\partial \xi}(z,\tau,\gamma_1(x,\eta),\eta)dz -  \int_{\gamma_2(x,\tau)}^{+\infty} z \dfrac{\partial G}{\partial \xi}(z,\tau,\gamma_2(x,\eta),\eta)dz \right)  v_1(x,\eta) d\eta,\\
        K_b &= \int_{0}^{\tau} \int_{\gamma_2(x,\tau)}^{+\infty} z \dfrac{\partial G}{\partial \xi}(z,\tau,\gamma_2(x,\eta),\eta) dz \big(v_1(x,\eta)-v_2(x,\eta)\big) d\eta.
    \end{align*}
    Using \eqref{eq:brick2}, the fact that $\erf$ is $1$-Lipschitz and $|\gamma_2(x,\tau)|\leqslant m$, we obtain  
    \begin{align*}
        &\left| \int_{\gamma_1(x,\tau)}^{+\infty} z  \dfrac{\partial G}{\partial \xi}(z,\tau,\gamma_1(x,\eta),\eta)dz -  \int_{\gamma_2(x,\tau)}^{+\infty} z \dfrac{\partial G}{\partial \xi}(z,\tau,\gamma_2(x,\eta),\eta)dz \right|\\
        &\leqslant \, \dfrac{1}{\sqrt{4\pi(\tau-\eta)}}\left|\gamma_1(x,\tau) e^{-\frac{(\gamma_1(x,\tau)-\gamma_1(x,\eta))^2}{4(\tau-\eta)}} - \gamma_2(x,\tau) e^{-\frac{(\gamma_2(x,\tau)-\gamma_2(x,\eta))^2}{4(\tau-\eta)}}\right|   \\
              &  \qquad + \frac12 \left| \erf\left( \dfrac{\gamma_1(x,\eta)-\gamma_1(x,\tau)}{2\sqrt{\tau-\eta}}\right) - \erf\left( \dfrac{\gamma_2(x,\eta)-\gamma_2(x,\tau)}{2\sqrt{\tau-\eta}}\right) \right| \\
              &\leqslant \dfrac{1}{\sqrt{4\pi(\tau-\eta)}}\bigg(|\gamma_1(x,\tau) - \gamma_2(x,\tau) |\, e^{-\frac{(\gamma_1(x,\tau)-\gamma_1(x,\eta))^2}{4(\tau-\eta)}} \\&\qquad\qquad\qquad\qquad\qquad+ |\gamma_2(x,\tau)|\,\left| e^{-\frac{(\gamma_1(x,\tau)-\gamma_1(x,\eta))^2}{4(\tau-\eta)}} -  e^{-\frac{(\gamma_2(x,\tau)-\gamma_2(x,\eta))^2}{4(\tau-\eta)}}\right|\bigg)   \\
              &  \qquad + \frac{1}{\sqrt{\pi}}  \dfrac{\left| \gamma_1(x,\tau)-\gamma_2(x,\tau)+ \gamma_1(x,\eta)-\gamma_2(x,\eta)\right|}{2\sqrt{\tau-\eta}}\\
              &\leqslant \dfrac{1}{\sqrt{4\pi(\tau-\eta)}}\left(|\gamma_1(x,\tau) - \gamma_2(x,\tau) | + m\left| e^{-\frac{(\gamma_1(x,\tau)-\gamma_1(x,\eta))^2}{4(\tau-\eta)}} -  e^{-\frac{(\gamma_2(x,\tau)-\gamma_2(x,\eta))^2}{4(\tau-\eta)}}\right|\right)   \\
              &  \qquad + \frac{1}{\sqrt{\pi}}   \dfrac{\left| \gamma_1(x,\tau)-\gamma_2(x,\tau)| + |\gamma_1(x,\eta)-\gamma_2(x,\eta)\right|}{2\sqrt{\tau-\eta}}
    \end{align*}
    Then, applying Lemma \ref{lm:gamma} and \eqref{eq:brick3a}, we progress to
    \begin{align*}
        &\left| \int_{\gamma_1(x,\tau)}^{+\infty} z  \dfrac{\partial G}{\partial \xi}(z,\tau,\gamma_1(x,\eta),\eta)dz -  \int_{\gamma_2(x,\tau)}^{+\infty} z \dfrac{\partial G}{\partial \xi}(z,\tau,\gamma_2(x,\eta),\eta)dz \right|\\
              &\leqslant \dfrac{1}{\sqrt{4\pi(\tau-\eta)}}\left(\tau\, C_\Psi\norme{\bar u_1 - \bar u_2 }_{L^{\infty}(\mathbb T^d\times [0,\tau_0])} + m C_\lambda\, \tau_0 \norme{\textbf{w}_1 - \textbf{w}_2}_\infty\right)   \\
              &  \qquad + \frac{1}{\sqrt{\pi}}   \dfrac{\,2\,\tau_0\, C_\Psi\norme{\bar u_1 - \bar u_2 }_{L^{\infty}(\mathbb T^d\times [0,\tau_0])}}{2\sqrt{\tau-\eta}}  
              \\
              &\leqslant \dfrac{\tau_0}{2\sqrt{\pi(\tau-\eta)}}\left[ 3 C_\Psi+mC_\lambda\right]\norme{\textbf{w}_1 - \textbf{w}_2}_\infty .
    \end{align*}
    Integrating over $[0,\tau]$ the last expression, we obtain
    \[  |K_a| \leqslant \frac{1}{\sqrt\pi}\left[ 3C_\Psi + m C_\lambda\right] \tau_0^\frac32 \norme{\textbf{w}_1 - \textbf{w}_2}_\infty. \]
    
    Regarding $K_b$, note that
    \[  \int_{\gamma_2(x,\tau)}^{+\infty} z \dfrac{\partial G}{\partial \xi}(z,\tau,\gamma_2(x,\eta),\eta) dz \leqslant \dfrac{\norme{\gamma_2}_{L^\infty(\mathbb T^d \times [0,\tau_0])}}{\sqrt{4\pi(\tau-\eta)}} + 1 .  \]
    Consequently, we claim
    \[   |K_b| \leqslant \left( \frac{m}{\sqrt\pi}  + \sqrt{\tau_0} \right)\sqrt{\tau_0}  \norme{\textbf{w}_1 - \textbf{w}_2}_\infty. \]
    Imposing for the last time adequate smallness upon $\tau_0$, we get down to
    \begin{equation}\label{eq:b6}
        |U_2^1 - U_2^2| \leqslant \dfrac14 \norme{\textbf{w}_1 - \textbf{w}_2}_\infty.
    \end{equation}
    
    Collecting now the fruits of our efforts in the form of inequalities \eqref{eq:b1},  \eqref{eq:b3}, \eqref{eq:b4} and \eqref{eq:b6}, we reach the conclusion we longed after:
    \begin{equation}
        \norme{\mathcal T(\textbf{w}_1) - \mathcal T(\textbf{w}_2)}_\infty \leqslant  \frac 12 \norme{\textbf{w}_1 - \textbf{w}_2}_\infty.
    \end{equation}

{\bf Step 4. Existence of a maximal solution.-} We can then apply the Banach fixed point theorem and conclude that there exists a small time $\tau_0 > 0$ depending only on $m$, $\beta$, $W$, $\Phi$ and $u^0$, and a unique local-in-time solution $(v,\gamma,\bar u)$ to \eqref{du}--\eqref{du_hast}--\eqref{du_hast_mich} on $\mathbb T^d\times [0,\tau_0]$. By successive applications of the Banach fixed point theorem (the new $u^0$ is given by application of \eqref{eq:Duhamel}) we can extend this solution to a unique maximal solution on the time interval $[0,\tau^*)$, $\tau^*\in(0,+\infty]$. As we mentioned before this proof, each time we extend the solution on the next time interval, say from $[\tau_k,\tau_{k+1}]$ to $[\tau_{k+1},\tau_{k+2}]$, we start by changing the variable $z$ so as to obtain $\gamma(\cdot,\tau_{k+1}) \equiv 0$ in order to apply our construction. We reason by contradiction and assume that both $\tau^*< +\infty$ and
\begin{equation*}
    \sup_{\tau\in [0,\tau^*)} \sup_{x\in\mathbb T^d}\max(|v(x,\tau)|,|\gamma(x,\tau)|,|\bar u(x,\tau)|) <  +\infty \ \} = \Lambda < +\infty 
\end{equation*}
hold. Then by \eqref{eq:Duhamel}, there exists a constant $C(\Lambda,u^0)$ such that for all $\tau_1\in[0,\tau^*)$, 
$$
\norme{u(\cdot,\cdot,\tau_1)}_{L^\infty(\mathbb T^d\times \R_+)}\leqslant C(\Lambda,u^0)
$$
and for all $x\in\mathbb T^d$, $\norme{z\mapsto z u(x,z,\tau_1)}_{L^\infty(\mathbb T^d,L^1(\gamma(x,\tau_1),+\infty))}\leqslant C(\Lambda, u^0)$. Then we can do the proof above with $\tilde m = K( \Lambda,u^0)$ depending only on $\Lambda$, $\Phi$, $B$, $W$, $u^0$. The existence time $\tau_2$ for the local solution arising from $u(\cdot,\cdot,\tau_1)$ depends on $K(\Lambda,u^0)$ and not on $\tau_1$; thereby, we can choose $\tau_1=\tau^*-\varepsilon$ for $\varepsilon$ positive and small enough and the initial condition $(v(x,\tau^*-\varepsilon),\gamma(x,\tau^*-\varepsilon),\bar u(x,\tau^*-\varepsilon))$ and obtain a contradiction with the maximality of $\tau^*$.
\end{proof}

We can then exploit the Duhamel formula \eqref{eq:Duhamel} and deduce the following existence result for the solution $u$ of \eqref{eq:Stefan}.

\begin{theorem}
    Assume $\Phi$ is locally Lipschitz, $\beta\in C^0(\R_+)$ and $W\in L^1(\mathbb T^d)$. For all compatible initial conditions 
    $u^0$, there exists a unique maximal local-in-time solution
    \[ u\in C^{0}(\mathbb T^d \times  \Omega_{z,\tau} )\cap C^{0,2,1}(\mathbb T^d \times \mathring\Omega_{z,\tau} )  \]
    of \eqref{eq:Stefan}, where
    \[ \Omega_{z,\tau} = \{ (z,\tau) \ | \ \gamma(x,\tau) \leqslant z < +\infty,\ 0 \leqslant \tau < \tau^*  \}. \]
    and $\mathring\Omega_{z,\tau}$ is the interior of $\Omega_{z,\tau}$. Moreover, for all $\tau\in [0,\tau^*)$, $u(\cdot,\cdot,\tau)$ and $\partial_z u(\cdot,\cdot,\tau)$ are fast-decaying.
\end{theorem}

\begin{proof} We first apply Proposition \ref{prop:local} and obtain a maximal continuous solution $(v,\gamma,\bar u)$. Then, using the Duhamel formula \eqref{eq:Duhamel} we can recover the solution
\begin{align*}
    u(x,z,\tau) = \int_{0}^{+\infty} G(z,\tau, \xi ,0) u^0(x,\xi) d\xi + \int_{0}^{\tau} \dfrac{\partial G}{\partial \xi}(z,\tau,\gamma(x,\eta),\eta) v(x,\eta) d\eta.
\end{align*}
of \eqref{eq:Stefan}. We can now consider $v(x,\tau)$ and $\Psi(x,\tau)$ as exterior continuous functions and thus the properties of the heat kernel $G$ imply the interior smoothness.
\end{proof}

\begin{remark}\label{rm:boundary_reg}
    Note that we have also $C^1$ regularity in $z$ of the solution near the boundary: since $\Psi$ is continuous and
\[ \dfrac{\partial u}{\partial z}(x,\gamma(x,\tau),\tau) = \Psi(x,\tau) u(x,\gamma(x,\tau),\tau),  \]
then $\partial_z u$ is continuous up to the boundary. To get $C^2$ regularity near the boundary by a bootstrap argument, we would need $\Psi$, and hence $\Phi$, to be $C^1$. For smoother $\Phi$ and $\beta$, we can get more smoothness up to the boundary.
\end{remark}

We conclude this subsection with a local existence result in the original variables of \eqref{eq:PDE}-\eqref{eq:BC}.

\begin{theorem}
    Assume $\Phi$ is locally Lipschitz, $B\in C^0(\R_+)$ and $W\in L^1(\mathbb T^d)$. For all compatible initial conditions $ \rho^0$,
     \eqref{eq:PDE}-\eqref{eq:BC} admits a unique maximal local-in-time solution in the sense of Definition \ref{def:classical}.
\end{theorem}

\subsection{Global-in-time existence}

If we choose a function $\Phi$ which is globally Lipschitz and not just locally Lipschitz, then the only possible obstruction for global-in-time existence is blowup of either $v$ or $\bar u$. Note that we have the original variables correspondences
\begin{align} \rho(x,0,t) &= \dfrac{1}{\sqrt\sigma}e^{\frac{t}{\tau_c}} v\left(x,\frac12 (e^{\frac{2t}{\tau_c}}-1)\right),\label{eq:modif_s=0}\\
\bar \rho(x,t) &= \sqrt\sigma e^{-\frac{t}{\tau_c}} \left( \bar u\left(x,\frac12 (e^{\frac{2t}{\tau_c}}-1)\right) - \gamma\left(x,\frac12 (e^{\frac{2t}{\tau_c}}-1)\right) \right)\label{eq:modif_mean},  \end{align}
where $v,\gamma,\bar u$ are obtained with $\tau_c = \sigma = 1$ and the modified $\tilde W, \tilde \Phi$previously described.

We will also make good use of a weaker formulation of our problem. Consider a solution $\rho$ to \eqref{eq:PDE}; for any test function $h\in C^2(\R_+)$ with at most polynomial growth, we obtain by integration by parts and applying the boundary condition,
\begin{align}\label{eq:test_functions}
    \forall x\in \mathbb T^d, \quad \tau_c \dfrac{\partial}{\partial t} \int_0^{+\infty} h(s) \rho(x,s,t) ds =&  \int_0^{+\infty}\left[ (\Phi_{\bar \rho}(x,t) - s)\dfrac{d h}{ds}(s) + \sigma \dfrac{d^2 h}{ds^2}(s)\right] \rho(x,s,t) ds \nonumber\\ &+ \sigma \dfrac{d h}{ds}(0) \rho(x,0,t),
\end{align}
where
\[ \Phi_{\bar \rho}(x,t) = \Phi( W\ast \bar\rho (x,t) + B(t) ). \]

Let us now prove that if the value $\rho(x,0,t)$ at the boundary is controlled uniformly in time, then we also have control over the average $\bar \rho$.
\begin{lemma}\label{lm:v_to_bar_rho}
    Assume $\Phi$ is globally Lipschitz, $B\in C^0(\R_+)$ and $W\in L^1(\mathbb T^d)$. Let $\rho$ be a solution of \eqref{eq:PDE}-\eqref{eq:BC} in the sense of Definition \ref{def:classical}. Let $T,C_T>0$ such that for all $t\in[0,T]$, for all $x\in\mathbb T^d$, $ |\rho(x,0,t)| \leqslant C_T$. Then, there exists a constant $\bar C_T$ depending only on $T$, $C_T$, $\Phi$, $B$, $L^d$, $\sigma$, $W$ and  $\norme{\bar \rho^0}_\infty$ such that
    \[  \forall t\in[0,T],\ \forall x\in\mathbb T^d, \quad |\bar\rho(x,t)| \leqslant \bar C_T . \]
    If instead we assume that there exists one point $x_0\in\mathbb T^d$ such that for all $t\in[0,T]$, $ |\rho(x_0,0,t)| \leqslant C_T$ and $\Phi(x_0,t)\leqslant C_T$, then, there exists a constant $\bar C_T$ depending only on $T$, $C_T$, $L^d$, $\sigma$ and  $\norme{\bar \rho^0}_\infty$ such that
    \[  \forall t\in[0,T], \quad |\bar\rho(x_0,t)| \leqslant \bar C_T . \]
\end{lemma}

\begin{proof}
    Applying \eqref{eq:test_functions} with $h(s)=s$ and using the conservation of mass property, we obtain
    \begin{equation}\label{eq:test_s} \tau_c \dfrac{\partial \bar\rho}{\partial t}(x,t) =  - \bar\rho(x,t) + \dfrac{\Phi( W \ast \bar \rho(x,t) + B(t) )}{L^d} + \sigma\rho(x,0,t).    \end{equation}
    We denote again $C_\Phi$ the Lipschitz constant of $\Phi$. Let $T_1\in(0,T]$ such that 
    \[ T_1 < \dfrac{L^d}{L^d+C_\Phi \norme{W}_{L^1}}. \]
    For all $t\in[0,T_1]$, we integrate \eqref{eq:test_s} and use the available bounds to obtain
    \begin{align*} \norme{\bar \rho(\cdot,t)}_{L^\infty(\mathbb T^d)} \leqslant\,& \norme{\bar \rho^0}_{L^\infty(\mathbb T^d)} + \dfrac{1}{\tau_c} \int_0^t \norme{\bar \rho(\cdot,t')}_{L^\infty(\mathbb T^d)} dt' \\& + \dfrac{1}{\tau_c L^d} C_\Phi \int_0^t \left( \norme{W}_{L^1}\norme{\bar \rho(\cdot,t')}_{L^\infty(\mathbb T^d)} + \norme{B}_{L^\infty([0,T])}  \right)dt'  + \dfrac{1}{\tau_c}\sigma C_T T_1. \end{align*}
    Taking the supremum on $[0,T_1]$ on both sides and under the integrals, we obtain
    \begin{align*} \norme{\bar \rho}_{L^\infty(\mathbb T^d\times [0,T_1])} \leqslant\,& \norme{\bar \rho^0}_{L^\infty(\mathbb T^d)} + \frac{T_1}{\tau_c}  \norme{\bar \rho}_{L^\infty(\mathbb T^d\times [0,T_1])} \\& + \dfrac{1}{\tau_c L^d} C_\Phi T_1 \left( \norme{W}_{L^1} \norme{\bar \rho}_{L^\infty(\mathbb T^d\times [0,T_1])} + \norme{B}_{L^\infty([0,T])}  \right)  + \frac{\sigma C_T T_1}{\tau_c} . \end{align*}
    Hence, we have
    \[  \norme{\bar \rho}_{L^\infty(\mathbb T^d\times [0,T_1])} \leqslant \dfrac{  \norme{\bar \rho^0}_{L^\infty(\mathbb T^d)}  + \dfrac{ C_\Phi T_1}{\tau_c L^d}\norme{B}_{L^\infty([0,T])} +  \frac{\sigma C_T T_1}{\tau_c}   }{ 1 - \frac{T_1}{\tau_c}\left(1 + \frac{C_\Phi \norme{W}_{L^1}}{L^d} \right)   }.  \]
    Let $N$ the smallest integer such that $N T_1 \geqslant T$; we can deduce by iteration and majorations that
    \[  \norme{\bar \rho}_{L^\infty(\mathbb T^d\times [0,T])} \leqslant \dfrac{  \norme{\bar \rho^0}_{L^\infty(\mathbb T^d)}  + N\dfrac{C_\Phi T_1}{\tau_c L^d}\norme{B}_{L^\infty([0,T])} + \frac {N}{\tau_c} \sigma C_T T_1   }{ \left(1 - \frac{T_1}{\tau_c}\left(1 + \frac{C_\Phi \norme{W}_{L^1}}{L^d} \right) \right)^N  }  =: \bar C_T.    \]
    In the second case, we simply consider $\Phi( W \ast \bar \rho(x,t) + B(t) )$ as an external function and write for the specific point $x$,
    \[ \bar \rho (x,t) = \bar\rho(x,0)e^{-\frac{t}{\tau_c}} + \frac{1}{\tau_c}\int_{0}^t  e^{\frac{t'-t}{\tau_c}} \left( \dfrac{\Phi( W \ast \bar \rho(x,t') + B(t') )}{L^d} + \sigma\rho(x,0,t') \right)dt', \]
    which leads to
    \[\bar \rho (x,t) \leqslant \bar\rho(x,0)e^{-\frac{t}{\tau_c}} + C_T\dfrac{1+\sigma L^d}{L^d}(1- e^{-\frac{t}{\tau_c}}) \leqslant  \norme{\bar \rho^0}_{L^\infty(\mathbb T^d)} + C_T\dfrac{1+\sigma L^d}{L^d}. \]
\end{proof}
Then, we obtain the opposite control of the boundary value from bounds over the average.
\begin{lemma}\label{lm:bar_rho_to_v}
    Assume $\Phi$ is globally Lipschitz, $B\in C^0(\R_+)$ and $W\in L^1(\mathbb T^d)$. Let $\rho$ be a solution of \eqref{eq:PDE}-\eqref{eq:BC} in the sense of Definition \ref{def:classical}. Let $T,\bar C_T>0$ such that for all $t\in[0,T]$, for all $x\in\mathbb T^d$, $ |\bar\rho(x,t)| \leqslant \bar C_T $. Then, there exists a constant $C_T$ depending only on $T$, $\bar C_T$, $\Phi$, $B$, $\sigma$, $W$ and $\bar \rho^0$ such that
    \[  \forall t\in[0,T], \forall x\in\mathbb T^d, \quad |\rho(x,0,t)| \leqslant C_T. \]
\end{lemma}

\begin{proof}
    We proceed in modified variables by looking at the equation satisfied by $v$:
    \begin{align}\label{eq:v}
    v(x,\tau) = 2\int_{0}^{+\infty} G(\gamma(x,\tau),\tau, \xi ,0) u^0(x,\xi) d\xi +2\int_{0}^{\tau}\dfrac{\partial G}{\partial \xi}(\gamma(x,\tau),\tau,\gamma(x,\eta),\eta) v(x,\eta) d\eta,
    \end{align}
    where
    \[ \gamma(x,\tau) = -\int_0^\tau\Psi(x,\eta)d\eta,\quad \Psi(x,\tau) = \frac1{\sqrt\sigma}\Phi\left( \alpha(\tau) \frac1{\sqrt\sigma}W \ast [ \bar u (x,\tau) - \gamma(x,\tau)] + \beta(\tau)\right)\alpha(\tau).  \]
    By hypothesis and Lemma \ref{lm:gamma}, there exists a constant $C_1$ such that for all $\tau\in[0,\tau^*]$, for all $x\in\mathbb T^d$, $\Psi(x,\tau)\leqslant C_1$.
    Recalling the estimate \eqref{eq:int_partial_G}, there exists a constant $C_2$ such that
    \begin{equation*}
        \displaystyle\int_0^\tau \left| \dfrac{\partial G}{\partial \xi}(\gamma(x,\tau),\tau,\gamma(x,\eta),\eta)\right| d\eta \leqslant C_2 \sqrt{\tau}.
    \end{equation*}
    Consider $\tau^* = \frac12 (e^{\frac{2t}{\tau_c}}-1)$ and let $\tau_1 \leqslant \tau^*$ such that $ 2C_2\sqrt{\tau_1} < 1$. Then, for all $\tau\in[0,\tau_1]$, for all $x\in\mathbb T^d$, we have
    \[ v(x,\tau) \leqslant 2\norme{u^0}_{L^\infty(\mathbb T^d \times \R_+)} + 2C_2\sqrt{\tau_1} \norme{v}_{L^\infty(\mathbb T^d \times [0,\tau_1])}.     \]
    Taking the supremum on $x,\tau$ in the left-hand side of the previous inequality yields
    \[  \norme{v}_{L^\infty(\mathbb T^d \times [0,\tau_1])} \leqslant \dfrac{2\norme{u^0}_{L^\infty(\mathbb T^d \times \R_+)}}{1-2C_2\sqrt{\tau_1} }. \]
    Let $N$ be the smallest integer such that $N \tau_1 \geqslant \tau^*$; we deduce by iteration and majorations
    \[  \norme{v}_{L^\infty(\mathbb T^d \times [0,\tau^*])} \leqslant \dfrac{2\norme{u^0}_{L^\infty(\mathbb T^d \times \R_+)}}{\left(1-2C_2\sqrt{\tau_1}\right)^N } =: C_3. \]
    We eventually apply \eqref{eq:modif_s=0}:
    \[ \rho(x,0,t) = \dfrac{1}{\sqrt\sigma}e^{\frac{t}{\tau_c}} v\left(x,\frac12 (e^{\frac{2t}{\tau_c}}-1)\right) \leqslant \dfrac{1}{\sqrt\sigma}e^{\frac{T}{\tau_c}} C_3 := C_T. \]
\end{proof}

We can now state and prove the following global existence result.

\begin{theorem}\label{thm:global}
    Assume $\Phi$ is globally Lipschitz, $B\in C^0(\R_+)$, $W\in L^1(\mathbb T^d)$, and that $\rho^0$ is a compatible initial condition. Then,
    \eqref{eq:PDE}-\eqref{eq:BC} admits a unique global-in-time solution $\rho$ in the sense of Definition \ref{def:classical}.
\end{theorem}

\begin{proof} Consider a local solution $\rho$ over the maximal time interval $[0,T^*)$, $T^*\in (0,+\infty]$. If we consider $\Phi_{\bar \rho}(x,t)$ as an external input, the maximum principle for the linear Fokker-Planck equation implies that $\rho$ is non-negative. As a consequence, $\bar \rho(x,t)$ and $\rho(x,0,t)$ must remain non-negative as well. In modified variables, due to the equivalence formulae above, $v(x,\tau)$ is also positive but not necessarily $\bar u(x,\tau)$. Let us denote by $C_\Phi$ the Lipschitz coefficient of $\Phi$.\\

Assume for the sake of contradiction that $T^*$ is finite. Note that we can always choose $T^*$ as small as we want without loss of generality by changing the initial time to $T^*-\varepsilon$ and taking the new initial condition $\rho(\cdot,T^*-\varepsilon)$, for some small $\varepsilon$; then this solution blows-up at time $\varepsilon$. We proceed with tge proof in the following steps.
\begin{enumerate}
    \item Assume the average $\bar\rho$ and the boundary value $\rho(\cdot,0,\cdot)$ are uniformly bounded on $[0,T^*)$ and conclude that the free boundary position is then uniformly bounded.
    \item Prove a uniform bound on $[0,T^*)$ for the average $\bar\rho$.
    \item Use Lemma \ref{lm:bar_rho_to_v} and step 1 to get bounds on all the quantities involved in condition \eqref{eq:time_max} for the maximal time of existence.
\end{enumerate}

    {\bf Step 1. Conditional bound for the free boundary.-} Assume $\bar \rho$ and $\rho(\cdot,0,\cdot)$ (respectively $\bar u$ and $v$) are uniformly bounded on $[0,T^*)$ (respectively $[0,\tau^*)$, $\tau^* = \frac12 (e^{2\frac{T^*}{\tau_c}}-1)$). For all $t\in [0, \tau^*)$,
    \begin{align*} |\gamma(x,\tau)| &= \left| \gamma^0(x) - \int_{0}^{\tau} \frac{1}{\sqrt\sigma}\Phi\left( \alpha(\eta) \frac{1}{\sqrt\sigma}W \ast [ \bar u (x,\eta) - \gamma(x,\eta)] + \beta(\eta)\right)\alpha(\eta) d\eta \right|\\
    & \leqslant \norme{\gamma^0}_\infty + \frac{ T^* C_\Phi}{\sqrt\sigma}  \left( \frac{\norme{W}_{L^1}}{\sqrt\sigma}\left( \norme{\bar u}_{L^\infty(\mathbb T^d\times [0,\tau^*))} + \norme{\gamma(\cdot,t)}_{L^\infty(\mathbb T^d)}\right) + \norme{\beta}_{L^\infty([0,\tau^*])}  \right).
    \end{align*}
    Rearranging, taking the supremum on the left-hand side and choosing $T^*$ small enough without loss of generality (as explained above), we obtain the uniform bound
    \[  \norme{\gamma}_{L^\infty(\mathbb T^d\times [0,\tau^*))} \leqslant \dfrac{ \sqrt\sigma\norme{\gamma^0}_{L^\infty(\mathbb T^d)} + T^* C_\Phi \left( \frac{\norme{W}_{L^1}}{\sqrt\sigma}\norme{\bar u}_{L^\infty(\mathbb T^d\times [0,\tau^*))} + \norme{\beta}_{L^\infty([0,\tau^*])}  \right) }{\sqrt\sigma - T^* C_\Phi \norme{W}_{L^1} }.  \]

    {\bf Step 2. Bound for the average $\bar\rho$.-}  Consider a test function $h\in C^2([0,+\infty))$ satisfying
    \begin{itemize}
        \item $h$ is increasing on $[0,+\infty)$;
        \item $\forall s\in[1,+\infty)$, $h(s)=s$;
        \item $h(0) = \dfrac{dh}{ds}(0)=0$.
    \end{itemize}
    Hence, there exist constants $C_h^1,C_h^2>0$ such that
    \[  \norme{ \dfrac{dh}{ds}(s) }_{L^\infty(\R_+)} \leqslant C_h^1,\qquad \mathrm{and} \qquad \norme{ \dfrac{d^2h}{ds^2}(s) }_{L^\infty(\R_+)} \leqslant C_h^2. \]
    Applying \eqref{eq:test_functions} to this function $h$, we obtain
    \begin{multline}
    \forall x\in \mathbb T^d, \quad  \int_0^{+\infty} h(s) \rho(x,s,t) ds =\int_0^{+\infty} h(s) \rho^0(x,s) ds \\  \frac{1}{\tau_c} \int_0^t\int_0^{+\infty}\left[ (\Phi_{\bar \rho}(x,t') - s)\dfrac{d h}{ds}(s) + \sigma \dfrac{d^2 h}{ds^2}(s)\right] \rho(x,s,t') dsdt',
    \end{multline}
    and applying the bounds on the derivative of $h$, it give us for all $x\in\mathbb T^d$,
    \begin{align*}
    & \int_0^{+\infty} h(s) \rho(x,s,t) ds \leqslant   \int_0^{+\infty} h(s) \rho^0(x,s) ds + \frac{1}{\tau_c} \int_0^t \left( \dfrac{C_h^1}{L^d}|\Phi_{\bar \rho}(x,t')|  + \dfrac{\sigma C_h^2}{L^d} \right)dt'\\
     \leqslant & \int_0^{+\infty} h(s) \rho^0(x,s) ds + \frac{1}{\tau_c} \int_0^t \left( \dfrac{C_h^1 C_\Phi}{L^d}(\norme{W}_{L^1} \norme{\bar\rho(\cdot,t)}_{L^\infty(\mathbb T^d)} + \norme{B}_{L^\infty([0,T^*))}  + \dfrac{\sigma C_h^2}{L^d}\right)dt'\\
     \leqslant & \int_0^{+\infty} h(s) \rho^0(x,s) ds + \dfrac{T^*}{\tau_c L^d}\left( C_h^1 C_\Phi\left(\norme{W}_{L^1} \norme{\bar\rho}_{L^\infty(\mathbb T^d\times[0,t))} + \norme{B}_{L^\infty([0,T^*))}\right)  + \sigma C_h^2\right).
    \end{align*}
    Now, we notice that
    \[ \bar\rho(x,t) =  \int_{0}^1 s\rho(x,s,t) ds + \int_{1}^{+\infty}h(s)\rho(x,s,t) ds \leqslant  \dfrac{1}{L^d} + \int_{0}^{+\infty}h(s)\rho(x,s,t) ds.   \]
    Therefore, we have
    \begin{align*} \bar\rho(x,t) \leqslant \,&\dfrac{1}{L^d} + \int_0^{+\infty} h(s) \rho^0(x,s) ds\\ &+ \dfrac{ T^*}{\tau_c L^d}\left( C_h^1 C_\Phi\left(\norme{W}_{L^1} \norme{\bar\rho}_{L^\infty(\mathbb T^d\times[0,t))} + \norme{B}_{L^\infty([0,T^*))}\right)  + \sigma C_h^2\right).   \end{align*}
    Assuming $T^*$ is small enough and taking the supremum on $\mathbb T^d\times [0,t]$ in the left-hand side, we eventually get
    \[  \norme{\bar\rho}_{L^\infty(\mathbb T^d\times[0,t))}  \leqslant \dfrac{\tau_c + \tau_c L^d\int_0^{+\infty} h(s) \rho^0(x,s) ds + \sigma C_h^2 T^* +  C_h^1 C_\Phi T^* \norme{B}_{L^\infty([0,T^*))}  }{\tau_c L^d- T^* C_h^1 C_\Phi \norme{W}_{L^1}}. \]
    Letting $t$ tend to $T^*$, we obtain a uniform bound for $\bar \rho$: $\forall (x,t)\in \mathbb T^d\times [0,T^*)$, $\bar \rho \leqslant C_{T^*}$. 
    
    {\bf Step 3. Reaching the contradiction.-} Since we have a bound for the average, we can apply Lemma \ref{lm:bar_rho_to_v} and obtain
    a uniform bound for the boundary term $\rho(x,0,t)$ on $\mathbb T^d\times[0,T^*)$. Step 1 then also provides a uniform bound for $\gamma$ on $\mathbb T^d\times[0,T^*)$. By condition \eqref{eq:time_max}, these three uniform bounds constitute a contradiction with the finiteness of $T^*$.
\end{proof}

\begin{remark}
    It can be checked by looking at \eqref{eq:Duhamel} in modified variables that this global-in-time solution propagates many properties of the initial condition for all times. Like previously noted in Remark \ref{rm:boundary_reg}, if the functions $\Phi$ and $B$ are $C^\infty$, then the global solution $\rho$ becomes smooth up to the boundary for any strictly positive time.
\end{remark}

\subsection{Uniform bounds when $\Phi$ is of constant sign}

When we know the sign of the function $\Phi$, we can establish more precise bounds on the two main quantities $\bar\rho(x,t)$ and $\rho(x,0,t)$. The case $\Phi\geqslant 0$ is quite natural in computational neuroscience.

\begin{lemma}\label{lm:uniform_bound_rho(0)}
    Assume $\Phi$ is globally Lipschitz, $B\in C^0(\R_+)$, $W\in L^1(\mathbb T^d)$, and $\rho^0$ is a compatible initial condition. If $\Phi$ is non-negative, then
    \begin{equation}\label{eq:uniform_bound_rho(0)}
    \forall t\in\R_+,\ \forall x\in\mathbb T^d, \quad \rho(x,0,t) \leqslant \dfrac{1}{L^d}\sqrt{ \dfrac{2}{ \pi \sigma } } \dfrac1{\sqrt{1-e^{-\frac{2t}{\tau_c}}}}. 
    \end{equation} 
    For all $T\in\R_+^*$, there exists $\bar C_T$ such that
    \[ \forall t\in[0,T],\ \forall x\in\mathbb T^d, \quad \bar\rho(x,t) \leqslant \bar C_T.  \]
\end{lemma}

\begin{proof} 
    Let us proceed in modified variables and prove a bound on $v(x,t)$. Note that the non-negativity of $\Phi$ implies non-negativity of $\Psi$. Then, building on the fact that
    \begin{align*} \dfrac{\partial G}{\partial \xi}(\gamma(x,\tau),\tau,\gamma(x,\eta),\eta) =  \dfrac{ \gamma(x,\tau) - \gamma(x,\eta) }{4\sqrt{\pi}(\tau-\eta)^\frac32} e^{-\frac{(\gamma(x,\tau)-\gamma(x,\eta))^2}{4(\tau-\eta)}} 
    = \dfrac{ - \int_\eta^\tau \Psi(x,\zeta)d\zeta }{4\sqrt{\pi}(\tau-\eta)^\frac32} e^{-\frac{(\gamma(x,\tau)-\gamma(x,\eta))^2}{4(\tau-\eta)}},
    \end{align*}
    we rewrite the expression \eqref{du} for $v$ in order to make signs salient:
    \begin{align}
    v(x,\tau) =\,\,& 2 \int_{0}^{+\infty} G(\gamma(x,\tau),\tau, \xi ,0) u^0(x,\xi) d\xi\nonumber\\& - 2\int_{0}^{\tau}  \dfrac{  \int_\eta^\tau \Psi(x,\zeta)d\zeta }{4\sqrt{\pi}(\tau-\eta)^\frac32} e^{-\frac{(\gamma(x,\tau)-\gamma(x,\eta))^2}{4(\tau-\eta)}}v(x,\eta) d\eta.
    \end{align}
    All the factors under the second integral being non-negative, and $u^0$ having unit mass, we claim
    \[  0 \leqslant v(x,\tau) \leqslant  2\int_{0}^{+\infty} G(\gamma(x,\tau),\tau, \xi ,0) u^0(x,\xi) d\xi \leqslant \dfrac{1}{ L^d \sqrt{\pi\tau}}.  \]
    In original variables (see formula \eqref{eq:modif_s=0}), this translates into the uniform bound
    \[ 0\leqslant \rho(x,0,t) \leqslant \dfrac{e^{\frac{t}{\tau_c}}}{ \sqrt\sigma L^d \sqrt{2\pi \left(e^{\frac{2t}{\tau_c}} - 1 \right)}} = \dfrac{1}{L^d}\sqrt{ \dfrac{2}{ \pi \sigma } } \dfrac1{\sqrt{1-e^{-\frac{2t}{\tau_c}}}}. \]
    By Lemma \ref{lm:v_to_bar_rho}, there is also a uniform bound over the average $\bar \rho(x,t)$.
    \end{proof}
    
    \begin{remark}
        Note that the bound over the boundary term is quantitative and universal, \textit{i.e.} independent of the initial condition $\rho^0$, which is made possible by the fact that the upper bound tends to $+\infty$ exponentially fast when time goes to $0$. However, the bound $\bar C_T$ over the average can tend to $+\infty$ when $T$ tends to $+\infty$, as can be seen through the special case $\Phi(z)=z$, $W(x)\equiv 2$, $\tau_c=\sigma=1$, $B(t)\equiv 0$, $L=1$, $\rho^0$ constant in $x$:
        \begin{equation*} \dfrac{\partial \bar\rho}{\partial t}(t) =  \bar \rho(t) + \rho(0,t),    \end{equation*}
        which admits the exponentially increasing subsolution $\bar\rho^0 e^t$.
    \end{remark}
    
    \begin{remark}[Optimality of the bound \eqref{eq:uniform_bound_rho(0)}]\label{rem:optimality}
        The bound is saturated in the high noise regime under the application-relevant assumptions of \cite{CRS}. Assume $\Phi$ is an increasing function with $\Phi(x) = 0$ for $x\leq 0$, the external input is constant and positive ($B(t) = B > 0$), and the connectivity kernel is average inhibitory:
        \[\int_{\mathbb T^d} W(x)dx = W_0 < 0.\]
        Then, by \cite[Lemma 3.1]{CRS}, if
        \[ \sigma > \dfrac{L^{2d }\pi B^2}{2 |W_0|^2},  \]
        there exists a unique $x$-independent stationary state $\rho_\infty$ of the form
        \[  \rho_\infty(s) = \dfrac{1}{L^d} \sqrt{\dfrac{2}{\pi \sigma}}\, e^{-\frac{s^2}{2\sigma}}.  \]
        We readily see that
        \[ \rho_\infty(0) = \dfrac{1}{L^d} \sqrt{\dfrac{2}{\pi \sigma}},\]
        which corresponds exactly to the bound \eqref{eq:uniform_bound_rho(0)} in the limit $t\to+\infty$.
    \end{remark}

    \begin{lemma}
        Assume $\Phi$ is globally Lipschitz, $B\in C^0(\R_+)$, $W\in L^1(\mathbb T^d)$ and $\rho^0$ is a compatible initial condition. If $\Phi$ is non-positive, then
        \[ \forall t\in \R_+,\ \forall x\in \mathbb T^d, \quad  \bar\rho(x,t) \leqslant  \dfrac{1}{L^d}(1+\sigma - \sigma e^{-\frac{2t}{\tau_c}}) +  e^{-\frac{2t}{\tau_c}} \sup_{x\in\mathbb T^d}\int_0^{+\infty} s^2\rho^0(x,s)ds . \]
        For all $T\in\R_+^*$, there exists $C_T$ such that
        \[ \forall t\in[0,T],\ \forall x\in\mathbb T^d, \quad \rho(x,0,t) \leqslant C_T.  \]
    \end{lemma}
    
    \begin{proof}
    Let $x\in \mathbb T^d$. Applying \eqref{eq:test_functions} to $h(s)=s^2$ yields
    \[  \tau_c\dfrac{\partial}{\partial t} \int_{0}^{+\infty} s^2 \rho(x,s,t) ds = 2 \Phi_{\bar\rho}(x,t) \bar \rho(x,t) - 2 \int_{0}^{+\infty} s^2 \rho(x,s,t) ds  + \dfrac{2\sigma}{L^d}.      \]
    Denoting
    \[ M_2(x,t) =  \int_{0}^{+\infty} s^2 \rho(x,s,t) ds   \]
    and using non-negativity of $\bar \rho$ and non-positivity of $\Phi$, we come to
    \[  \dfrac{\partial }{\partial t} \left( M_2(x,t) - \dfrac{1}{L^d} \right) \leqslant
    -  \frac{2}{\tau_c} \left( M_2(x,t) - \dfrac{\sigma}{L^d} \right).
    \]
    Hence, by Grönwall's lemma applied to each point $x\in\mathbb T^d$ separately, we have
    \[ M_2(x,t) \leqslant \dfrac{\sigma}{L^d}(1-e^{-\frac{2t}{\tau_c}}) +  M_2(x,0)e^{-\frac{2t}{\tau_c}},  \]
    which translates into the uniform bound
    \[ \forall x\in \mathbb T^d, \forall t\in \R_+, \quad M_2(x,t) \leqslant \dfrac{\sigma}{L^d}(1-e^{-\frac{2t}{\tau_c}}) +  \norme{M_2(\cdot,0)}_{L^\infty(\mathbb T^d)} e^{-\frac{2t}{\tau_c}} .  \]
    We eventually notice that
    \begin{align*} \bar\rho(x,t) & = \int_{0}^1 s \rho(x,s,t) ds + \int_1^{+\infty} s \rho(x,s,t) ds\\
    &\leqslant \int_{0}^1 \rho(x,s,t) ds + \int_1^{+\infty} s^2 \rho(x,s,t) ds \\
    & \leqslant \int_{0}^{+\infty} \rho(x,s,t) ds + \int_0^{+\infty} s^2 \rho(x,s,t) ds \\
    & \leqslant \dfrac{1}{L^d} + \dfrac{\sigma}{L^d}(1-e^{-2 \frac{t}{\tau_c}}) +  \norme{M_2(\cdot,0)}_{L^\infty(\mathbb T^d)} e^{-\frac{2t}{\tau_c}} .
    \end{align*}
    Thanks to Lemma \ref{lm:bar_rho_to_v}, we can lift this result to uniform bounds over the boundary term $\rho(x,0,t)$ on all compact time intervals.\\
\end{proof}

\begin{remark}
    The uniform estimate on $\bar\rho$ is meaningful only when the second moment is finite for all $x\in\mathbb T^d$, which is enforced by the fast-decay in Definitions \ref{def:compatible} and \ref{def:classical}. For a solution which would not be fast-decaying, but only have linear decay, finite first moment, and linear derivative decay, a less precise uniform estimate can be obtained by applying \eqref{eq:test_functions} to the function $h$ we used in the proof of Theorem \ref{thm:global}.
\end{remark}

\section{Stability of stationary states}\label{sec:3}
We study the nonlinear stability of the stationary states of \eqref{eq:PDE}-\eqref{eq:BC} under the assumption of constant input $B$. 

By setting the left hand side of \eqref{eq:PDE} to zero and integrating with respect to $s$, and using the no-flux boundary condition \eqref{eq:BC}, the stationary states satisfy
\[   \sigma \partial_s \rho(x,s) =   - \left(s - \Phi\left( W \ast \bar \rho (x) + B \right) \right) \rho(x,s).  \]
Thus, the stationary states of \eqref{eq:PDE} must solve
\begin{equation}\label{eq:stationarystateequ}
\rho(x,s) = \dfrac{1}{Z_\rho} \exp\left( - \dfrac{ \big( s - \Phi\left(  W\ast \bar \rho (x) + B \big) \right)^2}{2\sigma} \right), 
\end{equation}
where $Z_\rho$ is the normalisation factor,
\begin{align*}
    Z_\rho = L^d \int_{0}^{+\infty} \exp\left( - \dfrac{ \big( s - \Phi\left(  W\ast \bar \rho (x) + B \big) \right)^2}{2\sigma} \right) \diff s. 
\end{align*}
A full characterisation of the stationary solutions of \eqref{eq:PDE}-\eqref{eq:BC} satisfying \eqref{eq:stationarystateequ} does not exist. It was shown in \cite[Prop. 3.1]{CHS} that when $W,B$, and $\Phi$ satisfy: $\Phi$ is an increasing function with $\Phi(x) = 0$ for $x\leq 0$, the external input is constant and positive ($B(t) = B > 0$), and the connectivity kernel is average inhibitory, that is
        \[\int_{\mathbb T^d} W(x)dx = W_0 < 0,\]
then \eqref{eq:PDE}-\eqref{eq:BC} admits a unique spatially homogeneous stationary state for any $\sigma > 0$.

Before embarking on the stability of stationary states of \eqref{eq:PDE}-\eqref{eq:BC}, as a warm-up we investigate in what sense the noiseless ($\sigma=0$) dynamics are relaxed as time increases.

\subsection{A Lyapunov functional for the noiseless case}

Assume that $\Phi \in C^1(\R)$ is a positive and strictly increasing function. Following \cite{KA} and others, we define the functional
\begin{equation}
    \mathcal E : \zeta \to \mathcal E[\zeta] = - \frac 1{2L^{d}}\int_{\mathbb T^d} \int_{\mathbb T^d} W(x-y) \zeta(y) \zeta(x) dx dy + \int_{\mathbb T^d} \int_{0}^{\zeta(x)} \left( \Phi^{-1}(\omega) - B \right) d\omega dx.
\end{equation}
on $C^0(\mathbb T^d)$. Given a solution $\bar \rho$ to the problem
\begin{equation}
    \dfrac{\partial \bar\rho}{\partial t}(x,t) = -\bar \rho(x,t) + \dfrac{\Phi( W \ast \bar\rho(x,t) + B)}{L^d},
\end{equation}
and denoting
\[ \Phi_{\bar\rho}(x,t) = \Phi( W\ast\bar\rho + B), \qquad \Phi_{\bar\rho}' (x,t) = \Phi'( W\ast\bar\rho + B),\]
we compute
\begin{align*}
   \dfrac{\partial \mathcal E[\Phi_{\bar\rho}]}{\partial t}(t)
    = &\, -\frac1{L^{d}}\int_{\mathbb T^d} \int_{\mathbb T^d} W(x-y) \Phi_{\bar\rho}(y) \dfrac{\partial \Phi_{\bar\rho}}{\partial t}(x,t) dx dy \\ &  + \int_{\mathbb T^d}  W\ast \bar\rho(x)\dfrac{\partial \Phi_{\bar\rho}}{\partial t}(x,t) dx\\
    = &\,  -\int_{\mathbb T^d} \dfrac{\partial \Phi_{\bar\rho}}{\partial t}(x,t) \left( -W\ast\bar\rho(x,t) + W\ast \dfrac{\Phi_{\bar\rho}(x,t)}{L^d} \right)dx \\
    = &\,  -\int_{\mathbb T^d} \dfrac{\partial \Phi_{\bar\rho}}{\partial t}(x,t) W \ast \dfrac{\partial \bar\rho}{\partial t}(x,t)dx \\
    = &\,  -\int_{\mathbb T^d}  \Phi'_{\bar\rho}(x,t) \left(  W \ast \dfrac{\partial \bar\rho}{\partial t}(x,t) \right)^2 dx \ \leq  \ 0 . 
\end{align*}
If $\Phi_{\bar\rho}$ is uniformly bounded in time, so is $ \mathcal E[\Phi_{\bar\rho}] $ and then, owing to the compactness of $\mathbb T^d$ the function $\Phi_{\bar\rho}$ converges in time towards a function $\Phi_0$ which is a local minimiser of $\mathcal E$. Hence, without noise, the (now first order) PDE \eqref{eq:PDE} supplemented with homogeneous Neumann boundary condition has a firing rate $\Phi(W\ast\bar \rho + B)$ that converges towards a constant-in-time $\Phi_0$, and hence this $\Phi_0$ must be such that $\Phi_0 = \Phi(W\ast\bar \rho_\infty + B)$ for some stationary state $\bar\rho_\infty$. In summary, when $\sigma = 0$, the nonlinear PDE \eqref{eq:PDE} completely reduces, through the firing rate $\Phi_{\bar\rho}(x,t)$, to the dynamics of a classical deterministic Wilson--Cowan model \cite{WC1,WC2}, which is already well understood. 

\subsection{Asymptotic stability of the stationary states for the system with noise}
For simplicity, we assume that $\tau_c=L=1$ in the following proofs. The results can be extended to any $\tau_c, L >0$ by an appropriate change of variables. Let $\rhostat$ be a stationary state of \eqref{eq:PDE}-\eqref{eq:BC} and define
 \[ \Phi_0(x) = \Phi( W\ast\bar\rho_\infty(x) + B), \qquad \Phi_0' (x) = \Phi'( W\ast\bar\rho_\infty (x) + B),\]
and $h=\frac{\rho}{\rhostat}$. Let $G(h)$ be a convex function of $h$. Then the relative entropy at position $x$ with respect to $G$, 
$$\int_0^{+\infty} G(h(x)) \rhostat(x) ds,$$ satisfies
\begin{align*}
\frac{d}{dt}\int_0^{+\infty} G(h) \rhostat ds & = \int_0^{+\infty} G'(h) \left( -\partial_s\left(\Phi_{\bar{\rho}}-s\right)\rho +\sigma\partial_{ss}\rho\right)   ds \\
& = - \int_0^{+\infty} G''(h) \partial_s h  \left( -\left(\Phi_{\bar{\rho}}-s\right)\rho +\sigma\partial_{s}\rho\right)  ds \\
& = - \int_0^{+\infty} G''(h) \partial_s h  \left( -\left(\Phi_{\bar{\rho}}-s\right)h +\frac{\sigma}{\rhostat}\partial_{s}\rho\right) \rhostat  ds \\
& = - \int_0^{+\infty} G''(h) \left( \left(\Phi_{0}-\Phi_{\bar{\rho}}\right) \frac{\partial h}{\partial s}  h +\sigma| \partial_s h |^2\right) \rhostat  ds,
\end{align*}
by an integration by parts and then using the relation $\frac{\partial_s \rho}{\rhostat} =  \partial_s h  + \frac1{\sigma} (\Phi_0-s) h$. 

Let us define the variance of the probability measure $L^d  \rho_\infty(x,\cdot)$ at each $x\in {\mathbb T^d}$ as   
\begin{align}\label{eq:minfty}
    M_\infty(x) = \int_0^{+\infty} (s-  L^d  \bar \rho_\infty(x))^2\rhostat(x,s) ds,
\end{align}
where $L$ is set to 1 in the proofs.
Letting $G(h) = \hf (h-1)^2$ and utilising the Poincar\'e inequality with respect to the measure $\rhostat$,
\begin{align}\label{eq:poincare}
\gamma(\rhostat) \int_0^{+\infty} |g - \bar g |^2\rhostat  ds \leq \int_0^{+\infty} |\partial_s g|^2\rhostat  ds,
\end{align}
where $\bar g = \int g(s) \rhostat(s) ds,$ see \cite{Mu72,RBI17}, we can prove the following proposition.

\begin{proposition}\label{prop:stab1}
Let $\Phi$ be Lipschitz, $W \in L^2(\mathbb T^d)$, and $\rhostat$ be a stationary state of \eqref{eq:PDE}-\eqref{eq:BC} given a $\sigma$. Assume that $\sigma$ satisfies 
\begin{align}\label{eq:poincareestimate}
\|\Phi' \|_\infty \frac{\|W\|_{L^2({\mathbb T^d})}}{L^\frac d2} \sup_{x \in {\mathbb T^d}} M_\infty^\hf (x) < \frac{\sigma}{2} \tilde\gamma (\rhostat )^\hf,
\end{align}
where $\tilde\gamma (\rhostat ) = \inf_{x \in \mathbb T^d} \gamma(\rhostat (x))$, and $\gamma(\rhostat(x))$ is the Poincar\'e constant of \eqref{eq:poincare} for $\rhostat (x)$, and $M_\infty$ is defined in \eqref{eq:minfty}. Then, the $L^1$ norm of the relative entropy 
\begin{align*}
 \int_{\mathbb T^d} \int_{0}^{+\infty} \left( \dfrac{ \rho(x,s,t) - \rho_\infty(x,s) }{ \rho_\infty(x,s) } \right)^2 \rho_\infty(x,s) ds dx,
\end{align*} 
decays exponentially fast whenever the compatible initial data $\rho_0$ is close enough to the stationary state $\rhostat$ in relative entropy.
\end{proposition}
\begin{remark}
Notice that there exist stationary states of \eqref{eq:PDE}-\eqref{eq:BC} satisfying condition \eqref{eq:poincareestimate}. Indeed, the bound $\gamma(\rhostat (x)) \geq \frac{1}{\sigma}$ can be generically ensured due to sharp Poincaré constants, see for example \cite[Thm. 3]{RBI17} or the book \cite{BGL2014}. As a consequence, stationary states satisfying     
\begin{align}\label{eq:stabcond2}
\|\Phi' \|_\infty \frac{\|W\|_{L^2({\mathbb T^d})}}{L^\frac d2} \sup_{x \in {\mathbb T^d}} M_\infty^\hf (x) < \frac{\sigma^\hf}{2},
\end{align}
are by Proposition \ref{prop:stab1} exponentially stable. In \cite[Lem. 2.3-2.4]{CRS} it was shown that 
\begin{align*}
    \frac{M_\infty}{\sigma} = \dfrac{1}{L^d} \, g\left(\frac{\Phi_0}{\sqrt{2\sigma}} \right),  \qquad g(\eta):= 1-\frac{2}{\sqrt{\pi}}\frac{\exp (-\eta^2)}{1+\erf (\eta)}\left[\frac{1}{\sqrt{\pi}}\frac{\exp (-\eta^2)}{1+\erf (\eta)}+\eta\right], 
\end{align*}
where $g$ is an increasing function whose values lie in the interval $\big(1-\frac{2}{\pi}, 1\big)$. Thus, a sufficient condition for local asymptotic stability of all stationary states is 
\begin{equation}\label{eq:stabcond3}
   \frac1{L^d} \norme{\Phi'}_\infty \norme{W}_{L^2(\mathbb T^d)} <  \frac1{2}.
\end{equation}
It is an open problem to determine whether multiple stationary states can coexist under \eqref{eq:stabcond2} and if there is only one stationary state under \eqref{eq:stabcond3}. Note also that the length of the domain $L$ plays a special role: when the domain becomes smaller, the sufficient condition for stability requires that W becomes smaller in $L^2$ norm. Another way to see this is to take $W\equiv w$, where $w\in\R$ is a constant, then the condition is not on $w$ but on the rescaled quantity $\tfrac w{L^d}$.
\end{remark}

\begin{remark}
     Assume $\Phi$ is an increasing function with $\Phi(x) = 0$ for $x\leq 0$, the external input is constant and positive ($B(t) = B > 0$), and the connectivity kernel is average inhibitory, that is
        \[\int_{\mathbb T^d} W(x)\,dx = W_0 < 0.\] 
    Then, we can deduce that in the high-noise case $\sigma > \tfrac{\pi B^2}{2 |W_0|^2}$,
        the unique space-homogeneous stationary state
        \[  \rho_\infty(s) = \dfrac{1}{L^d} \sqrt{\dfrac{2}{\pi \sigma}}\, e^{-\frac{s^2}{2\sigma}},  \]
        is locally asymptotically stable under under the sufficient condition
        \begin{equation*}
    \norme{\Phi'}_\infty \norme{W}_{L^2(\mathbb T^d)} \leqslant \frac {L^d} {2 \sqrt{ \left(1-\frac2\pi\right)}}.
\end{equation*}
\end{remark}

\begin{proof}[Proof of Proposition \ref{prop:stab1}]
Let $G(h) = \hf (h-1)^2$. Then 
\begin{align}
\frac{d}{dt}\int_0^{+\infty} & \frac{1}{2}(h-1)^2 \rhostat ds \nonumber \\
= & - \int_0^{+\infty} \left(\Phi_{0}-\Phi_{\bar{\rho}}\right)  \frac{\partial h}{\partial s}  h \rhostat  ds -\sigma \int_0^{+\infty} | \partial_s h |^2\rhostat  ds \nonumber \\
= & - \int_0^{+\infty} \left(\Phi_{0}-\Phi_{\bar{\rho}}\right)  \partial_s h  (h-1) \rhostat  ds -  \left(\Phi_{0}-\Phi_{\bar{\rho}}\right)\int_0^{+\infty}  \partial_s h  \rhostat  ds -\sigma \int_0^{+\infty} | \partial_s h |^2\rhostat  ds \label{eq:h-timeevo}\\
 \leq & \, \frac{1}{2\sigma}\left|\Phi_{0}-\Phi_{\bar{\rho}}\right|^2 \int_0^{+\infty} (h-1)^2 \rhostat  ds -\frac{\sigma}{2} \int_0^{+\infty} | \partial_s h |^2\rhostat  ds -  \left(\Phi_{0}-\Phi_{\bar{\rho}}\right)\int_0^{+\infty}  \partial_s h  \rhostat  ds, \nonumber
\end{align}
by applying the inequality $|ab| \leq \hf (a^2 + b^2)$ with $a = \sqrt{\frac{\rhostat}{\sigma}}(h-1)(\Phi_{0}-\Phi_{\bar{\rho}})$ and $b=\sqrt{\sigma\rhostat}  \partial_s h $. We bound the difference $\Phi_{0}-\Phi_{\bar{\rho}}$ as follows:
\begin{align*}
    \left|\Phi_{0}-\Phi_{\bar{\rho}}\right| & \leq \|\Phi' \|_\infty \left| \int_{\mathbb T^d} W(x-y) \int_0^{+\infty} s(\rhostat -\rho)(s,y) ds dy \right| \\
    & = \|\Phi' \|_\infty \left| \int_{\mathbb T^d} W(x-y) \int_0^{+\infty} s(1 -h)(s,y) \rhostat ds dy \right| \\
    & = \|\Phi' \|_\infty \left| \int_{\mathbb T^d} W(x-y) \int_0^{+\infty} (s-\bar\rho_\infty)(1 -h)(s,y) \rhostat ds dy \right| \\
    & \leq \|\Phi' \|_\infty \int_{\mathbb T^d} |W(x-y)|  M_\infty^\hf\left(\int_0^{+\infty} (h-1)^2 \rhostat ds \right)^\hf dy \\
    & \leq \|\Phi' \|_\infty \left(\int_{\mathbb T^d} |W(x-y)|^2 dy \right)^\hf \left(\int_{\mathbb T^d} \left[  M_\infty \int_0^{+\infty} (h-1)^2 \rhostat ds \right]dy \right)^\hf \\
    & = \|\Phi' \|_\infty \|W\|_{L^2({\mathbb T^d})} \left(\int_{\mathbb T^d} \left[  M_\infty \int_0^{+\infty} (h-1)^2 \rhostat ds \right]dy\right)^\hf \\
    & \leq  \|\Phi' \|_\infty \|W\|_{L^2({\mathbb T^d})} \sup_{x \in {\mathbb T^d}} M_\infty^\hf(x) \left(\int_{\mathbb T^d} \int_0^{+\infty} (h-1)^2 \rhostat ds dy \right)^\hf \\
    & = C(\Phi,W, \rhostat) \left(\int_{\mathbb T^d} \int_0^{+\infty} (h-1)^2 \rhostat ds dy \right)^\hf,
\end{align*}
by applying the Cauchy--Schwarz inequality twice, first in $s$, then in $y$. Summing up and by applying the Poincar\'e inequality \eqref{eq:poincare}, we arrive at 
\begin{align*}
\frac{d}{dt}\int_{\mathbb T^d}&\int_0^{+\infty} \frac{1}{2}(h-1)^2 \rhostat dsdy \\
\leq &\, \frac{1}{2\sigma} C^2 \left(\int_{\mathbb T^d} \int_0^{+\infty} (h-1)^2 \rhostat ds dy \right)^2 \\
& + C \left|\int_{\mathbb T^d}\int_0^{+\infty}  \partial_s h  \rhostat  dsdy \right|\left(\int_{\mathbb T^d} \int_0^{+\infty} (h-1)^2 \rhostat ds dy \right)^\hf -\frac{\sigma}{2} \int_{\mathbb T^d}\int_0^{+\infty} | \partial_s h |^2\rhostat  ds dy \\
\leq &\,\frac{1}{2\sigma} C^2 \left(\int_{\mathbb T^d} \int_0^{+\infty} (h-1)^2 \rhostat ds dy \right)^2 \\
& + \left(C-\frac{\sigma}{2} \tilde\gamma (\rhostat)^\hf\right) \left(\int_{\mathbb T^d}\int_0^{+\infty} | \partial_s h |^2\rhostat  dsdx\right)^\hf \left(\int_{\mathbb T^d} \int_0^{+\infty} (h-1)^2 \rhostat ds dy \right)^\hf.
\end{align*}
 By assumption \eqref{eq:poincareestimate}, $\tilde C:=C-\frac{\sigma}{2} \tilde\gamma (\rhostat)^\hf < 0$. Thus, we can apply the Poincar\'e inequality \eqref{eq:poincare} again, which results in
\begin{align*}
\frac{d}{dt}\int_{\mathbb T^d} & \int_0^{+\infty} \frac{1}{2}(h-1)^2 \rhostat dsdy \\
& \leq \left[\frac{C^2}{2\sigma}  \int_{\mathbb T^d} \int_0^{+\infty} (h-1)^2 \rhostat ds dy + \tilde\gamma (\rhostat)^\hf\tilde C \right] \int_{\mathbb T^d} \int_0^{+\infty} (h-1)^2 \rhostat ds dy.
\end{align*}
The expression between the brackets will be negative as long as we choose $\int_{\mathbb T^d} \int_0^{+\infty} (h-1)^2 \rhostat ds dy$ small enough at $t=0$ and $\sigma$ satisfies \eqref{eq:minfty}. Thus, we get exponential decay in time of the relative entropy under condition \eqref{eq:poincareestimate}.
\end{proof}

\begin{remark}
    The exponential decay rate obtained in Proposition \ref{prop:stab1} depends on the initial condition: we can write
    \begin{align*}
 \int_{\mathbb T^d} \int_{0}^{+\infty} \left( \dfrac{ \rho - \rho_\infty }{ \rho_\infty } \right)^2 \rho_\infty ds dx \leqslant e^{- 2 K t } \int_{\mathbb T^d} \int_{0}^{+\infty} \left( \dfrac{ \rho^0 - \rho_\infty }{ \rho_\infty } \right)^2 \rho_\infty ds dx.
\end{align*} 
with
\begin{align*} K =&\  \tilde\gamma (\rhostat)^\hf\left( \frac{\sigma}{2} \tilde\gamma (\rhostat)^\hf - \frac1{L^d}\|\Phi' \|_\infty \|W\|_{L^2({\mathbb T^d})} \sup_{x \in {\mathbb T^d}} M_\infty^\hf(x) \right)\\
&\ -\frac{\|\Phi' \|_\infty^2 \|W\|_{L^2({\mathbb T^d})}^2 \sup_{x \in {\mathbb T^d}} M_\infty(x)}{2L^d\sigma} \int_{\mathbb T^d} \int_{0}^{+\infty} \left( \dfrac{ \rho^0 - \rho_\infty }{ \rho_\infty } \right)^2 \rho_\infty ds dx.
\end{align*}
However, for all $\varepsilon\in\R_+^*$, there exists $T_\varepsilon$, depending on $\rho^0$, such that for all $t\in(T_\varepsilon,+\infty)$,
\begin{align*}
 &\int_{\mathbb T^d} \int_{0}^{+\infty} \left( \dfrac{ \rho(x,s,t) - \rho_\infty(x,s) }{ \rho_\infty(x,s) } \right)^2 \rho_\infty(x,s) ds dx \\ \leqslant&\, e^{- 2 K_\varepsilon t } \int_{\mathbb T^d} \int_{0}^{+\infty} \left( \dfrac{ \rho(x,s,T_\varepsilon) - \rho_\infty(x,s) }{ \rho_\infty(x,s) } \right)^2 \rho_\infty(x,s) ds dx,
\end{align*} 
where $K_\varepsilon =\tilde\gamma (\rhostat)^\hf\left( \frac{\sigma}{2} \tilde\gamma (\rhostat)^\hf - \frac1{L^\frac d2}\|\Phi' \|_\infty \|W\|_{L^2({\mathbb T^d})} \sup_{x \in {\mathbb T^d}} M_\infty^\hf(x) \right) - \varepsilon$.
\end{remark}

\begin{corollary}
    Grant the assumptions of Proposition \ref{prop:stab1}. Assume further that either
    \begin{itemize}
        \item $\Phi\in L^{\infty}(\R_+),$ and
    \[  -K:=\dfrac{2\norme{\Phi}_\infty^2}{\sigma}  +   \tilde\gamma (\rhostat )^\hf \left(\|\Phi' \|_\infty \frac{\|W\|_{L^2({\mathbb T^d})}}{L^\frac d2} \sup_{x \in {\mathbb T^d}} M_\infty^\hf (x)  - \frac{\sigma}{2} \tilde\gamma (\rhostat )^\hf\right) < 0, \quad \mathrm{or}   \]

        \item $W\leqslant 0$, $B\geqslant 0$, $\Phi$ is non-negative and increasing, and
        \[  -K := \dfrac{\Phi(B)^2}{\sigma}  +   \tilde\gamma (\rhostat )^\hf \left(\|\Phi' \|_\infty \frac{\|W\|_{L^2({\mathbb T^d})}}{L^\frac d2} \sup_{x \in {\mathbb T^d}} M_\infty^\hf (x)  - \frac{\sigma}{2} \tilde\gamma (\rhostat )^\hf\right) < 0.   \]
    \end{itemize}    
Then, for all compatible initial data $\rho_0$ such that
\[   \int_{\mathbb T^d} \int_{0}^{+\infty} \left( \dfrac{ \rho^0(x,s) - \rho_\infty(x,s) }{ \rho_\infty(x,s) } \right)^2 \rho_\infty(x,s) ds dx <  +\infty,      \]
the $L^1$ norm of the relative entropy satisfies
\begin{align*}
 \int_{\mathbb T^d} \int_{0}^{+\infty} \left( \dfrac{ \rho - \rho_\infty }{ \rho_\infty } \right)^2 \rho_\infty ds dx \leqslant e^{- 2 K t } \int_{\mathbb T^d} \int_{0}^{+\infty} \left( \dfrac{ \rho^0 - \rho_\infty }{ \rho_\infty } \right)^2 \rho_\infty ds dx.
\end{align*} 
In particular, the stationary state $\rho_\infty$ is the unique stationary state of the equation.
\end{corollary}

\begin{proof}
Let $M$ be either $\norme{\Phi}_\infty$ or $\Phi(B)$. We notice that in both cases, for all $(x,t)\in\mathbb T^d\times \R_+$, $\Phi_{\bar\rho} (x,t) \leqslant M$.
Then, coming back to \eqref{eq:h-timeevo}, we now use the bound 
\[\left|\Phi_{0}-\Phi_{\bar{\rho}}\right| \leqslant 2M\ \ (\mathrm{first \ case}) \qquad \mathrm{or} \qquad \left|\Phi_{0}-\Phi_{\bar{\rho}}\right| \leqslant M\ \ (\mathrm{second \ case})\]
 in the first term of the last inequality, while leaving the remainder of the proof of Proposition \ref{prop:stab1} unchanged. It yields, for example in the first case, the differential inequality
\begin{align*}
\frac{d}{dt}\int_{\mathbb T^d}  \int_0^{+\infty} \frac{1}{2}(h-1)^2 \rhostat dsdy
& \leq \left[\frac{2M^2}{\sigma} + \tilde\gamma (\rhostat)^\hf\tilde C \right] \int_{\mathbb T^d} \int_0^{+\infty} (h-1)^2 \rhostat ds dy.\\
& \leq - 2 K \int_{\mathbb T^d} \int_0^{+\infty} \frac12 (h-1)^2 \rhostat ds dy. 
\end{align*}
Since all stationary states have finite relative entropy, any other stationary state $\rho^1_\infty$ converges towards $\rho_\infty$ in relative entropy, and hence the uniqueness of the stationary state.    
\end{proof}

\begin{remark}
    The second case of the corrollary applies to the setting of \cite{CRS} with the additional assumption of a negative (fully inhibitory) connectivity kernel $W\leqslant 0$. With the same arguments as in the previous remarks, the sufficient condition can be simplified to the less sharp, but clearer,
    \[  \sigma >\frac{4\Phi(B)^4}{\left(\frac1{2} -\frac{1}{L^d} \norme{\Phi'}_\infty \norme{W}_{L^2(\mathbb T^d)}\right)^2} > 0\qquad \mathrm{and} \qquad  \frac{1}{L^d}\norme{\Phi'}_\infty \norme{W}_{L^2(\mathbb T^d)} <  \frac12, \]
    which partially confirms the conjecture that in the high-noise case there is at most one stationary state which is homogeneous in time and globally asymptotically stable.
\end{remark}

Although the above result applies to some stationary states, it is difficult to determine the full range of $\sigma$ for which condition \eqref{eq:poincareestimate} actually holds. Therefore, building on the linear stability estimate in \cite{CHS}, we establish below that the spatially homogeneous stationary states are nonlinearly asymptotically stable under similar conditions to the linearised case. In \cite{CHS} it was shown that spatially homogeneous stationary states of \eqref{eq:PDE}-\eqref{eq:BC} are linearly exponentially stable if all the unnormalised Fourier modes of $W$, $\hat{W}_k$, defined as
\begin{align*}
\hat{W}_k = \int_{{\mathbb{T}^d}} W(\bx)\exp(2\pi k\cdot\bx) d\bx, \qquad k \in \Z^d,
\end{align*}
satisfy the condition 
\begin{equation}\label{eqn:W_tilde_s_0}
        \Phi_0' \hat W_k < \frac{\sigma}{M_\infty}.
\end{equation}  
This was done by carefully combining estimates for the time derivative of the relative entropy of each Fourier mode $(\hat \rho - \hat \rho_\infty)_k(s,t)$ of the perturbation $(\rho-\rhostat)(x,s,t)$,  
\begin{align*}
\int_0^{+\infty}  \frac{1}{2}\left(\frac{(\hat \rho - \hat \rho_\infty)_k}{\rhostat}\right)^2 \rhostat ds
\end{align*}
and the square of the Fourier modes of the perturbation of the mean, $ (\hat \brho - \hat \brho_\infty)_k$, $((\hat \brho - \hat \brho_\infty)_k)^2$, yielding exponential decay of the quantities
\begin{align}\label{eq:fouriermodes}
\int_0^{+\infty}  \left(\frac{(\hat \rho - \hat \rho_\infty)_k}{\rhostat}\right)^2 \rhostat ds - \frac{\Phi_0'}{\sigma} \hat W_k ((\hat \brho - \hat \brho_\infty)_k)^2, \qquad k \in \mathbb Z^d,
\end{align}
which where shown to be positive under condition \eqref{eqn:W_tilde_s_0}. This was sufficient to prove that perturbations away from the spatially homogeneous state are linearly exponentially stable in relative entropy. Condition \eqref{eqn:W_tilde_s_0} was later found to be optimal in the sense that the condition leading to bifurcations along the branch of spatially homogeneous solutions of \eqref{eq:PDE}-\eqref{eq:BC} is exactly \eqref{eqn:W_tilde_s_0} with the inequality replaced by an equality \cite{CRS}.

Motivated by the approach of the proof in \cite{CHS}, we carefully obtain estimates for a combination of the time derivatives of the quantities
\begin{align*}
\mathcal{E} =  \int_0^{+\infty}  \frac{1}{2}\left(\frac{\rho- \rhostat}{ \rhostat}\right)^2 \rhostat ds, \quad \mathrm{and} \quad \mathcal{H}=    \Phi^\delta \bar \rho^\delta,
\end{align*}
where $ \Phi^\delta = \Phi_{\bar \rho}-\Phi_0$, $ \bar \rho^\delta = \bar \rho - \bar \rho_\infty$. Linearizing $\Phi^\delta$, we get
\begin{align*}
2\mathcal{E}-\frac{\mathcal{H}}{\sigma} \simeq  \int_0^{+\infty} \left(\frac{\rho- \rhostat}{ \rhostat}\right)^2 \rhostat ds - \frac{\Phi'_0}{\sigma} W\ast(\bar \rho - \bar \rho_\infty) (\bar \rho - \bar \rho_\infty).
\end{align*}
Integrating the previous expression in $x\in{\mathbb{T}^d}$, we obtain
\begin{align*}
\int_{\mathbb T^d} 2\mathcal{E}-\frac{\mathcal{H}}{\sigma} dx \simeq  \int_{\mathbb T^d}\int_0^{+\infty} \left(\frac{\rho- \rhostat}{ \rhostat}\right)^2 \rhostat ds dx - \int_{\mathbb T^d} \frac{\Phi'_0}{\sigma} W\ast(\bar \rho - \bar \rho_\infty) (\bar \rho - \bar \rho_\infty) dx.
\end{align*}
Notice that assuming $\rhostat$ is spatially homogeneous and using the Parseval--Plancherel identity, we obtain the sum of the Fourier coefficients in \eqref{eq:fouriermodes}.

Now, utilising the similarity between \eqref{eq:fouriermodes} and $2\mathcal{E}-\frac{\mathcal{H}}{\sigma}$, we show exponential decay of the latter in $L^1(\mathbb T ^d)$ under conditions closely related to the linear stability condition \eqref{eqn:W_tilde_s_0}, see Remark \ref{rem:goodRemark} for an ellaboration on the relation between the conditions. This will in turn give exponential decay of the relative entropy \eqref{eq:relentropy}. To complete the proof, we will, in addition to the equality \eqref{eq:h-timeevo} and the Poincar\'e inequality \eqref{eq:poincare}, also need the time evolution of $\bar \rho^\delta=\brho -\brho_\infty$:
\begin{align}\label{eq:barrhoevo}
    \frac{\partial}{\partial t}(\brho-\brho_\infty) = \left(\Phi_{\bar{\rho}}-\Phi_{0}\right) - \sigma \int_0^{+\infty}  \partial_s h  \rhostat ds.
\end{align} 
Equation \eqref{eq:barrhoevo} is derived from \eqref{eq:test_s} and noting that
\begin{align*}
 \brho - \brho_\infty & =   \int_0^{+\infty} s (\rho -\rhostat) ds = \int_0^{+\infty} s (h-1)\rhostat ds = \sigma \int_0^{+\infty} \frac{s-\Phi_0}{\sigma} (h-1)\rhostat ds \\
 & = - \sigma \int_0^{+\infty} (h-1) \partial_s\rhostat ds = \sigma(\rho-\rhostat)(0) + \sigma \int_0^{+\infty}  \partial_s h  \rhostat ds.
\end{align*}
We will also use the notation
\begin{align*}
    \|f\|_{L^2_{\rhostat(x)}} = \left(\int_0^{+\infty} f(x,s)^2\rhostat(x,s) ds \right)^\hf\!\!, \qquad \|f\|_{L^2_{\rhostat}} = \left(\int_{\mathbb T^d}\int_0^{+\infty} f(x,s)^2\rhostat(x,s) ds dx \right)^\hf\!\!.
\end{align*}

\begin{proposition}\label{prop:stab}
    Let $\Phi \in C^2(\mathbb R)\cap \mathcal W^{2,\infty}(\R)$, $W \in L^2(\mathbb T^d)$, and $\rhostat$ be a spatially homogeneous stationary state of \eqref{eq:PDE}-\eqref{eq:BC} given a $\sigma$. Assume that $W$ is componentwise symmetric, and that for a suitably small $1>\alpha >0$, $\sigma$ satisfies
    \begin{align}\label{eq:nonlinstab-condition}
    \int_{\mathbb T^d} (1-\alpha)g^2(x) - \frac{\Phi^\delta_g(x)}{\sigma} g(x) \, dx > 0, \quad g \in L^2(\mathbb T^d),
    \end{align}
    where $$\Phi^\delta_g(x) =\Phi\big( M_\infty W\ast g (x) +W_0 \brho_\infty +B\big)-\Phi_0.$$ Then, the $L^1$ norm of the relative entropy
\begin{align}\label{eq:relentropy}
 \int_{\mathbb T^d} \int_{0}^{+\infty} \left( \dfrac{ \rho(x,s,t) - \rho_\infty(s) }{ \rho_\infty(s) } \right)^2 \rho_\infty(s) ds dx,
\end{align} 
 decays exponentially fast whenever the compatible initial data $\rho_0$ is close enough to the stationary state $\rhostat$ in relative entropy.   
\end{proposition}
\begin{remark}\label{rem:goodRemark}
Note that 
\begin{align}\label{eq:nonlinstab2}
        \int_{\mathbb T^d} (1-\alpha) g^2 - \frac{\Phi^\delta_g}{\sigma} g \, dx & \geq  (1-\alpha) \left(\int_{\mathbb T^d}g dx\right)^2 - \int_{\mathbb T^d}\frac{\Phi^\delta_g}{\sigma} g \, dx \nonumber \\
        &=
        \int_{\mathbb T^d}\int_{\mathbb T^d} g(x)\left(1-\alpha-\frac{M_\infty}{\sigma}\tilde \Phi_g'(x) W(x-y) \right) g(y) dydx, 
\end{align}
with $\tilde \Phi_g'(x) =\frac{\Phi^\delta_g(x)}{W\ast g (x)}.$ Requiring that the right hand side of the equality in \eqref{eq:nonlinstab2} is positive is quite similar to the linear stability condition of \cite{CHS}, see \eqref{eqn:W_tilde_s_0}. In fact, assuming $\Phi$ is linear, then $ \Phi^\delta_g = M_\infty W \ast g$, and the right hand side of the equality of \eqref{eq:nonlinstab2} is 
\begin{align*}
    \int_{\mathbb T^d}\int_{\mathbb T^d} g(x)\left(1-\alpha-\frac{M_\infty}{\sigma} \Phi' W(x-y) \right) g(y) dydx.
\end{align*}
Thus, in the linear $\Phi$ case requiring that the right hand side of \eqref{eq:nonlinstab2} is positive, which is also known as requiring that $c(x) = 1-\alpha-\frac{M_\infty}{\sigma}\Phi' W(x)$ is H-stable, see \cite{CGPS20}, is identical to requiring that $c(x)$ only has positive unnormalised Fourier modes. This is again identical to the linear condition \eqref{eqn:W_tilde_s_0} if we replace the numerator on the right hand side of \eqref{eqn:W_tilde_s_0} with $\sigma(1-\alpha)$. Lastly, if $\Phi$ is linear and increasing, and $-W$ itself is H-stable, then any spatially homogeneous stationary state is locally asymptotically stable in relative entropy. This is in accordance with the finding of \cite{CRS}: there are no noise-driven bifurcations from the curve of homogeneous stationary states when $-W$ is H-stable.  
\end{remark}

\begin{proof}[Proof of Proposition \ref{prop:stab}] 
Using \eqref{eq:h-timeevo} and \eqref{eq:barrhoevo},
\begin{align*}
  \frac{d}{dt} \int_{\mathbb T^d} 2\mathcal{E}- \frac{1}{\sigma}\mathcal{H} \, dx =\,&  2 \int_{\mathbb T^d} \Phi^\delta \int_0^{+\infty}  \partial_s h  (h-1) \rhostat  ds dx - 2\sigma \int_{\mathbb T^d} \int | \partial_s h |^2 \rhostat ds dx \\
  & + 3 \int_{\mathbb T^d} \Phi^\delta \int_0^{+\infty}  \partial_s h   \rhostat  ds dx - \frac{1}{\sigma} \int_{\mathbb T^d} (\Phi^\delta)^2 dx \\
  &- \frac{1}{\sigma}\int_{\mathbb T^d} \brho^\delta \, \Phi'_\brho W \ast \left(\Phi^\delta - \sigma \int_0^{+\infty}  \partial_s h   \rhostat  ds \right) dx\\
  =& -\frac{2}{\sigma}\int_{\mathbb T^d}\int_0^{+\infty} \left(\Phi^\delta -\sigma  \partial_s h  \right)^2 \rhostat  ds dx \\
  &+\frac{1}{\sigma} \int_{\mathbb T^d} \Phi^\delta \left(\Phi^\delta-\sigma \int_0^{+\infty} \partial_s h   \rhostat  ds \right) dx + 2 \int_{\mathbb T^d} \Phi^\delta \int_0^{+\infty}  \partial_s h  (h-1) \rhostat  ds dx\\
  &- \frac{1}{\sigma}\int_{\mathbb T^d} \brho^\delta \, \Phi'_\brho W \ast \left(\Phi^\delta - \sigma \int_0^{+\infty}  \partial_s h   \rhostat  ds \right) dx \\
   =& -\frac{2}{\sigma}\int_{\mathbb T^d}\int_0^{+\infty} \left(\Phi^\delta -\sigma  \partial_s h  \right)^2 \rhostat  ds dx \\
& +\frac{1}{\sigma} \int_{\mathbb T^d} \left(\Phi^\delta - W\ast (\Phi_\brho'\brho^\delta)\right)\int_0^{+\infty} \left(\Phi^\delta-\sigma  \partial_s h   \right) \rhostat  ds dx \\
& - \frac{2}{\sigma} \int_{\mathbb T^d} \int_0^{+\infty} (h-1) \Phi^\delta \left( \Phi^\delta -\sigma \partial_s h \right) \rhostat ds dx \\
  =&\,I+II+III,
\end{align*}
where the last equality follows by the symmetry of $W$ and noting that $\Phi^\delta$ does not depend on $s$ and that $\int_0^{+\infty} (h-1) \rhostat ds = 0$. Notice first that
\[ 
I  = -\frac2\sigma \big\| \Phi^\delta -\sigma  \partial_s h  \big\|_{L^2_{\rhostat}}^2 \,.
\]
Then, for for $II$, 
\begin{align*}
    II \leq\, &\frac{1}{\sigma} \int_{\mathbb T^d}\big|\Phi^\delta - W\ast (\Phi_\brho'\brho^\delta)\big|\| \Phi^\delta-\sigma  \partial_s h  \|_{L^2_{\rhostat(x)}} dx\\ \leqslant\, & \frac{1}{\sigma}\int_{\mathbb T^d}\big\| \Phi^\delta -\sigma  \partial_s h  \big\|_{L^2_{\rhostat(x)}} \big|\Phi^\delta - W\ast (\Phi_\brho'\brho^\delta)\big|dx\\
    \leqslant\, & \frac{1}{\sigma}\big\| \Phi^\delta -\sigma  \partial_s h  \big\|_{L^2_{\rhostat}}\left(\int_{\mathbb T^d}\big|\Phi^\delta - W\ast (\Phi_\brho'\brho^\delta)\big|^2 dx\right)^\hf,
\end{align*}
and for $III$, we have
\begin{align*}
    III \leq\, &\frac{2}{\sigma} \int_{\mathbb T^d} |\Phi^\delta| \|h-1\|_{L^2_{\rhostat(x)}} \| \Phi^\delta -\sigma  \partial_s h \|_{L^2_{\rhostat(x)}}dx\\
    \leq \, & \frac{2}{\sigma}\int_{\mathbb T^d} \big\| \Phi^\delta -\sigma  \partial_s h \big\|_{L^2_{\rhostat(x)}} |\Phi^\delta| \|h-1\|_{L^2_{\rhostat(x)}}  dx\\
    \leq \, & \frac{1}{\sigma} \big\| \Phi^\delta -\sigma  \partial_s h \big\|_{L^2_{\rhostat}}  2 \left(\int_{\mathbb T^d}|\Phi^\delta|^2 \|h-1\|_{L^2_{\rhostat(x)}}^2 dx\right)^\hf,
\end{align*}
yielding
\begin{align}
   \frac{d}{dt}\int_{\mathbb T^d} & 2\mathcal{E}- \frac{1}{\sigma}\mathcal{H} \, dx \nonumber \\
\leq & \, -\frac{1}{\sigma}\big\| \Phi^\delta -\sigma  \partial_s h  \big\|_{L^2_{\rhostat}}\left(2\big\| \Phi^\delta -\sigma  \partial_s h  \big\|_{L^2_{\rhostat}}- \left(\int_{\mathbb T^d}\big|\Phi^\delta - W\ast (\Phi_\brho'\brho^\delta)\big|^2 dx\right)^\hf\right. \nonumber \\
   & \qquad\qquad\qquad\qquad\qquad\qquad\qquad\qquad\qquad\left. - 2 \left(\int_{\mathbb T^d}|\Phi^\delta|^2 \|h-1\|_{L^2_{\rhostat(x)}}^2 dx\right)^\hf\right). \label{eq:estsofar}
\end{align}
Let us denote 
\begin{align*}
    \mathcal{Q}:=\int_{\mathbb T^d} \mathcal E - \frac1\sigma \mathcal H\,  dx.
\end{align*}
We now sketch the strategy of the remaining steps of the proof. We first claim that $\mathcal{Q}$ is equivalent to the relative entropy as
\begin{equation}\label{new0}
0\,\leqslant\, 2\alpha \int_{\mathbb T^d}\mathcal{E}dx\,  \leqslant\, \mathcal Q \,\leqslant \,\left(1+\frac{\|\Phi'\|_\infty}{\sigma}\|W\|_{L^2(\mathbb T^d)}M_\infty\right) \int_{\mathbb T^d}\mathcal{E}dx. 
\end{equation}
This will be shown in Step 2 below. Then, with the estimates 
\begin{align}\label{new1}
\big\| \Phi^\delta -\sigma  \partial_s h \big\|_{L^2_{\rhostat}} \geqslant \sigma\gamma(\rhostat)^\hf\alpha \mathcal Q^\hf
\end{align}
and
\begin{align}\label{new2}
\left(\int_{\mathbb T^d}\big|\Phi^\delta - W\ast (\Phi_\brho'\brho^\delta)\big|^2 dx\right)^\hf +  2 \left(\int_{\mathbb T^d}|\Phi^\delta|^2 \|h-1\|_{L^2_{\rhostat(x)}}^2 dx\right)^\hf \leqslant   2 C_4 \mathcal Q,
\end{align}
 where $C_4$ is a positive constant, we show exponential decay of $\mathcal{Q}$, which then yields exponential decay of the relative entropy. The proofs of the estimates are postponed to Step 3 and 4. 
 
{\bf Step 1. Exponential decay of $\mathcal{Q}$.-} Assuming \eqref{new0}-\eqref{new2}, let us finalize the proof. From \eqref{eq:estsofar}, we conclude
\begin{align}\label{eq:Qdecay}
\frac{d}{dt} \mathcal{Q} \leq \, -2\gamma(\rhostat)^\hf \big\| \Phi^\delta -\sigma  \partial_s h \big\|_{L^2_{\rhostat}}\left(\alpha \mathcal{Q}^\hf - C_4 \mathcal{Q} \right).
\end{align}
Thus, by choosing $\rho_0$ close enough in relative entropy to $\rhostat$, we can ensure that the term between the parantheses in \eqref{eq:Qdecay} is positive initially, and so $\frac{d}{dt} \mathcal{Q} \leq 0$, which then will keep the term between the parenthesis positive. This allows us to apply \eqref{new1} again, such that we arrive at 
\begin{align}\label{eq:Q}
\frac{d}{dt} \mathcal{Q} &\leq \, -2\sigma\gamma(\rhostat) \alpha \mathcal{Q}^\hf \left(\alpha \mathcal{Q}^\hf - C_4 \mathcal{Q} \right) \nonumber \\
& = -2\sigma\gamma(\rhostat) \alpha \mathcal{Q} \left(\alpha - C_4 \mathcal{Q}^\hf \right)\\
& \leq -2\sigma\gamma(\rhostat)  \alpha \left(\alpha - C_4 \mathcal{Q}^\hf_0 \right) \mathcal{Q}, \nonumber
\end{align}
where $\mathcal{Q}^\hf_0 = \mathcal{Q}^\hf(t=0)$, yielding exponential decay of $\mathcal{Q}$, and, due to \eqref{new0}, consequently of \eqref{eq:relentropy} whenever $\rho_0$ is close enough in relative entropy to $\rhostat$.

{\bf Step 2. Positivity and equivalence to relative entropy of $\mathcal{Q}$.-} As in the proof of \cite[Thm. 3.6]{CHS}, let 
\begin{align}\label{eq:changevar}
h(x,s,t)-1 = V(x,s,t) + A(x,t)(s-\brho_\infty),    
\end{align}
where $A(x,t) $ is chosen such that $\int_0^{+\infty} V(s-\brho_\infty)\rhostat ds = 0$ for all $x \in \mathbb T^d$, i.e. $A(x,t) = \frac{\brho^\delta}{M_\infty}$. We first rewrite
\begin{align*}
    \mathcal{Q}=\int_{\mathbb T^d} \int_0^{+\infty} V^2 \rhostat ds +M_\infty \left( A^2 - \frac{\Phi^\delta}{\sigma} A \right)\, dx  . 
\end{align*}
Due to \eqref{eq:nonlinstab-condition}, we have 
\begin{align*}
 \int_{\mathbb T^d}A^2 dx \leq \frac{1}{\alpha}  \left(  \int_{\mathbb T^d}A^2 - \frac{1}{\sigma}\Phi^\delta A \, dx \right).
\end{align*}
This fact entails
\begin{align*}
    2\alpha \int_{\mathbb T^d}\mathcal{E}dx &= \alpha \int_{\mathbb T^d} \int_0^{+\infty}V^2\rhostat ds + M_\infty A^2 dx  \nonumber \\
    &\leq \int_{\mathbb T^d} \alpha \int_0^{+\infty}V^2\rhostat ds + M_\infty \left(A^2 - \frac{\Phi^\delta}{\sigma}A\right) dx \nonumber \\
    &\leq \mathcal{Q} \\
    &\leq \int_{\mathbb T^d}\mathcal{E}dx + \frac{1}{\sigma}\int_{\mathbb T^d} |\Phi^\delta||\brho^\delta|dx \nonumber\\
    &\leq \left(1+\frac{\|\Phi'\|_\infty}{\sigma}\|W\|_{L^2(\mathbb T^d)}M_\infty\right) \int_{\mathbb T^d}\mathcal{E}dx, \nonumber 
\end{align*}
and concludes the desired equivalence \eqref{new0}.

{\bf Step 3. Proof of \eqref{new1}.-}
By the Poincar\'e inequality with respect to the measure $\rhostat$ in \eqref{eq:poincare} with $g = s\Phi^\delta - \sigma (h-1)$, at each $x$ we have 
\begin{align}
 \big\| \Phi^\delta -\sigma  \partial_s h \big\|_{L^2_{\rhostat(x)}}^2 &\geq \gamma (\rhostat )\sigma^2 \left(\|h-1\|_{L^2_{\rhostat(x)}}^2+\frac{M_\infty}{\sigma^2} (\Phi^\delta)^2 -\frac{2}{\sigma}\Phi^\delta \brho^\delta\right)\nonumber\\
 &= \gamma (\rhostat )\sigma^2 \left(2\mathcal{E}-\frac{1}{\sigma }\mathcal{H} +\frac{M_\infty}{\sigma^2} (\Phi^\delta)^2 -\frac{1}{\sigma}\Phi^\delta \brho^\delta\right).\label{eq:important_poincare}
\end{align}
Since
\begin{align*}
  2\mathcal{E}-\frac{1}{\sigma }\mathcal{H} = \int_0^{+\infty} V^2 \rhostat ds + M_\infty A^2 - \frac{M_\infty}{\sigma}\Phi^\delta A,
\end{align*}
the terms between the brackets of \eqref{eq:important_poincare} simplify to 
\begin{align*}
 2\mathcal{E}-\frac{1}{\sigma }\mathcal{H} +\frac{M_\infty}{\sigma^2} (\Phi^\delta)^2 -\frac{1}{\sigma}\Phi^\delta \brho^\delta = \int_0^{+\infty} V^2 \rhostat ds + M_\infty \left(A-\frac{\Phi^\delta}{\sigma} \right)^2.
\end{align*}
Notice that
\begin{align}\label{eq:APhisquared}
 \left(A-\frac{\Phi^\delta}{\sigma} \right)^2 \geq   
 \begin{cases}
 A^2 - \frac{\Phi^\delta}{\sigma} A, \quad &\sgn(\Phi^\delta) \neq \sgn(A),  \\
\alpha \left(  A^2 - \frac{\Phi^\delta}{\sigma} A\right), \quad &\sgn(\Phi^\delta) = \sgn(A), |\Phi^\delta| \geq \sigma (1-\alpha)|A|,  \\
 \alpha^2 \left(  A^2 - \frac{\Phi^\delta}{\sigma} A\right), \quad &\sgn(\Phi^\delta) = \sgn(A), |\Phi^\delta| < \sigma (1-\alpha)|A|,
 \end{cases}
\end{align}
and thus, we conclude that
\[
\big\| \Phi^\delta -\sigma  \partial_s h \big\|_{L^2_{\rhostat}}^2 \geqslant \sigma^2\gamma(\rhostat) \alpha^2 \left(\int_{\mathbb T^d}\int_0^{+\infty} V^2 \rhostat ds + M_\infty \left(A^2-\frac{\Phi^\delta}{\sigma}A \right)dx\right) = \sigma^2\gamma(\rhostat) \alpha^2 \mathcal Q,
\]
leading to the claim \eqref{new1}.

{\bf Step 4. Proof of \eqref{new2}.-}
For the first term of \eqref{eq:estsofar}, using the mean value theorem, introduce
\begin{align*}
   \phi'(x) = \frac{\Phi^\delta (x)}{W \ast \brho^\delta (x)} = \Phi'\big( \eta(x) W\ast\brho+(1-\eta(x))W\ast\brho_\infty+B\big), 
\end{align*}
where $\eta(x)$ is a parameter which, for each $x$, yields the exact point between $W \ast \brho (x)$ and $W \ast \brho_\infty = \brho_\infty \int_{\mathbb T^d} W(x) dx $ to be used in the mean value theorem, $0\leq \eta(x)\leq 1$.
Then,
\begin{align}\label{eq:thedifference}
\left|\Phi^\delta - W\ast (\Phi_\brho'\brho^\delta)\right| & = \left|\int_{\mathbb{T}^d} \brho^\delta (y) W(x-y) \left(\phi'(x) - \Phi_\brho' (y) \right) dy \right|  \\
& \leq \|\Phi''\|_\infty \int_{\mathbb{T}^d} \int_{\mathbb{T}^d} \left| \brho^\delta (y) W(x-y) (W(x-z)\eta(x)+W(y-z))\brho^\delta(z) \right| dy dz \nonumber \\
& \leq 2 \|\Phi''\|_\infty \|W\|_{L^2(\mathbb T^d)}^2\int_{\mathbb{T}^d} (\brho^\delta )^2 dx. \nonumber
\end{align}

To summarise, replacing $h-1$ and $\brho^\delta$ by $V$ and $A$ of \eqref{eq:changevar}, and inserting the above estimates, we now have
\begin{align*}
  &\left(\int_{\mathbb T^d}\big|\Phi^\delta - W\ast (\Phi_\brho'\brho^\delta)\big|^2 dx\right)^\hf +  2 \left(\int_{\mathbb T^d}|\Phi^\delta|^2 \|h-1\|_{L^2_{\rhostat(x)}}^2 dx\right)^\hf\\ &\qquad\qquad\qquad\qquad\leq\, - 2 C_2 \int_{\mathbb{T}^d} A^2 dx   - \frac{2}{\sigma C_1}  \left( \int_{\mathbb T^d}(\Phi^\delta)^2 \left(\int_0^{+\infty} V^2 \rhostat ds + M_\infty A^2\right) dx \right)^\hf,
\end{align*}
where $C_1 = \gamma (\rhostat )^\hf$, and $C_2 = \frac{M_\infty^2\|\Phi''\|_\infty\|W\|_{L^2(\mathbb T^d)}^2}{\gamma (\rhostat )^\hf \sigma}$. Estimating further, we find that
\begin{align*}
  &\left(\int_{\mathbb T^d}\big|\Phi^\delta - W\ast (\Phi_\brho'\brho^\delta)\big|^2 dx\right)^\hf +  2 \left(\int_{\mathbb T^d}|\Phi^\delta|^2 \|h-1\|_{L^2_{\rhostat(x)}}^2 dx\right)^\hf\\ &\qquad\qquad\qquad\qquad\leq\,  C_2 \int_{\mathbb{T}^d} A^2 dx +C_3 \int_{\mathbb T^d}\int_0^{+\infty} V^2 \rhostat ds + M_\infty A^2 dx \\
  &\qquad\qquad\qquad\qquad \leq\,  C_3 \int_{\mathbb T^d}\int_0^{+\infty} V^2 \rhostat dsdx + (C_2+M_\infty) \int_{\mathbb T^d}A^2 dx,
\end{align*}
with $C_3= \frac{M_\infty\|\Phi'\|_\infty^2\|W\|_{L^2(\mathbb T^d)}^2}{\gamma (\rhostat )^\hf \sigma}$, after applying the estimate of $\Phi^\delta$ in the proof of Proposition \ref{prop:stab1} to the last term. 
It leads to
\begin{align*}
&\left(\int_{\mathbb T^d}\big|\Phi^\delta - W\ast (\Phi_\brho'\brho^\delta)\big|^2 dx\right)^\hf +  2 \left(\int_{\mathbb T^d}|\Phi^\delta|^2 \|h-1\|_{L^2_{\rhostat(x)}}^2 dx\right)^\hf\\ &\qquad\qquad\qquad\qquad \leq\,
C_4 \int_{\mathbb T^d}\int_0^{+\infty} V^2 \rhostat ds + M_\infty \left(A^2-\frac{\Phi^\delta}{\sigma}A \right) dx  = C_4 \mathcal Q,
\end{align*}
where $C_4 = \max \left\{C_3,\frac{C_2+M_\infty}{\alpha M_\infty} \right\}$, finalizing the proof of \eqref{new2}.
\end{proof}

\begin{remark}
There is a trade-off between the restrictiveness of the stability condition \eqref{eq:nonlinstab-condition} and the decay rate of $\mathcal Q$ in \eqref{eq:Q} with respect to $\alpha$. Keeping all other quantities fixed, increasing $\alpha$ improves the decay rate of $\mathcal Q$, while decreasing $\alpha$, makes condition \eqref{eq:nonlinstab-condition} less restrictive.  
\end{remark}

\begin{remark}
If we drop the symmetry assumption on the connectivity $W$ in Proposition \ref{prop:stab}, and instead assume $W \in L^2(\mathbb T^d)\cap \mathcal W^{1,\infty}(\mathbb T^d)$, we obtain exponential decay of \eqref{eq:relentropy} under \eqref{eq:stabcond3} and the additional condition
\begin{align}\label{eq:additionalcond}
       2L^d \leq \frac{\alpha^{\unitfrac{3}{2}}(2-\xi)\gamma(\rhostat)^\hf \sigma}{\|\nabla W\|_\infty\|\Phi'\|_\infty M_\infty^\hf},
\end{align}
for a suitably small $\xi >0$.
The result can be obtained by replacing $W\ast (\Phi_\brho'\brho^\delta) = \int_{\mathbb T^d}W(x-y)\Phi_\brho'(y)\brho^\delta(y) dy$ with $\int_{\mathbb T^d}W(y-x)\Phi_\brho'(y)\brho^\delta(y) dy$ in II and \eqref{eq:thedifference}, and then estimating the resulting difference as 
\begin{align*}
&\left|\Phi^\delta(x) - \int_{\mathbb T^d}W(y-x)\Phi_\brho'(y)\brho^\delta(y) dy\right| \\
&\quad \leq 2 \|\Phi''\|_\infty \|W\|_{L^2(\mathbb T^d)}^2\int_{\mathbb{T}^d} (\brho^\delta )^2 dx + \|\Phi'\|_\infty \int_{\mathbb T^d}\left|\brho^\delta (y) (W(x-y)-W(y-x))\right|dy,
\end{align*}
where the first term is as in the proof above and the last term pops up due to the lack of symmetry of $W$. By a similar argument to the last paragraphs of the proof of Proposition \ref{prop:4stab} for the four population model \eqref{eq:4PDE}, we get the additional condition \eqref{eq:additionalcond}, cf. \eqref{eq:shiftcondition}.
\end{remark}

\section{Extension of the main results to a model for grid cells}
\label{sec:4}

The stochastic neural field PDE system \eqref{eq:PDE} was first proposed in \cite{CHS} in the following more general form, with four populations of neurons with an orientation preference:
\begin{align}\label{eq:4PDE}
\tau_c \frac{\partial \rho^\beta}{\partial t} =
-\frac{\partial}{\partial s}\Bigg(
\Big[\Phi^\beta(x,t) -s\Big] \rho^\beta
\Bigg) + \sigma \frac{\partial^2 \rho^\beta}{\partial s^2},
\end{align}
where $\Phi^\beta(x,t)$ is given by
\begin{align}\label{eq:phi}
\Phi^\beta(x,t) = \Phi \left(\frac{1}{4}\sum_{\beta'=1}^4 \int_{\mathbb T^d} W^{\beta'}(x-y) \int_{0}^\infty s \rho^{\beta'} (y,s,t)\, \diff s \diff y + B^\beta(t) \right),
\end{align}
where the orientation preference in physical space is represented by the index $\beta =1,2,3,4$ (north, west, south, east), and with corresponding no-flux boundary and mass normalisation conditions. In the case $d=2$, the system \eqref{eq:4PDE} models a network of noisy grid cells \cite{CHS}. The orientation preference is modelled by $\beta-$dependent shifts $r^\beta$ and connectivity $W^\beta(x-y)=W(x-y-r^\beta)$. The movement of an animal can be connected to this PDE system \eqref{eq:4PDE} through the orientation and time dependent input $B^\beta(t)$. The system was rigorously derived from a stochastic particle system in \cite{CCS} and its noise-driven bifurcations from the homogeneous stationary state were rigorously characterised in \cite{CRS}.

In the same fashion as in \cite{CRS}, the main analytical results obtained on the solutions of \eqref{eq:PDE}-\eqref{eq:BC} can be extended to the more general system \eqref{eq:4PDE}. By checking the key points of the proof and without writing all the technical details, we first extend the local and global in time existence results of Theorem \ref{thm:main} to the four population system to a very general setting, before the asymptotic convergence results of Theorem \ref{thm:main2} are extended to \eqref{eq:4PDE}.

\subsection{Global-in-time solutions for Lipschitz modulation functions}

The proof of local and global existence of solutions for the four population model \eqref{eq:4PDE}-\eqref{eq:phi} being a direct extension of our work in Section \ref{sec:2}, we will only go through the main steps in order to check that no additional difficulty endangers the method. 

We can apply the same change of variables as before to each population
\[  y=e^{t}s      \qquad  \tau = \dfrac{1}{2}(e^{2t}-1)   ,\qquad  q^\beta(x,y,\tau) = \alpha(\tau) \rho^\beta\Big(x,y\alpha(\tau), -\log(\alpha(\tau))\Big), \]
and then we can apply separately to each population the change of variable
\begin{align} &\gamma^\beta(x,\tau) = - \int_{0}^{\tau} \Phi^\beta\left( \alpha(\eta)\sum_{\beta=1}^{4} W \ast \bar q^\beta(x,\eta)  + \lambda^\beta(\eta)\right)\alpha(\eta) d \eta,\\ & u^\beta(x,z,\tau)=q^\beta(x,y,\tau), \qquad\bar u^\beta = \int_{\gamma^\beta(x,\tau)}^{+\infty} zu^\beta(z,\tau)dz, \quad v^\beta(x,\tau)= u^\beta(x,\gamma(x,\tau),\tau). \end{align}

Then, we obtain a similar free boundary Stefan problem with non-local Robin boundary condition: for $\beta\in\{1,2,3,4\}$,
\begin{equation}\label{eq:4Stefan}
\left\{\begin{array}{rcl}
\displaystyle \dfrac{\partial u^\beta}{\partial \tau}(x,z,\tau) &=& \dfrac{\partial^2 u^\beta}{\partial z^2}(x,z,\tau), \qquad\qquad\qquad\  z\in(\gamma^\beta(x,\tau), +\infty), \\[2mm]
\displaystyle \gamma^\beta(x,\tau) &=& - \displaystyle\int_{0}^{\tau} \Psi^\beta(x,\eta) d \eta, \\
\dfrac{\partial u^\beta}{\partial y}(x,\gamma(x,\tau),\tau) &=& \Psi^\beta(x,\tau) u^\beta(x,\gamma(x,\tau),\tau), \\
\Psi^\beta(x,\tau) &=&\displaystyle \Phi\left( \alpha(\tau) \sum_{\beta=1}^{4} W^\beta \ast [ \bar u^\beta (x,\tau) - \gamma^\beta(x,\tau)] + \lambda^\beta(\tau)\right)\alpha(\tau),  \\
\displaystyle u^\beta(x,z,0)&=&u^{0,\beta}(x,z),\qquad\qquad \qquad\qquad z\in\,(0,+\infty).
\end{array}\right.
\end{equation}

The exact same method that we used in Section \ref{sec:2} for the one population model allows us to derive Duhamel formulae for $u^\beta, v^\beta, \gamma^\beta$.\\

For the solution:
\begin{equation}\label{eq:Duhamel4pop}
    u^\beta(x,z,\tau) = \int_{0}^{+\infty} G(z,\tau, \xi ,0) u^{0,\beta}(x,\xi) d\xi + \int_{0}^{\tau}  \dfrac{\partial G}{\partial \xi}(z,\tau,\gamma^\beta(x,\eta),\eta)\, v^\beta(x,\eta) d\eta.
\end{equation}

For the boundary value:
\begin{align}\label{never_gonna_give_you_up}
    v^\beta(x,\tau) =&\, 2 \int_{0}^{+\infty} G(\gamma^\beta(x,\tau),\tau, \xi ,0) u^{0,\beta}(x,\xi) d\xi\nonumber\\ &\,+ 2\int_{0}^{\tau}  \dfrac{\partial G}{\partial \xi}(\gamma^\beta(x,\tau),\tau,\gamma^\beta(x,\eta),\eta)\, v^\beta(x,\eta) d\eta.
\end{align}

For the average:
\begin{align}\label{never_gonna_let_you_down}
    \bar u^\beta (x,\tau) =&\, \int_{0}^{+\infty} \int_{\gamma^\beta(x,\tau)}^{+\infty} zG(z,\tau, \xi ,0)dz u^{0,\beta}(x,\xi) d\xi\nonumber\\&\, + \int_{0}^{\tau} \int_{\gamma(x,\tau)}^{+\infty} z \dfrac{\partial G}{\partial \xi}(z,\tau,\gamma^\beta(x,\eta),\eta) dz\, v^\beta(x,\eta) d\eta.
\end{align}

For the motion of the boundary:
\begin{equation}\label{never_gonna_run_around_and_desert_you}
    \gamma^\beta(x,\tau) =  - \int_0^\tau \Phi\left( \alpha(\eta) \sum_{\beta=1}^{4}W^\beta \ast [ \bar u^\beta (x,\eta) - \gamma^\beta(x,\eta)] + \lambda^\beta(\eta)\right)\alpha(\eta)d\eta.
\end{equation}

Hence, if we denote in capital letters the vectors containing entries for all $\beta$, we can write the same kind of closed system as before
\begin{equation}\label{eq:closed4pop}
    \left\{\begin{array}{rcl}
        V(x,\tau)      &=&   F_v[V,\Gamma, \bar U](x,\tau) \\
        \Gamma(x,\tau) &=&   F_\gamma[V,\Gamma,\bar U](x,\tau)\\
        \bar U(x,\tau) &=&  F_{\bar u}[V,\Gamma,\bar U](x,\tau).
    \end{array}\right.
\end{equation}

We can then consider the space
    \[  \mathcal C_{\tau_0,m} = \{ \textbf{w} \in C^0(\mathbb T^d \times [0,\tau_0])^8\ | \  \norme{\textbf{w}}_{\infty} < m\},  \]
where we denote \[\norme{\textbf{w}}_\infty = \max_{i\in\{1,\dots,8\}}\ \norme{w_i}_{L^{\infty}(\mathbb T^d\times [0,\tau_0])}.\]
On this space, we consider the functional
\begin{equation}\label{eq:functional4pops}
    \begin{array}{rccl}
         \mathcal T:& \mathcal C_{\tau_0,m} & \to & C^0(\mathbb T^d\times [0,\tau_0])^8  \\
                   & (V,\bar U) & \mapsto &  ( F_v[V,\Gamma, \bar U], F_{\bar u}[V,\Gamma,\bar U]),
    \end{array}
\end{equation}
where  $\Gamma=(\gamma^1,\gamma^2,\gamma^3,\gamma^4)$ is the unique continuous solution of the coupled initial value problems
\begin{equation}\label{eq:gamma_ODE_4}
     \left\{\begin{array}{rcl}
          \dfrac{\partial \gamma^\beta}{\partial \tau}(x,\tau) & = &\displaystyle -\Phi\left( \alpha(\tau) \sum_{\beta=0}^{4} W^\beta \ast [ \bar u^\beta (x,\tau) - \gamma^\beta(x,\tau)] + \lambda^\beta(\tau)\right)\alpha(\tau),  \quad  \\
          \gamma^\beta(x,0)& = &  \gamma^{0,\beta}(x).
     \end{array}\right.
\end{equation}

Like previously in Lemma \ref{lm:gamma}, we can prove that if $\Phi$ is locally Lipschitz, $W^\beta\in L^1(\mathbb T^d)$ and $m>0$ is fixed, then there exists $\tau_0\in(0,1]\,$ small enough depending only on $\Phi,\lambda^\beta$ and $m$ such that for all $\bar u\in C^0(\mathbb T^d\times [0,\tau_0])$ and $\gamma^0\in C^0(\mathbb T^d)$ satisfying 
    \[  \max_{\beta\in\{1,2,3,4\}}\norme{\bar u^\beta}_{L^\infty(\mathbb T^d\times[0,\tau_0])} \leqslant m, \qquad \mathrm{and} \qquad \max_{\beta\in\{1,2,3,4\}}\norme{\gamma^{0,\beta}}_{L^\infty(\mathbb T^d)}\leqslant \frac m2, \]
    there exists a unique solution $\Gamma\in C^0(\mathbb T^d\times [0,\tau_0])^4$ to \eqref{eq:gamma_ODE_4}. For all $x\in \mathbb T^d$, $\tau\mapsto \Gamma(x,\tau)$ is a $C^1$ function and, denoting $\Psi  =  (\Psi^1,\Psi^2,\Psi^3,\Psi^4)$,
    \begin{equation*}
        \norme{\Psi}_{L^\infty(\mathbb T^d\times [0,\tau_0])^4} \leqslant C_\Psi,\qquad \mathrm{and} \qquad \norme{\Gamma}_{L^\infty(\mathbb T^d\times [0,\tau_0])^4} \leqslant m.
    \end{equation*}

Now we can assume without loss of generality that $\Gamma^0 \equiv (0,0,0,0)$ like we did above for the one population case, and choose
\[ m = \max_{\beta\in\{1,2,3,4\}}\left( \max\left(\ 2\norme{u^{0,\beta}}_{L^\infty(\mathbb T^d \times \R_+)}\ ,\ \norme{z u^{0,\beta}}_{L^\infty(\mathbb T^d, L^1(\R_+))}\ \right) + 1\right).  \]
Then, we can follow, for each component $\beta\in\{1,2,3,4\}$, the proof of Proposition \ref{prop:local}. The interaction appearing only in the terms $\Psi^\beta$ which are uniformly controlled, no additional difficulty can arise and the exact same techniques lead to a local existence result for the closed system \eqref{eq:closed4pop}.

Then, for global existence, as we argued before, it is enough to obtain uniform bounds on $\Gamma$ or $\bar U$ for any compact interval of time. We can follow the proof of Theorem \ref{thm:global}: step 1 can be done with the sole modification of bounding the sum in oder to obtain
\[  \norme{\Gamma}_{L^\infty(\mathbb T^d\times [0,\tau^*))^4} \leqslant C^*;  \]
then, for step 2, we can apply the same test function $h$ in original variables and obtain
\begin{multline}
    \forall x\in \mathbb T^d, \quad  \int_0^{+\infty} h(s) \rho^\beta(x,s,t) ds =\int_0^{+\infty} h(s) \rho^{0,\beta}(x,s) ds \\  +\frac{1}{\tau_c} \int_0^t\int_0^{+\infty}\left[ (\Phi_{\bar \rho}^\beta(x,t') - s)\dfrac{d h}{ds}(s) + \sigma \dfrac{d^2 h}{ds^2}(s)\right] \rho(x,s,t') dsdt',
    \end{multline}
    and like before get to
    \begin{align*} \bar\rho^\beta(x,t) \leqslant \dfrac{1}{L^d} &+ \int_0^{+\infty} h(s) \rho^{0,\beta}(x,s) ds\\ &+ \dfrac{ T^*}{\tau_c L^d}\left( C_1^* \max_{\beta\in\{1,2,3,4\}} \left(\norme{W}_{L^1} \norme{\bar\rho^\beta}_{L^\infty(\mathbb T^d\times[0,t))} + \norme{B}_{L^\infty([0,T^*))}\right)  + C_2^*\right);   
\end{align*}
taking the supremum on $\beta$ in the left-hand side and $T^*$ small enough yields a locally uniform in time bound on $(\bar\rho^1,\bar\rho^2,\bar\rho^3,\bar\rho^4)$.
Hence, we can claim the following result.

\begin{theorem}[Solutions for the four population model]
    Assume $\Phi$ is locally Lipschitz, $B^\beta\in C^0(\R_+)$, $\beta\in\{1,2,3,4\}$, and $W\in L^1(\mathbb T^d)$. Then, for any compatible initial data $\rho^{0,\beta}$, $\beta\in\{1,2,3,4\}$,
    equation \eqref{eq:4PDE} admits a unique maximal solution in the sense of Definition \ref{def:classical} for each component.
    Moreover, if $\Phi$ is globally Lipschitz,
    \begin{itemize}
        \item the solution is global-in-time;
        \item if $\Phi$ non-negative, then
    \begin{equation*}
    \forall \beta\in\{1,2,3,4\},\ \forall t\in\R_+,\ \forall x\in\mathbb T^d, \quad \rho^\beta(x,0,t) \leqslant \dfrac{1}{L^d}\sqrt{ \dfrac{2}{ \pi \sigma } } \dfrac1{\sqrt{1-e^{-\frac{2t}{\tau_c}}}};
    \end{equation*}
    \item if $\Phi$ is non-positive, then for all  $\beta\in\{1,2,3,4\}$,
        \[  \forall t\in \R_+,\ \forall x\in \mathbb T^d, \quad  \bar\rho^\beta(x,t) \leqslant  \dfrac{1}{L^d}(1+\sigma - \sigma e^{-\frac{2t}{\tau_c}}) +  e^{-\frac{2t}{\tau_c}} \sup_{x\in\mathbb T^d}\int_0^{+\infty} s^2\rho^{0,\beta}(x,s)ds . \]
    \end{itemize}
\end{theorem}

\begin{remark}
    An alternative route for the extension of existence to the four population system would be, denoting $\Lambda= (\Lambda^1,\Lambda^2,\Lambda^3,\Lambda^4)$ a vector of four continuous functions of time, to introduce the following functional: consider for each $\beta$ the solution of the one population problem \eqref{eq:Stefan} with exterior input $\Lambda^\beta$. Given the four solutions of these one population problems, we define $\mathcal F(\Lambda)$ as the vector whose entries $\beta\in\{1,2,3,4\}$ are
    \[\mathcal F( \Lambda )_\beta = \sum_{\beta'=1,\ \beta'\neq\beta}^{4} W^\beta \ast [ \bar u^\beta (x,\tau) - \gamma^\beta(x,\tau)] + \lambda^\beta(\tau). \]
    From the regularities proved in the one population existence proof, it is possible to prove that this functional is a contraction on a suitable Banach space. Applications of the Banach fixed point theorem then yields a unique fixed point $(\Lambda^1,\Lambda^2,\Lambda^3,\Lambda^4)$. Using these four functions as exterior input in the one-population existence theorem and assembling these carefully crafted pieces, we recover the unique solution of the four population model.
    
    Yet another way to frame the extension is to introduce, instead of the vector $\bar U = (\bar u^1,\bar u^2,\bar u^3,\bar u^4)$, the scalar variable
    \[  \bar U = \sum_{\beta'=1}^4 W^{\beta'} \ast \bar u^{\beta'},   \]
    and then to sum the Duhamel formulae for $\bar u^\beta$ in order to include this quantity in the fixed point functional \eqref{eq:functional4pops}, thus symplifying the space upon which it is defined. This last extension framework was used in \cite{CRS} to extend bifurcation results from the one population system to the model for grid cells.
\end{remark}

\subsection{Long-time behaviour for constant and identical input functions}
Note that due to the arguments of \cite[Sec. 4]{CRS}, the stationary states for the four different orientation preferences $\beta$ are identical. Checking the proof of Proposition \ref{prop:stab1}, the result can be straightforwardly extended to \eqref{eq:4PDE}: 
\begin{corollary}
Let $B^\beta$ be constant and identical for $\beta = 1,2,3,4$. Let $\Phi$ be Lipschitz, $W \in L^2(\mathbb T^d)$, and let $\rhostat$ be a stationary state of \eqref{eq:4PDE}, equal for $\beta = 1,2,3,4$, given a $\sigma$. Assume that $\sigma$ satisfies 
\begin{align*}
\|\Phi' \|_\infty \|W\|_{L^2({\mathbb T^d})} \left(\sup_{x \in {\mathbb T^d}} M_\infty^\hf (x)\right) < 2 \sigma \tilde\gamma (\rhostat )^\hf,
\end{align*}
where $\tilde\gamma (\rhostat ) = \inf_{x \in \mathbb T^d} \gamma(\rhostat^\beta (x))$, and $\gamma(\rhostat(x))$ is the Poincar\'e constant of \eqref{eq:poincare} for $\rhostat (x)$, and $M_\infty$ is defined in \eqref{eq:minfty}. Then, the sum of the $L^1$ norms of the relative entropies,
\begin{align*}
 \sum_\beta \int_{\mathbb T^d} \int_{0}^{+\infty} \left( \dfrac{ \rho^\beta(x,s,t) - \rho_\infty(x,s) }{ \rho_\infty(x,s) } \right)^2 \rho_\infty(x,s) ds dx,
\end{align*} 
 decays exponentially fast whenever the initial data $\rho_0^\beta$ are close enough to the stationary state $\rhostat$ in relative entropy.
 \end{corollary}
 
Extending the result of Proposition \ref{prop:stab} to the system \eqref{eq:4PDE} requires a bit more work as the stability condition in this case will depend on the size of the shifts $r^\beta$ as well.
\begin{proposition}\label{prop:4stab}
Let $B^\beta$ be constant and identical for $\beta = 1,2,3,4$, $\Phi \in C^2(\mathbb R)\cap \mathcal W^{2,\infty}(\R)$, and $W \in L^2(\mathbb T^d)\cap \mathcal W^{1,\infty}(\mathbb T^d)$. Assume that $\rhostat$ is a spatially homogeneous stationary state of \eqref{eq:4PDE}, equal for $\beta = 1,2,3,4$, given a $\sigma$. Furthermore, assume that $W$ is componentwise symmetric, and that $\sigma$ satisfies 
    \begin{align*}
    \sum_\beta \int_{\mathbb T^d} (1-\alpha)(g^\beta)^2(x) - \frac{\Phi^\delta_g(x)}{\sigma} g^\beta(x) \, dx > 0, \quad g^\beta \in L^2(\mathbb T^d),
    \end{align*}
where $$\Phi^\delta_g(x) =\Phi\left(\frac{M_\infty}{4} \sum_\beta W \ast g^\beta (x-r^\beta) +W_0 \brho_\infty +B\right)-\Phi_0,$$ and $M_\infty$ is defined in \eqref{eq:minfty}. Then, if the shifts $r^\beta$ satisfy 
\begin{align}\label{eq:shiftcondition}
    \max_{\beta,j} |r^\beta+r^j| \leq \frac{\alpha^{\unitfrac{3}{2}}(2-4\xi)\gamma(\rhostat)^\hf \sigma}{\|\nabla W\|_\infty\|\Phi'\|_\infty M_\infty}
\end{align}
for a suitably small $\xi>0$, the sum of the $L^1$ norms of the relative entropies,
\begin{align*}
 \sum_\beta \int_{\mathbb T^d} \int_{0}^{+\infty} \left( \dfrac{ \rho^\beta(x,s,t) - \rho_\infty(s) }{ \rho_\infty(s) } \right)^2 \rho_\infty(s) ds dx,
\end{align*} 
 decays exponentially fast whenever the initial data $\rho_0^\beta$ are close enough to the stationary state $\rhostat$ in relative entropy.
 \end{proposition}

\begin{proof}
    We follow the lines of the proof of Proposition \ref{prop:stab}, but this time denote $$ \Phi^\delta = \Phi\left(\frac{1}{4}\sum_\beta W^\beta \ast \brho^\beta +B\right)-\Phi_0, \quad \bar \rho^\delta_\beta = \bar \rho^\beta - \bar \rho_\infty, \quad \Phi'_{\brho} = \Phi'\left(\frac{1}{4}\sum_\beta W^\beta \ast \brho_\beta + B\right),$$ and
\begin{align*}
\mathcal{E} =  \sum_\beta \int_0^{+\infty}  \frac{1}{2}(h^\beta-1)^2 \rhostat ds, \quad \mathrm{and} \quad \mathcal{H}=    \sum_\beta \Phi^\delta \bar \rho^\delta_\beta .
\end{align*}
We have
\begin{align*}
  \frac{d}{dt} \int_{\mathbb T^d} 2\mathcal{E}- \frac{1}{\sigma}\mathcal{H} \, dx =\,&  2 \int_{\mathbb T^d} \Phi^\delta \sum_\beta \int_0^{+\infty} \partial_s h^\beta (h^\beta-1) \rhostat  ds dx - 2\sigma \sum_\beta\int_{\mathbb T^d} \int |\partial_s h^\beta|^2 \rhostat ds dx \\
  & + 3 \int_{\mathbb T^d} \Phi^\delta \sum_\beta\int_0^{+\infty} \partial_s h^\beta  \rhostat  ds dx - \frac{4}{\sigma} \int_{\mathbb T^d} (\Phi^\delta)^2 dx \\
  &- \frac{1}{\sigma}\int_{\mathbb T^d} \left(\sum_\beta \brho^\delta_\beta \right)\, \Phi'_\brho \frac{1}{4}\sum_\beta W^\beta \ast \left(\Phi^\delta - \sigma \int_0^{+\infty} \partial_s h^\beta  \rhostat  ds \right) dx\\
  =& -\frac{2}{\sigma}\sum_\beta\int_{\mathbb T^d}\int_0^{+\infty} \left(\Phi^\delta -\sigma \partial_s h^\beta \right)^2 \rhostat  ds dx \\
  &+\frac{1}{\sigma} \sum_\beta\int_{\mathbb T^d} \Phi^\delta \left(\Phi^\delta-\sigma \int_0^{+\infty} \partial_s h^\beta  \rhostat  ds \right) dx \\
  &+ 2 \int_{\mathbb T^d} \Phi^\delta \sum_\beta \int_0^{+\infty} \partial_s h^\beta (h^\beta-1) \rhostat  dsdx\\
  &- \frac{1}{\sigma}\int_{\mathbb T^d} \left(\sum_\beta \brho^\delta_\beta \right)\, \Phi'_\brho \frac{1}{4}\sum_\beta W^\beta \ast \left(\Phi^\delta - \sigma \int_0^{+\infty} \partial_s h^\beta  \rhostat  ds \right) dx \,.
\end{align*}
Using the symmetry of $W$,
\begin{align*}
& - \frac{1}{\sigma}\int_{\mathbb T^d} \left(\sum_\beta \brho^\delta_\beta \right)\, \Phi'_\brho \frac{1}{4}\sum_\beta W^\beta \ast \left(\Phi^\delta - \sigma \int_0^{+\infty} \partial_s h^\beta  \rhostat  ds \right) dx\\
=&- \frac{1}{4\sigma}\sum_j \sum_\beta\int_{\mathbb T^d}\int_{\mathbb T^d}  \brho^\delta_\beta(x) \, \Phi'_\brho(x)\,  W(x-y-r^j)  \left(\Phi^\delta(y) - \sigma \int_0^{+\infty} \partial_s h^j(y)  \rhostat  ds \right) dydx\\
=&- \frac{1}{4\sigma}\sum_j \sum_\beta\int_{\mathbb T^d}\int_{\mathbb T^d}  \brho^\delta_\beta(x) \, \Phi'_\brho(x)\,  W(y+r^j-x)  \left(\Phi^\delta(y) - \sigma \int_0^{+\infty} \partial_s h^j(y)  \rhostat  ds \right) dydx\\
=&- \frac{1}{4\sigma}\sum_j \sum_\beta\int_{\mathbb T^d} W\ast\left( \brho^\delta_\beta \, \Phi'_\brho\right)(y+r^j)\,   \left(\Phi^\delta(y) - \sigma \int_0^{+\infty} \partial_s h^j(y)  \rhostat  ds \right) dy\\
=& -\frac{1}{\sigma} \sum_j \int_{\mathbb T^d} \left(\frac{1}{4}\sum_\beta W\ast \big(\Phi_\brho'\,\brho^\delta_\beta\big)(x)\right) \int_0^{+\infty} \left(\Phi^\delta-\sigma \partial_s h^j  \right)(x-r^j) \rhostat  ds dx
\end{align*}
where the last equality is just a change of variable $x=y+r^j$. Now, taking into account 
\begin{align*}
\int_{\mathbb T^d} \Phi^\delta \left(\Phi^\delta-\sigma \int_0^{+\infty} \partial_s h^j  \rhostat  ds \right) dx 
= \int_{\mathbb T^d} \Phi^\delta(x-r^j) \int_0^{+\infty} \left(\Phi^\delta-\sigma \partial_s h^j  \right)(x-r^j) \rhostat  ds dx
\end{align*}
and
$$
\int_{\mathbb T^d} \int_0^{+\infty} (\Phi^\delta)^2 \sum_\beta (h^\beta-1) \rhostat ds dx = 0\,,
$$
we can rewrite the right hand side as 
\begin{align*}
  \frac{d}{dt} \int_{\mathbb T^d} 2\mathcal{E}- \frac{1}{\sigma}\mathcal{H} \, dx =\,&
 -\frac{2}{\sigma}\sum_\beta\int_{\mathbb T^d}\int_0^{+\infty} \left(\Phi^\delta -\sigma \partial_s h^\beta \right)^2 \rhostat  ds dx \\
& +\frac{1}{\sigma} \sum_j \int_{\mathbb T^d} \left(\Phi^\delta(x-r^j) - \frac{1}{4}\sum_\beta W\ast \big(\Phi_\brho'\,\brho^\delta_\beta\big)(x)\right)\\
& \qquad\qquad\qquad  \times \int_0^{+\infty} \left(\Phi^\delta-\sigma \partial_s h^j  \right)(x-r^j) \rhostat  ds dx \\
& - \frac{2}{\sigma} \int_{\mathbb T^d} \int_0^{+\infty} \Phi^\delta \sum_\beta (h^\beta-1) \left( \Phi^\delta -\sigma \partial_s h^\beta\right) \rhostat ds dx \\
& = I + II + III, 
\end{align*}
where 
\begin{align*}
    II &\leq \frac{1}{\sigma} \sum_j \int_{\mathbb T^d}\bigg|\Phi^\delta(x-r^j) - \frac{1}{4}\sum_\beta W\ast \big(\Phi_\brho'\brho^\delta_\beta\big)(x)\bigg|\| \Phi^\delta-\sigma  \partial_s h^j  \|_{L^2_{\rhostat(x)}}(x-r^j) dx, \\
    III &\leq \frac{2}{\sigma} \int_{\mathbb T^d} |\Phi^\delta| \sum_\beta \|h^\beta-1\|_{L^2_{\rhostat(x)}} \| \Phi^\delta -\sigma \partial_s h^\beta \|_{L^2_{\rhostat(x)}}dx.
\end{align*}
Now let
\begin{align*}
   \phi'(x) = \frac{\Phi^\delta (x)}{\frac{1}{4}\sum_\beta W^\beta \ast \brho^\delta_\beta (x)} = \Phi'\left( \frac{1}{4}\sum_\beta \left[\eta(x) (W^\beta\ast\brho^\beta)+(1-\eta(x))(W^\beta\ast\brho_\infty)\right]+B\right), 
\end{align*}
for some $0\leq \eta(x)\leq 1$ in virtue of the mean-value theorem. Then, through the following estimate, the shifts come into play:
\begin{align*}
\bigg|&\Phi^\delta(x-r^j) - \frac{1}{4}\sum_\beta W\ast \big(\Phi_\brho'\,\brho^\delta_\beta\big)(x)\bigg|\\
& \qquad = \frac{1}{4}\left|\sum_\beta \int_{\mathbb{T}^d} \brho^\delta_\beta (y) \left( \phi'(x-r^j)W(x-y-r^\beta-r^j) - \Phi_\brho' (y)W(x-y) \right) dy \right|  \\
& \qquad \leqslant \frac{1}{4}\sum_\beta \int_{\mathbb{T}^d} |\brho^\delta_\beta (y)|\left|\phi'(x-r^j)-\Phi_\brho' (y)\right| |W(x-y)| dy \\
& \qquad \quad + \frac{1}{4}\sum_\beta \int_{\mathbb{T}^d} |\brho^\delta_\beta (y)||\phi'(x-r^j)| \left|W(x-y-r^\beta-r^j) -W(x-y)\right|  dy \\
& \qquad = A + B.
\end{align*}
By using the definition of $\phi'(x)$ and applying the mean value theorem to $\Phi'$, we get
\begin{align*}
\left|\phi'(x-r^j)-\Phi_\brho' (y)\right| \leq \,&\frac{\|\Phi''\|_\infty }{4} \sum_i \left| \eta(x-r_j)(W^i\ast \brho^i)(x-r_j) \right.\\
                                            &\qquad \qquad\,\,\,\,\,\, \left. + (1-\eta(x-r_j))(W^i \ast \brho_\infty)-W^i\ast \brho^i(y) \right| \\
                                           = \,& \frac{\|\Phi''\|_\infty }{4} \sum_i \eta(x-r_j) \left| (W^i\ast \brho_i^\delta)(x-r_j) -W^i \ast  \brho^\delta_i(y) \right|
                                           \\
                                           \leq \,& \frac{\|\Phi''\|_\infty }{2} \|W\|_{L^2(\mathbb T^d)} \sum_{i} \|\brho_i^\delta\|_{L^2(\mathbb T^d)},
\end{align*}
where we used the Cauchy--Schwarz inequality in the convolutions. A further application of the Cauchy--Schwarz inequality yields
$$
A \leq  \frac{\|\Phi''\|_\infty }{8} \|W\|_{L^2(\mathbb T^d)}^2\sum_{\beta,i} \|\brho_\beta^\delta\|_{L^2(\mathbb T^d)}\|\brho_i^\delta\|_{L^2(\mathbb T^d)}.
$$
We proceed to estimate $B$ as
\begin{align*}
B &\leq  \frac{1}{4}\|\nabla W\|_\infty\|\Phi'\|_\infty\sum_\beta |r^\beta+r^j|\int_{\mathbb T^d} |\brho_\beta^\delta|dx \leq  \frac{1}{4}\|\nabla W\|_\infty\|\Phi'\|_\infty\sum_\beta |r^\beta+r^j|\,\|\brho_\beta^\delta\|_{L^2(\mathbb T^d)}. 
\end{align*}
Collecting the estimates above we finally conclude that
\begin{align*}
\bigg|\Phi^\delta(x-r^j) - \frac{1}{4}\sum_\beta W\ast \big(\Phi_\brho'\,\brho^\delta_\beta\big)(x)\bigg| \leq C_1\sum_{\beta} \|\brho_\beta^\delta\|_{L^2(\mathbb T^d)}^2 + C_2\sum_\beta |r^\beta+r^j|\, \|\brho_\beta^\delta\|_{L^2(\mathbb T^d)},
\end{align*}
with $C_1 = \frac{1}{2}\|\Phi''\|_\infty \|W\|_{L^2(\mathbb T^d)}^2$ and $C_2 = \frac{1}{4}\|\nabla W\|_\infty\|\Phi'\|_\infty$.
Thus,
\begin{align*}
 \frac{d}{dt} \int_{\mathbb T^d} 2\mathcal{E}- \frac{1}{\sigma}\mathcal{H} \, dx & \leq -\frac{2}{\sigma}\sum_\beta\|\Phi^\delta -\sigma h_s^\beta \|_{L^2_{\rhostat}}^2 \\
& + \frac{1}{\sigma} \sum_j \int_{\mathbb T^d}\bigg|\Phi^\delta(x-r^j) - \frac{1}{4}\sum_\beta W\ast \big(\Phi_\brho'\brho^\delta_\beta\big)(x)\bigg|\| \Phi^\delta-\sigma  \partial_s h^j  \|_{L^2_{\rhostat(x)}}(x-r^j) dx\\
& + \frac{2}{\sigma} \int_{\mathbb T^d} |\Phi^\delta| \sum_\beta \|h^\beta-1\|_{L^2_{\rhostat(x)}} \| \Phi^\delta -\sigma \partial_s h^\beta \|_{L^2_{\rhostat(x)}}dx \\
\leq & \, - \frac{2}{\sigma}\sum_j\|\Phi^\delta -\sigma h_s^j \|_{L^2_{\rhostat}}^2 + \frac{1}{\sigma}\sum_j\|\Phi^\delta -\sigma h_s^j \|_{L^2_{\rhostat}} \bigg( C_1\sum_{\beta} \|\brho_\beta^\delta\|_{L^2(\mathbb T^d)}^2 \\ 
& \qquad \qquad + C_2\sum_\beta |r^\beta+r^j|\|\brho_\beta^\delta\|_{L^2(\mathbb T^d)}  
+2\left(\int_{\mathbb T^d}|\Phi^\delta|^2 \|h^j-1\|_{L^2_{\rhostat(x)}}^2 dx \right)^\hf\bigg) \\
\leq & \,  - \frac{1}{\sigma}\sum_j\|\Phi^\delta -\sigma h_s^j \|_{L^2_{\rhostat}} \sum_\beta \bigg(\frac{1}{2}\|\Phi^\delta -\sigma h_s^\beta \|_{L^2_{\rhostat}} -C_1 \|\brho_\beta^\delta\|_{L^2(\mathbb T^d)}^2 \\ 
& \qquad \qquad - C_2 |r^\beta+r^j|\|\brho_\beta^\delta\|_{L^2(\mathbb T^d)}  
-2\left(\int_{\mathbb T^d}|\Phi^\delta|^2 \|h^j-1\|_{L^2_{\rhostat(x)}}^2 dx \right)^\hf\bigg).
\end{align*}
Note that the third term in the parenthesis might decay at the same order as the first term. Thus, to ensure that the right hand side of the inequality can be made strictly smaller than zero for all $t\geq 0$ by enforcing a smallness assumption on the pertubation away from the stationary state $\rhostat$ as was done in the proof of Proposition \ref{prop:stab}, we need to ensure that
\begin{align*}
\sum_\beta \left(\frac{1}{2}\|\Phi^\delta -\sigma h_s^\beta \|_{L^2_{\rhostat}} - C_2 |r^\beta+r^j|\|\brho_\beta^\delta\|_{L^2(\mathbb T^d)}\right) \geq \xi \sum_\beta \|\Phi^\delta -\sigma h_s^\beta \|_{L^2_{\rhostat}},
\end{align*}
for some constant $\xi>0$,
which holds under assumption \eqref{eq:shiftcondition}. This can be checked by applying the Poincar\'e inequality \eqref{eq:poincare} to the first term and then estimates corresponding to \eqref{eq:APhisquared}.
The remaining steps now follow by taking extra care of the different orientation preferences, doing a change of variables as in the proof of Proposition \ref{prop:stab} and following the remaining arguments therein.
\end{proof}

\subsection*{Acknowledgements}
JAC and PR were supported by the Advanced Grant Nonlocal-CPD (Nonlocal PDEs for Complex Particle Dynamics: Phase Transitions, Patterns and Synchronization) of the European Research Council Executive Agency (ERC) under the European Union's Horizon 2020 research and innovation programme (grant agreement No. 883363).
JAC was also supported by the grants EP/T022132/1 and EP/V051121/1 of the Engineering and Physical Sciences Research Council (EPSRC, UK).
Parts of this research was conducted while SS was an academic visitor at the Mathematical Institute, University of Oxford, and the author would like to thank the institution for its warm hospitality. The authors would like to thank the Isaac Newton Institute for Mathematical Sciences, University of Cambridge, for support and hospitality during the programme ``Frontiers in kinetic theory'' where parts of the work on this paper was undertaken, supported by EPSRC grant no EP/R014604/1.

\bibliographystyle{abbrv}
\bibliography{Biblio.bib}

\end{document}